\documentclass[12pt]{article}
\usepackage {epsfig,amsmath,euscript}
\usepackage[latin1]{inputenc}
\usepackage[english]{babel}
\usepackage{color}
\usepackage[scr=boondoxo]{mathalfa}
\usepackage[amsmath,thmmarks,standard,thref]{ntheorem}
\usepackage{ae}
\usepackage{subfigure}

\usepackage[T1]{fontenc}
\usepackage{graphicx}
\usepackage{epstopdf}
\usepackage[font=small]{caption}
\usepackage{todonotes}
\usepackage{soul}
\usepackage{soul}
\usepackage{color}
\sethlcolor{yellow}
\definecolor{rred}{rgb}{0.7,0.0,0.2}
\definecolor{bblue}{rgb}{0.2,0.0,0.7}
\newcommand\edNew[1]{{\color{black} #1}}%
\newcommand\edNewNew[1]{{\color{black} #1}}%
\newcommand\ed[1]{{\color{black} #1}}

\newcommand{\secref}[1]{Section~\ref{sec:#1}}

\newcommand{\seclab}[1]{\label{sec:#1}}
\newcommand{\eqlab}[1]{\label{eq:#1}}
\renewcommand{\eqref}[1]{(\ref{eq:#1})}
\newcommand{\eqsref}[2]{(\ref{eq:#1}) and~(\ref{eq:#2})}

\newcommand{\figref}[1]{Fig.~\ref{fig:#1}}
\newcommand{\figlab}[1]{\label{fig:#1}}
\newcommand{\propref}[1]{Proposition~\ref{proposition:#1}}
\newcommand{\proplab}[1]{\label{proposition:#1}}
\newcommand{\corref}[1]{Corollary~\ref{corollary:#1}}
\newcommand{\corlab}[1]{\label{corollary:#1}}
\newcommand{\lemmaref}[1]{Lemma~\ref{lemma:#1}}
\newcommand{\lemmalab}[1]{\label{lemma:#1}}
\newcommand{\remref}[1]{Remark~\ref{remark:#1}}
\newcommand{\remlab}[1]{\label{remark:#1}}
\newcommand{\thmref}[1]{Theorem~\ref{theorem:#1}}
\newcommand{\thmlab}[1]{\label{theorem:#1}}

\newcommand{\defnlab}[1]{\label{defn:#1}}
\newcommand{\defnref}[1]{Definition~\ref{defn:#1}}
\newcommand{\appref}[1]{Appendix~\ref{app:#1}}

\newcommand{\applab}[1]{\label{app:#1}}

\newcommand{\pws}{PWS} 
\newcommand{\pwl}{PWL} 
\newcommand{\rpwl}{RPWL} 
\newcommand{\rpwlf}{RPWL$^*$} 

\newcommand{\tpws}{\eqlab{\pws}}
\newcommand{\tpwl}{\eqlab{\pwl}}
\newcommand{\trpwl}{\eqlab{\rpwl}}
\newcommand{\trpwlf}{\eqlab{\rpwlf}}
\newcommand{\bpws}{\eqref{\pws}}
\newcommand{\bpwl}{\eqref{\pwl}}
\newcommand{\brpwl}{\eqref{\rpwl}}
\newcommand{\brpwlf}{\eqref{\rpwlf}}

\usepackage{tikz}

\title{{Resolution of the piecewise smooth visible-invisible two-fold singularity in $\mathbb{R}^3$ using regularization and blowup}}

\author{K. Uldall Kristiansen and S. J. Hogan\thanks{K. Uldall Kristiansen: Department of Applied Mathematics and Computer Science, Technical University of Denmark, 2800 Kgs. Lyngby, DK. S. J. Hogan: Department of Engineering Mathematics, University of Bristol, Bristol BS8 1UB, United Kingdom.  }}
\begin{document}
\maketitle

\begin{abstract}
Two-fold singularities in a piecewise smooth (PWS) dynamical system in $\mathbb{R}^3$ have long been the subject of intensive investigation. The interest stems from the fact that trajectories which enter the two-fold are associated with forward non-uniqueness. The key questions are: How do we continue orbits forward in time? Are there orbits that are distinguished among all the candidates?

We address these questions by regularizing the PWS dynamical system for the case of the visible-invisible two-fold. Within this framework, we consider a regularization function outside the class of Sotomayor and Teixera. We then undertake a rigorous investigation, using geometric singular perturbation theory and blowup. We show that there is indeed a forward orbit $U$ that is distinguished amongst all the possible forward orbits leaving the two-fold. Working with a normal form of the visible-invisible two-fold, we show that attracting limit cycles can be obtained (due to the contraction towards $U$), upon composition with a global return mechanism. We provide some illustrative examples. 

\end{abstract}

\pagestyle{myheadings}
\thispagestyle{plain}

\section{Introduction}\seclab{intro}
A piecewise smooth (PWS) dynamical system \cite{filippov1988differential, MakarenkovLamb12} consists of a finite set of ordinary differential equations
\begin{equation}
\dot{\mathbf{x}}=X^i(\mathbf{x}), \quad \mathbf{x}\in \Sigma_i\subset \mathbb{R}^n
\eqlab{pwsdef}
\end{equation}
where the smooth vector fields $X^i$, defined on disjoint open regions $\Sigma_i$, are smoothly extendable to the closure of $\Sigma_i$. The regions $\Sigma_i$ are separated by an $(n-1)$-dimensional set $\Sigma$ called the \textit{switching manifold}, which consists of finitely many smooth manifolds intersecting transversely.  The union of $\Sigma$ and all $\Sigma_i$ covers the whole state space $D \subseteq \mathbb{R}^n$. {In this paper, we consider $n=3$.}

Such systems are found in a wealth of applications, including problems in mechanics (impact, friction, backlash, free-play, gears, rocking blocks), electronics (switches and diodes, DC/DC converters, $\Sigma-\Delta$ modulators), control engineering (sliding mode control, digital control, optimal control), oceanography (global circulation models), economics (duopolies) and biology (genetic regulatory networks): see \cite{Bernardo08, MakarenkovLamb12} for a full set of references.

Although PWS systems are abundant in applications, they pose mathematical difficulties because they do not in general define a (classical) dynamical system. In particular, \ed{uniqueness} of solutions cannot always be guaranteed. A prominent example of a PWS system with non-uniqueness is the two-fold in $\mathbb R^3$ \cite{desroches_canards_2011}. The visible-invisible two-fold is the subject of the present paper.

Previous work has attempted to resolve the forward non-uniqueness associated with the two-fold. In \cite{simpson2014a}, the visible two-fold  in $\mathbb R^2$ was considered. The ambiguity of forward evolution was removed using separate small perturbations: hysteresis, time-delay and noise. In each case, a probabilistic notion of forward evolution close to the two-fold was developed. In the limit as the perturbation tended to zero, almost all orbits or sample paths followed one of the visible tangencies. Thus the possibility of evolution through the two-fold singularity into the escaping region for a nonzero length of time could be excluded, similar to other results for non-differentiable systems in the zero-noise limit. 

In \cite{colombojeffrey2011}, all three two-folds in $\mathbb R^3$ were considered. For the invisible two-fold, the authors asserted that forward evolution through the singularity into the region of unstable sliding {\it is} possible. Then after a finite time, evolution away from the unstable sliding region leads to a return mechanism to the stable sliding region and a subsequent forward evolution through the singularity, leading to what they called ``nondeterministic chaos''. No such conclusion were made for either the visible or the visible-invisible two-fold, in line with \cite{simpson2014a}. 

Frequently, PWS systems are idealisations of smooth systems with abrupt transitions. It is therefore perhaps natural to view a PWS system as a singular limit of a smooth regularized system. In this paper, we adopt this viewpoint and describe the dynamics of a regularization of the PWS visible-invisible two-fold in $\mathbb R^3$.

\subsection{PWS systems in $\mathbb R^3$}
\seclab{prelim}

Let $\textbf x=(x,y,z)\in \mathbb R^3$ and consider an open set $\mathcal U$ and a smooth function $f=f(\textbf x)$ having $0$ as a regular value. Then $\Sigma\subset \mathcal U$ defined by $\Sigma=f^{-1}(0)$ is a smooth $2D$ manifold. The manifold $\Sigma$ is our \textit{switching manifold}. It separates the set $\Sigma_+ =\{\textbf x\in \mathcal U\vert f(\textbf x)>0\}$ from the set $\Sigma_-=\{\textbf x\in \mathcal U\vert f(\textbf x)<0\}$. We introduce local coordinates so that $f(\textbf x)=y$, $\Sigma=\{\textbf x\in \mathcal U\vert y=0\}$ and consider a PWS system on $\mathcal U\subset \mathbb R^3$ in the following form
\begin{eqnarray}
  X(\textbf x) 
  =\left\{ \begin{array}{cc}
                                   X^-(\textbf x)& \text{for}\quad  \textbf x\in \Sigma_-\\
                                   X^+(\textbf x)& \text{for}\quad \textbf x\in \Sigma_+
                                  \end{array}\right.. \eqlab{PWSsystem}
\end{eqnarray}
Here $X^+=(X_1^+,X_2^+,X_3^+)^T$ and $X^-=(X_1^-,X_2^-,X_3^-)^T$ are smooth vector-fields. Then 
%
%
%
%
%
%
%
$\Sigma$ is divided into two types of region: crossing and sliding:
\begin{itemize}
 \item $\Sigma_{cr}\subset \Sigma$ is the \textit{crossing region}, where $$(X^+ f(x,0,z))(X^-f(x,0,z)) =X_2^+(x,0,z) X_2^-(x,0,z) >0.$$
 \item $\Sigma_{sl}\subset \Sigma$ is the \textit{sliding region}, where $$(X^+ f(x,0,z))(X^-f(x,0,z)) =X_2^+(x,0,z) X_2^-(x,0,z)<0.$$
\end{itemize}
Here $X^{\pm} f=\nabla f\cdot X^{\pm}$ denotes the Lie-derivative of $f$ along $X^{\pm}$. Since $f(\textbf x)=y$ in our coordinates we have simply that $X^{\pm} f =X_2^{\pm}$. On $\Sigma_{sl}$ we follow the Filippov convention \cite{filippov1988differential} and define the sliding vector-field $X_{sl}(\textbf x)$ as the convex combination of $X^+$ and $X^-$
\begin{eqnarray}
 X_{sl}(\textbf x) = \sigma X^+(\textbf x)+ (1-\sigma) X^-(\textbf x),\eqlab{XSliding}
\end{eqnarray}
where $\sigma = \sigma (\textbf x) \in (0,1)$ is defined so that $X_{sl}(\textbf x)$ is tangent to $\Sigma_{sl}$. Hence
\begin{eqnarray*}
 \sigma(\textbf x) = \frac{X^-f(x,0,z)}{X^-f(x,0,z)-X^+f(x,0,z)}.\eqlab{lambdaSliding}
\end{eqnarray*}

An orbit of a PWS system can be made up of a concatenation (respecting the direction of time) of orbit segments of $X_{sl}, X^\pm$ on $\Sigma$ and $\Sigma_{\pm}$, respectively. The boundaries of $\Sigma_{sl}$ and $\Sigma_{cr}$ where $X^+ f = X_2^+=0$ or $X^-f = X_2^-=0$ are {\it singularities}, sometimes called \textit{tangencies}. We define two different generic tangencies: the {\it fold} singularity and the {\it two-fold} singularity. 
\begin{definition}
A point $q\in \Sigma$ is a {\it fold} singularity if 
\begin{eqnarray}
X^{+} f(q)=0,\quad X^+(X^+f)(q)\ne 0,\quad \text{and}\quad X^-f(q)\ne 0,\eqlab{foldPlus}
\end{eqnarray}
or 
\begin{eqnarray}
X^{-} f(q)=0,\quad X^-(X^-f)(q)\ne 0,\quad \text{and}\quad X^+f(q)\ne 0.\eqlab{foldNegative}
\end{eqnarray}
A point $q\in \Sigma$ is a {\it two-fold} singularity if both $X^{+}f(q)=0$ and $X^-f(q)=0$, as well as $X^+(X^+f)(q)\ne 0$ and $X^-(X^-f)(q)\ne 0$ and if 
the vectors $X^+(q)$ and $X^-(q)$ are not parallel. 
\end{definition}
Then we have
\begin{proposition}
\cite[Proposition 2.2]{krihog}
A two-fold singularity $q$ is the transversal intersection of two smooth curves $l^+$ and $l^-$ (abbreviated fold lines in the following) of fold singularities satisfying \eqref{foldPlus} and \eqref{foldNegative}, respectively. 
\end{proposition}

For the fold singularity, we distinguish between the \textit{visible} and \textit{invisible} cases.
\begin{definition}\cite[Definition 2.1]{jeffrey_geometry_2011}
A fold singularity $q$ with $X^{+} f(q)=0$ or $X^{-} f(q)=0$ is {\it visible} if
 \begin{eqnarray*}
\mbox{{$X^{+ }(X^{+ } f)(q) >0\quad \text{or} \quad X^{-}(X^{- } f)(q) <0,\quad \text{respectively}$}},
 \end{eqnarray*}
and {\it invisible} if
\begin{eqnarray*}
\mbox{{$X^{+ }(X^{+ } f)(q) <0\quad \text{or}\quad X^{-}(X^{- } f)(q) >0,\quad \text{respectively}.$}}
\end{eqnarray*}
\end{definition}
 Similarly:
\begin{definition} \cite[Definition 2.3]{jeffrey_geometry_2011}
 The two-fold singularity $q$ is 
 \begin{itemize}
  \item {\it visible} if the fold lines $l^+$ and $l^-$ are both visible;
  \item {\it visible-invisible} if $l^+$ ($l^-$) is visible and $l^-$ ($l^+)$ is invisible;
  \item {\it invisible} if $l^+$ and $l^-$ are both invisible. 
 \end{itemize}
\end{definition}

\subsection{The visible-invisible two-fold}
From the general expressions in \cite[Proposition 3.1]{krihog}, the visible-invisible two-fold can be locally
described by the following set\footnote{In this paper, we set $\beta\mapsto - \beta$, in contrast to \cite{krihog}.} of \textit{normalized equations}:
\begin{align}
 \dot x &=\left\{\begin{array}{cc} 
                   \beta^{-1} c+\mathcal O(\vert x\vert +\vert y\vert+\vert z\vert) &\text{for}\quad y>0\\
                  -1+\mathcal O(\vert x\vert +\vert y\vert+\vert z\vert) & \text{for}\quad y<0
                  \end{array}\right.,\tpws\\
 \dot y&=\left\{\begin{array}{cc}  b z+\mathcal O(\vert y\vert +(\vert y\vert +\vert z\vert)(\vert x\vert+\vert y\vert +\vert z\vert))&\text{for}\quad y>0 \\
                  -\beta x+\mathcal O(\vert y\vert +(\vert x\vert+\vert y\vert)(\vert x\vert+\vert y\vert+\vert z\vert)) &\text{for}\quad y<0
                 \end{array}\right.,\nonumber\\
 \dot z&=\left\{\begin{array}{cc} 1+\mathcal O(\vert x\vert +\vert y\vert+\vert z\vert) &\text{for}\quad y>0\\
                  b^{-1}\gamma+\mathcal O(\vert x\vert +\vert y\vert+\vert z\vert)&\text{for}\quad y<0                  
                 \end{array}\right.,\nonumber
\end{align}
where
 $(x,y,z)\in \mathcal U$ is a small neighborhood of the two-fold $q=(0,0,0)$.
 In \bpws{}, the fold lines
\begin{align}
l^+:&\quad y=z=0,\eqlab{foldLines1}\\
l^-:&\quad x=y=0,\nonumber
\end{align}
coincide with the $x$- and $z$-axes respectively, as shown in \figref{VISliding}.  The real numbers $b,c,\gamma$ and $\beta$ are parameters describing the two-fold. 
From \bpws{}, we have 
\begin{align*}
 X^+(X^+f)(x,0,0)& = b+\mathcal O(\vert x\vert ),\\
 X^-(X^-f)(0,0,z) &= \beta +\mathcal O(\vert z\vert),
\end{align*}
within $\mathcal U$. $l^{\pm}$ are therefore lines of visible and invisible folds, respectively, whenever
%
\begin{align}
 \beta>0,\,b>0,\eqlab{signBetaB}
\end{align}
 which we assume henceforth.  
The constants $c$ and $\gamma$ relate to the sliding flow, see our assumption (A), discussed below. Full details of the derivation of \bpws{} can be found in \cite{krihog}.
In \bpws{} we also have that
\begin{eqnarray}
 \Sigma_{sl}:&\,y=0,xz>0,\eqlab{SigmaSliding}\\
 \Sigma_{cr}:&\,y=0,xz<0 ,\nonumber
\end{eqnarray}
where
$\Sigma_{sl} = \Sigma_{sl}^-\cup \Sigma_{sl}^+$ with
$$\Sigma_{sl}^-:\,y=0,x<0,z<0,\quad \mbox{\text{(3rd quadrant of the $(x,z)$-plane)}}$$ and $$\Sigma_{sl}^+:\,y=0,x>0,z>0,\quad \mbox{\text{(1st quadrant of the $(x,z)$-plane)}}$$ are stable and unstable sliding  regions, respectively. Similarly, $\Sigma_{cr} = \Sigma_{cr}^-\cup \Sigma_{cr}^+$ where $$\Sigma_{cr}^-:\,y=0,x>0,z<0,\quad \mbox{\text{(2nd quadrant of the $(x,z)$-plane)}}$$ and $$\Sigma_{cr}^+:\,y=0,x<0,z>0,\quad \mbox{\text{(4th quadrant of the $(x,z)$-plane)}}$$ are regions with crossing downwards and upwards, respectively (see \figref{VISliding}). For later convenience, we define $\breve \Sigma_{sl}$ as the projection of $\Sigma_{sl}$ onto the $xz$-plane. In other words,
\begin{align*}
 (x,z)\in \breve \Sigma_{sl} \Leftrightarrow (x,0,z)\in \Sigma_{sl}.
\end{align*}
We define $\breve \Sigma$ and $\breve \Sigma_{sl}^\pm$ similarly.

We compute the sliding vector-field $X_{sl}$ in \eqref{XSliding} within $\Sigma_{sl}$, using \bpws{}, to give
\begin{eqnarray}
 \dot x&=&\sigma X_1^+(x,0,z)+(1-\sigma )X_1^-(x,0,z),\eqlab{slidingEqns}\\
 \dot y&=&0,\nonumber\\
 \dot z&=&\sigma X_3^+(x,0,z)+(1-\sigma )X_3^-(x,0,z),\nonumber
\end{eqnarray}
where
\begin{eqnarray}
 \sigma(x,z) &=& \frac{X^-f(x,0,z)}{X^-f(x,0,z)-X^+f(x,0,z)}=\frac{-\beta x + \mathcal O(\vert x\vert (\vert x\vert+\vert z\vert))}{-\beta x-b z+\mathcal O((\vert x\vert+\vert z\vert)^2)}.\eqlab{lambda}
\end{eqnarray}
The denominator of $\sigma$ vanishes only at the two-fold within $\overline{\Sigma}_{sl}$. System \eqref{slidingEqns} is therefore singular at $q$. But we can re-parameterize time by multiplying the sliding vector field by $\vert X^-f(x,0,z)-X^+f(x,0,z)\vert$ to obtain the following \textit{desingularized} sliding equations within $\breve \Sigma_{sl}^-$:
  \begin{eqnarray}
  \dot x&=&- c x+ b z+\mathcal O((\vert x\vert+\vert z\vert)^2)\eqlab{slidingVectorField}\\
  \dot z&=&- \beta x-\gamma z+\mathcal O((\vert x\vert+\vert z\vert)^2).\nonumber
 \end{eqnarray}
 See also \cite{krihog,filippov1988differential}. With respect to the new time in \eqref{slidingVectorField}, $(x,z)=0$ is now an equilibrium.
  Furthermore, the orbits of \eqref{slidingVectorField} coincide with the orbits of \eqref{slidingEqns} within $\Sigma_{sl}^+$; one only has to reverse the direction of time for them to agree as trajectories.  
  \ed{This process of \textit{desingularization}, through time-parametrisation, will be used in different versions throughout the manuscript. }

\begin{figure}[h!] 
\begin{center}
{\includegraphics[width=.95\textwidth]{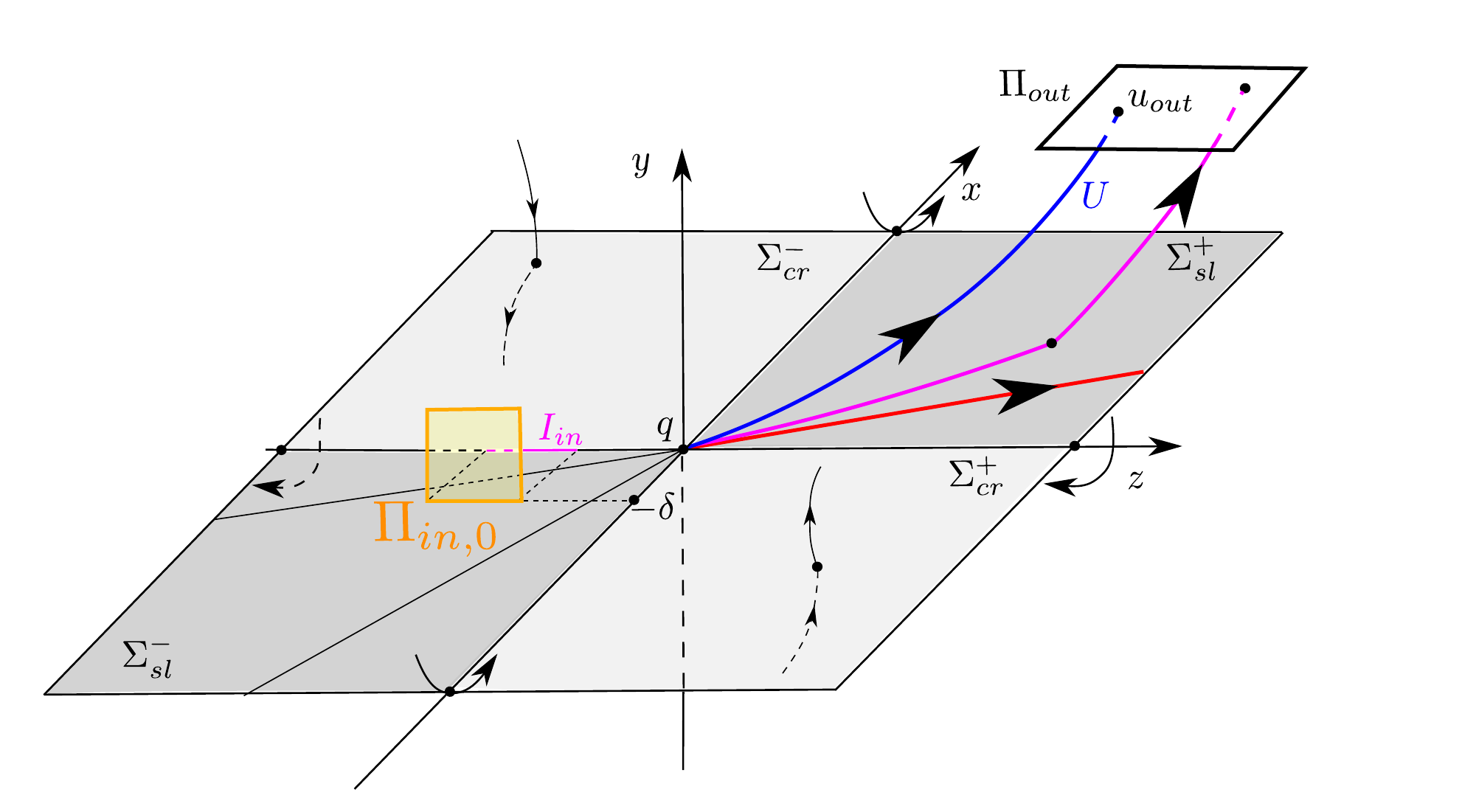}}
\end{center}
 \caption{ \ed{Local geometry of the PWS visible-invisible two-fold at $q=(0,0,0)$. The system is forward non-unique at $q$. We illustrate three possible forward orbits from $q$, shown in purple, red and blue. The purple orbit stays in $\Sigma^+_{sl}$ for a finite time, before escaping into $\Sigma_+=\{y>0\}$. The red orbit remains in $\Sigma^+_{sl}$. The blue orbit, denoted by $U$, does not enter $\Sigma^+_{sl}$, but follows $X^+$ instead. The visible tangencies of $X^+$ with $l^+$ lie in $\Sigma_+$. The invisible tangencies of $X^-$ with $l^-$ lie in $\Sigma_-$. The sections $\Pi_{\text{in},0}$ and $\Pi_{\text{out}}$ are defined in \eqsref{Pin}{Pout}, respectively. $\Pi_{\text{in},0}$ is within $y\ge 0$. The point $u_\text{out} = U\cap \Pi_{\text{out}}$.}}
\figlab{VISliding}
\end{figure}

The eigenvalues $\lambda_{\pm}$ of the linearization of \eqref{slidingVectorField} about $(x,z)=0$ are:
 \begin{align}
  \lambda_{\pm} = -\frac12 (c+\gamma)\pm \frac12 \sqrt{(c-\gamma)^2- 4b\beta}.\eqlab{lambdapm}
 \end{align}
with associated eigenvectors
\begin{align}
\breve v_{\pm} = \begin{pmatrix}1 \\
 -\chi_{\pm}
\end{pmatrix}, \eqlab{vpm}
\end{align}
respectively,
 where
\begin{align}
\chi_{\pm} =- \frac{1}{2b}\left({c-\gamma}\pm \sqrt{(c-\gamma)^2-4b\beta}\right),\eqlab{chipm}
\end{align}     
In the following, let $v_\pm\in \Sigma$ be $\breve v_\pm \in \breve \Sigma$ with an added $0$ $y$-component:
\begin{align}
 v_\pm = \begin{pmatrix}
          1\\
          0\\
          -\chi_\pm
         \end{pmatrix}.\eqlab{vpmNew}
\end{align}

In this paper, we will make the following assumptions, denoted by (A) and (B).
%
\begin{itemize}
 \item[(A)] The two-fold $q$ is a stable node of \eqref{slidingVectorField} with
 \begin{align}
 \lambda_-<\lambda_+&<0.\eqlab{test}
\end{align}
Moreover, we suppose that 
\begin{align}
\chi_+<\chi_{-}&<0,\eqlab{chiInEq}
\end{align}
so that the \textit{weak eigenspace} $\text{span}\,v_+$ and the \text{strong eigenspace} $\text{span}\,v_{-}$, see \eqref{vpmNew}, are both contained within ${\Sigma}_{sl}\cup\{q\}$. In technical terms, these assumptions imply the following inequalities:
 \begin{align}
c\pm \gamma > \sqrt{(c-\gamma)^2-4b\beta},\, (c-\gamma)^2-4b\beta>0.\eqlab{cgamma}
\end{align}
\item[(B)] The non-degeneracy condition holds:
\begin{align}
  \xi\equiv \lambda_+^{-1}\lambda_- \notin \mathbb N.\eqlab{nonResonance2}
\end{align}
\end{itemize}
By assumption (A), in particular \eqref{cgamma}, we have $c>0$ and hence $\dot x>0$ within $\Sigma_+=\{y>0\}$ from \bpws{}. 
Therefore we have the following proposition (see \figref{VICanards} and \cite[Proposition 4.2]{krihog}):
\begin{proposition}\proplab{pwsCanards}
 Suppose assumption (A) holds. Then for the Filippov system \bpws{}:
 \begin{itemize}
  \item[(a)] There exists a unique orbit of \eqref{slidingEqns}: \textnormal{the strong canard} $\gamma^s$, that is tangent to the strong eigenvector $v_-$ at the two-fold $q$. It connects $\Sigma_{sl}^+$ with $\Sigma_{sl}^-$ in finite time.
  \item[(b)] There exists a \textnormal{funnel}  within $\Sigma_{sl}^-$, confined by the strong canard and the invisible fold line $l^-$, consisting of \textnormal{weak canards} $\gamma^w$: orbits of \eqref{slidingEqns} that all pass through the two-fold tangent to the weak eigenvector $v_+$ at the two-fold $q$. All the weak canards connect $\Sigma_{sl}^+$ with $\Sigma_{sl}^-$ in finite time.
 \end{itemize}
\end{proposition}
\begin{remark}\remlab{canardsDefn}
 \edNew{The orbits in (a) and (b) are called canards since they are reminiscent of the canards in singular perturbed systems that connect attracting and repelling slow manifolds \cite{szmolyan_canards_2001}.}
\end{remark}

The forward non-uniqueness of solutions of the Filippov system \bpws{} at the two-fold $q$ is now apparent. We illustrate three possible forward orbits in \figref{VISliding}. One of these orbits, $U$, in blue in \figref{VISliding}, will be important in the following. It is the forward orbit of $q$ under the forward flow of $X^+$. 

The parameter $\xi$ in (B) is intrinsic to the PWS system under consideration. But its role in understanding the forward non-uniqueness at $q$ will only become clear when we regularize the PWS system (see \secref{intuition} and \lemmaref{twist} of \secref{sec:slow}).

To proceed, we consider the truncated, piecewise linear (PWL henceforth) versions of \bpws{}:
\begin{align}
 \dot x &=\left\{\begin{array}{cc} 
                   \beta^{-1} c &\text{for}\quad y>0\\
                  -1& \text{for}\quad y<0
                  \end{array}\right.,\tpwl\\
 \dot y&=\left\{\begin{array}{cc}  b z&\text{for}\quad y>0 \\
                  -\beta x&\text{for}\quad y<0
                 \end{array}\right.,\nonumber\\
 \dot z&=\left\{\begin{array}{cc} 1 &\text{for}\quad y>0\\
                  b^{-1}\gamma &\text{for}\quad y<0                  
                 \end{array}\right.,\nonumber
\end{align}
This simplification allows us to make some precise statements in the sequel, which would otherwise have to be prefaced by a further assumption. We discuss this at length in \secref{discuss}. For \bpwl{}, we take  $\mathcal U=\mathbb R^3$, although we still think of our system as a local one. By solving \bpwl{}$_{y>0}$ with $(x(0),y(0),z(0))=q=(0,0,0)$, we obtain the following explicit expression for $U$:
\begin{align}
 U = \{(x,y,z)\in \mathcal U\vert x= \beta^{-1} cz,\,y=\frac{1}{2}bz^2,\,z\ge 0\},\eqlab{Uexpr}
\end{align}
after elimination of time. 
The sliding dynamics of \bpwl{} are shown in \figref{VICanards}. The funnel mentioned in \propref{pwsCanards} is shaded dark grey. \ed{Notice that the PWL system preserves the essential properties of \bpws{}: existence of stable and unstable sliding regions $\Sigma_{sl}^\mp$, weak canards, strong canard, and folds. In fact, the strong canard for \bpwl{} now coincides with the span of $v_-$:
\begin{align}
 \gamma^s=\{(x,y,z) = v_-s\vert s\in \mathbb R\}.\eqlab{gammas}
\end{align}
Similarly, there is also a special (``geometrically'' unique) weak canard that coincides with the span of $v_+$:
\begin{align}
 \gamma^w=\{(x,y,z) = v_+s\vert s\in \mathbb R\}.\eqlab{gammaw}
\end{align}}

\begin{figure}[h!] 
\begin{center}
{\includegraphics[width=.59\textwidth]{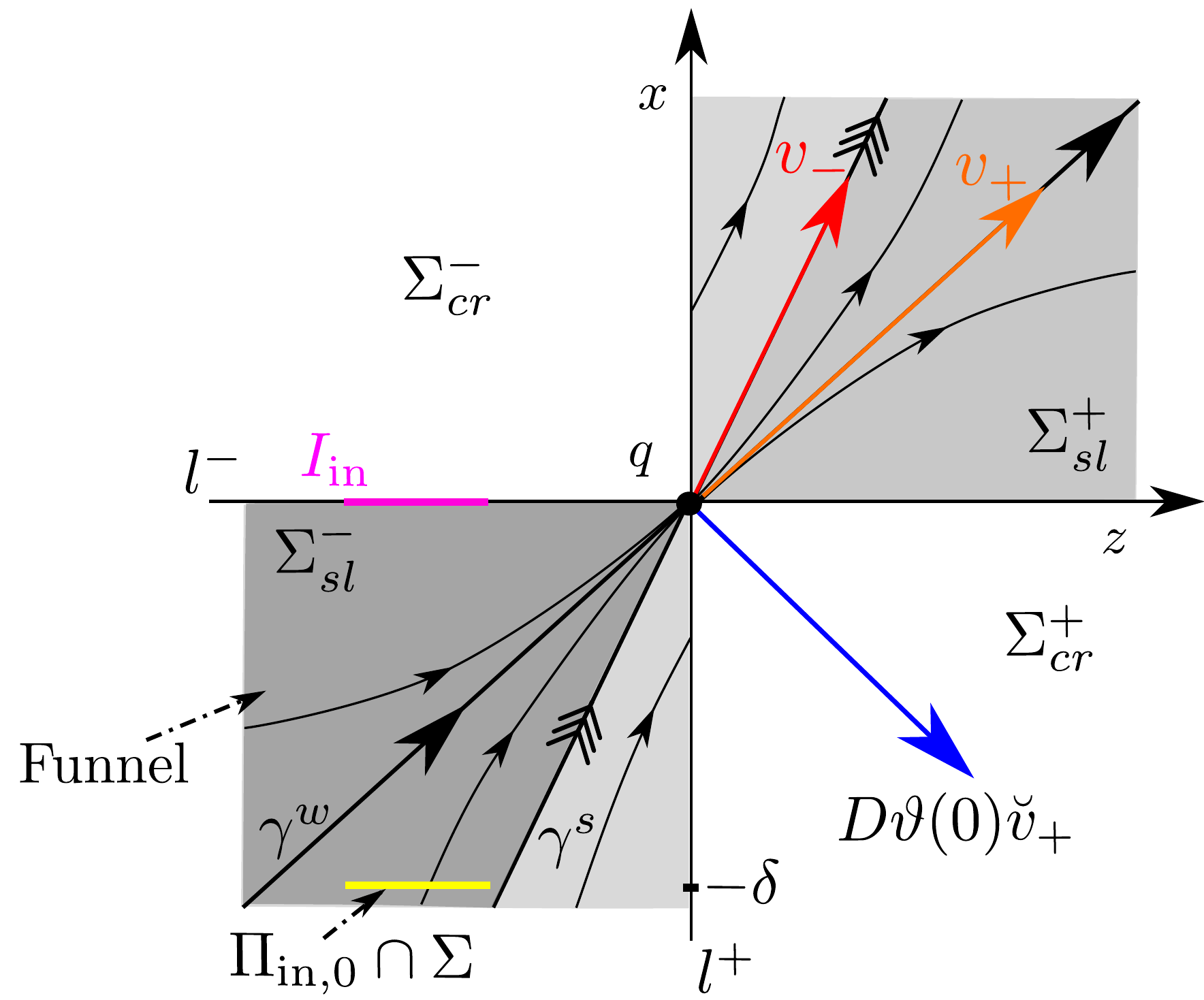}}
\end{center}
 \caption{The sliding flow within $\Sigma_{sl}^\pm$ for the truncated, piecewise linear equations \bpwl{}. The eigenvectors $v_{\mp}$ \eqref{vpmNew} are shown in red and orange, respectively. The funnel mentioned in \propref{pwsCanards} is shaded dark grey. It is confined within $\Sigma_{sl}^-$ by the invisible fold line $l^-$ and the strong canard $\gamma^s$. $I_{\text{in}}$ in purple is a closed interval on the negative $z$-axis (with properties described in \remref{xdotNeg}). The mapping $\vartheta$ is defined in \lemmaref{sigma0}. The image $D\vartheta(0)\breve v_+$ of $\breve v_+$ under $D\vartheta(0)$, shown in blue, is always contained outside the funnel. }
\figlab{VICanards}
\end{figure}

\subsection{Regularization of PWS systems}\seclab{regularization}
In this paper, we consider the following regularization of PWS systems.
\begin{definition}
A regularization of the PWS system $X^\pm(\textbf x)$ in \eqref{PWSsystem} is a smooth vector-field
\begin{eqnarray}
 X_\epsilon ({\bf x}) =
 \frac12 X^+({\bf x})(1+\phi(\epsilon^{-1} y)) +\frac12 X^-({\bf x})(1-\phi(\epsilon^{-1}y)),\eqlab{Xeps}
\end{eqnarray}
for $0<\epsilon\ll 1$, where the {\it regularization function} $\phi:\mathbb R\rightarrow [-1,1]$ is assumed to be sufficiently smooth and monotone, $\phi'(s) >0$ for $\phi(s)\in (-1,1)$ and
\begin{align*}
 \phi(s)\rightarrow \pm 1\quad \text{for}\quad s \rightarrow \pm \infty.
\end{align*}
\end{definition}
Then for $y>0$, by \eqref{Xeps}, we have  the convergence 
\begin{align*}
 X_\epsilon ({\bf x}) \rightarrow \frac12 X^+({\bf x})(1+1) +\frac12 X^-({\bf x})(1-1) =X^+({\bf x})  \quad \text{for}\quad \epsilon \rightarrow 0^+.
\end{align*}
Similarly, for $y<0$
\begin{align*}
 X_\epsilon ({\bf x}) \rightarrow X^-({\bf x})  \quad \text{for}\quad \epsilon \rightarrow 0^+.
\end{align*}

Using the truncated, piecewise linear system \bpwl{}, the regularized system becomes, after replacing time $t$ by $2t$:
\begin{align}
\dot x &=\beta^{-1} c(1+\phi(\epsilon^{-1} y))-(1-\phi(\epsilon^{-1} y)),\trpwl\\
\dot y&=b {z} (1+\phi(\epsilon^{-1} y)) -\beta  x(1-\phi(\epsilon^{-1} y)),\nonumber\\
 \dot z&=(1+\phi(\epsilon^{-1} y))+b^{-1} \gamma(1-\phi(\epsilon^{-1} y)),\nonumber
\end{align}
or more compactly $\dot{\textbf{x}}=2X_\epsilon(\textbf x)$. Clearly, it is \textit{time-reversible} with respect to the symmetry 
\begin{align}
(t,x,y,z)\mapsto (-t,-x, y,-z).\eqlab{timeRev}
\end{align}
\edNewNew{Define the smooth function
\begin{align*}
 h:\,(0,\infty)\rightarrow \mathbb R, 
 \end{align*}
by
 \begin{align}
  h(s) &=\phi^{-1}\left(\frac{1-\beta ^{-1}{b}s}{1+\beta ^{-1} b s}\right).
  \eqlab{haFunction}
 \end{align}
Then we can write the $y$-nullcline of \brpwl{} as follows
 \begin{align*}
  y = \epsilon h(x^{-1}z),\,\quad \text{for}\quad (x,z)\in \breve \Sigma_{sl}.
 \end{align*}
 Notice that $h(s)\rightarrow \pm \infty$ for $s\rightarrow 0^+$ and $s\rightarrow \infty$, respectively. }


System \brpwl{} is singular at $y=\epsilon=0$. Therefore it will be useful to work with two separate time scales. We say that $t$ in \brpwl{} is the \textit{slow time} whereas $\tau = t\epsilon^{-1}$ is the \textit{fast time}. In terms of the fast time scale we obtain the following equations
\begin{align}
x' &=\epsilon\left(\beta^{-1} c(1+\phi(\epsilon^{-1} y))-(1-\phi(\epsilon^{-1} y))\right),\trpwlf\\
y'&=\epsilon \left(b {z} (1+\phi(\epsilon^{-1} y)) -\beta  x(1-\phi(\epsilon^{-1} y))\right),\nonumber\\
z'&=\epsilon\left((1+\phi(\epsilon^{-1} y))+b^{-1} \gamma(1-\phi(\epsilon^{-1} y))\right),\nonumber
\end{align}
or simply $\textbf{x}'=2\epsilon X_\epsilon(\textbf x)$, where $'$ denotes $\frac{d}{d\tau}$. 
For $\epsilon=0$ in \brpwlf{} we have the trivial dynamics 
\begin{align}
\textbf{x}'=0 \Longrightarrow \mbox{all points $\textbf{x}\in \mathbb R^3$ are equilibria.}\eqlab{ppp}
\end{align}
\subsection{Regularization functions}\seclab{functions}
In our previous work \cite{krihog,krihog2} we restricted attention to the {\em Sotomayor and Teixeira} \cite{Sotomayor96} class of regularization functions:
\begin{definition}\defnlab{STphi}
The $C^k$-smooth, $k\ge 1$, Sotomayor and Teixeira regularization functions \cite{Sotomayor96} $$\phi:\,\mathbb R\rightarrow \mathbb R,$$ satisfy:
\begin{itemize}
 \item[$1^\circ$] \textnormal{Finite deformation}: \begin{eqnarray}
 \phi(\hat y)=\left\{\begin{array}{cc}
                 1 & \text{for}\quad \hat y\ge 1,\\
                 \in (-1,1)& \text{for}\quad \hat y\in (-1,1),\\
                 -1 & \text{for}\quad \hat y\le -1,\\
                \end{array}\right.\eqlab{phiFunc}
 \end{eqnarray}
\item[$2^\circ$] \textnormal{Monotonicity}: \begin{align}
 \phi'(\hat y)>0 \quad \text{within}\quad \hat y\in (-1,1).\eqlab{phiProperties}
 \end{align}
\end{itemize}
\end{definition}
The desirable property of this class of regularization functions is the finite deformation property. By $1^\circ$
 \begin{align}
  X_\epsilon = X^\pm \quad \text{for}\quad y\gtrless \pm \epsilon.\eqlab{XepsProp}
 \end{align}
 
In this paper, we consider the analytic, non-Sotomayor and Teixeira regularization function  
 \begin{align}
  \phi(s) = \frac{2}{\pi}\text{arctan}(s),\quad s\in \mathbb R\eqlab{arctan}
 \end{align}
because we wish to extend the theory of regularizations of PWS systems. Our results generalise to non-Sotomayor and Teixeira regularization functions other than \eqref{arctan}, but the detailed description depends on asymptotic properties of the regularization functions at $\pm \infty$. See \secref{discuss} for further details.

In the rest of this paper, we set $\phi$ as in \eqref{arctan}. It has the property that
\begin{align}
 \phi(s) = \pm 1 - \frac{2}{\pi}s^{-1} (1+\phi_2(s^{-2})),\eqlab{phiasympt}
\end{align}
using \eqref{hatEpshatY}, where we introduce the smooth function $\phi_2:[0,\infty)\rightarrow \mathbb R$ satisfying $\phi_2(0) = 0$. For brevity and simplicity in the sequel, we also introduce $\phi_+:[0,\infty)\rightarrow \mathbb R$ and $\phi_-:(-\infty,0]\rightarrow \mathbb R$ as
\begin{align}
 \phi_\pm(u) = \pm 1 -\frac{2}{\pi}u (1+\phi_2(u^2)),\eqlab{phipm}
\end{align}
so that 
\begin{align}
\phi(s) = \phi_\pm (s^{-1}),\quad \text{for} \quad s\gtrless 0. \eqlab{phipmProp}
\end{align}
\edNewNew{
Also, with $\phi$ as in \eqref{arctan} so that $\phi^{-1} =\tan\left(\frac{\pi}{2}s\right)$, $h:(0,\infty)\rightarrow \mathbb R$ in \eqref{haFunction} has the following asymptotics:
 Let 
  \begin{align}
 h_0(s) &=-\left(\beta^{-1} b\pi s\right)^{-1},\nonumber \\
  h_\infty(s) &= -\frac{1}{\pi}\left(1+\beta ^{-1}{b}s\right),\eqlab{asymptoteh}
 \end{align}
 for $s>0$. 
  Then 
 \begin{align*}
  h(s) &= h_0(s)(1+o(1)),\\
  h(s) &= h_\infty(s)+o(1),
 \end{align*}
%
 for $s\rightarrow 0^+$ and $s\rightarrow \infty$, respectively.
We sketch the graph of $h$ in \figref{hFunc}.  }
\begin{figure}[h!] 
\begin{center}
{\includegraphics[width=0.5\textwidth]{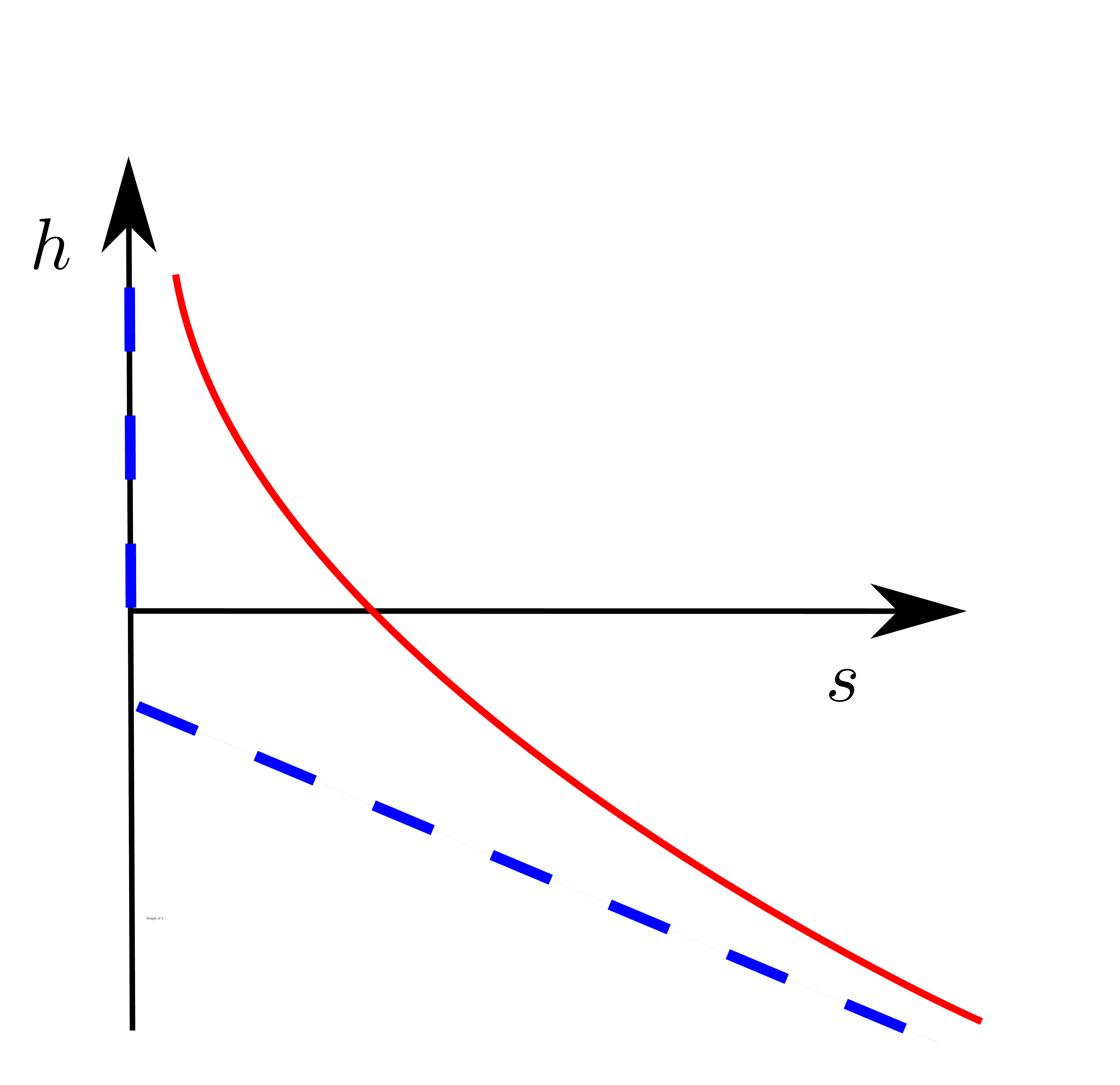}}
\end{center} 
 \caption{\ed{Sketch of the graph of $h$ in red, given by \eqref{haFunction}. The vertical asymptote $s=0$ and $h_\infty(s)$, given by \eqref{asymptoteh}, are dashed in blue}. }
\figlab{hFunc}
\end{figure}

\subsection{Main result}\seclab{main0}


Fix $\delta>0$ and $\nu>0$. Then consider the sections $\Pi_{\text{in},\epsilon}$, $0\le \epsilon\ll 1$, and $\Pi_{\text{out}}$ transverse to the flow of \brpwl{} for $0<\epsilon\ll 1$, defined as follows:
\begin{align}
 \Pi_{\text{in},\epsilon} &= \{(x,y,z)\in \mathbb R^3\vert x=-\delta,\,(y,z)\in R_{\text{in},\epsilon}\},\eqlab{Pin}\\
 \Pi_{\text{out}} &= \{(x,y,z)\in \mathbb R^3\vert y=\nu,\,(x,z)\in R_{\text{out}}\},\eqlab{Pout}
\end{align}
where\begin{itemize}
\item \edNew{The set
$R_{\text{in},\epsilon}$
is a suitable rectangle in the $(y,z)$-plane, depending continuously on $\epsilon>0$, so that within $\Pi_{\text{in},\epsilon}$ we have $\dot x\ge c_1>0$ for \brpwl{} for some $c_1>0$ and all $\epsilon>0$ sufficiently small.  
Furthermore, $\Pi_{\text{in},0}\cap \Sigma$ is contained inside the funnel but does not include the span of the weak eigendirection $v_+$. Also, $\Pi_{\text{in},0}\subset \{y\in [0,\zeta]\}$ for $\zeta>0$ sufficiently small so that all points in $\Pi_{\text{in},0}\cap \Sigma_+$ reach the funnel under the forward flow of $X^+$. Notice, that $\Pi_{\text{in},0}$, by these assumptions, is transverse ($\dot x>0$) to the PWS Filippov flow, illustrated in \figref{VISliding}. 
See also \remref{xdotNeg} below.}
%
\item $R_{\text{out}}$ is a suitably small rectangle in the $(x,z)$-plane so that $\Pi_{\text{out}}\subset\{y=\nu\}$ is a small section that is transverse to $U$ at the point 
\begin{align*}
u_{\text{out}} = U\cap \{y=\nu\}.
\end{align*}
Notice by \eqref{Uexpr}, that 
\begin{align*}
 u_{\text{out}} = \left( \beta^{-1}c\sqrt{2 \nu/b},\nu,\sqrt{2 \nu/b}\right).
\end{align*}
\end{itemize}

$\Pi_{\textnormal{in},\epsilon}$ and $\Pi_{\textnormal{out}}$ are naturally parametrized by $(y,z)\in R_{\text{in},\epsilon}$ and $(x,z)\in R_{\text{out}}$, respectively. We therefore consider the local mapping $$\mathcal L_\epsilon:\,\Pi_{\textnormal{in},\epsilon} \rightarrow \Pi_{\textnormal{out}}, \quad (y,z)\mapsto \mathcal L_\epsilon(y,z),$$ obtained by the forward flow of \brpwl{} for $0<\epsilon\ll 1$.

Our main technical result is the following theorem. 
\begin{theorem}\thmlab{mainThm}
Suppose assumptions (A) and (B) hold (see \eqsref{cgamma}{nonResonance2}).  Then there exists an $\epsilon_0>0$ such that for $0<\epsilon\le \epsilon_0$, $\mathcal L_\epsilon$ is well-defined and $C^1$. Moreover, 
\begin{align*}
 \max_{(y,z)\in R_{\textnormal{in},\epsilon}} \vert \mathcal L_\epsilon(y,z) - u_\textnormal{out}\vert + \max_{(y,z)\in R_{\textnormal{in},\epsilon}} \vert D\mathcal L_\epsilon(y,z)\vert =o(1),
\end{align*}
as $\epsilon\rightarrow 0^+$.
\end{theorem}

We prove \thmref{mainThm} in \secref{proof} after having introduced some further background in \secref{initialblowup} and \secref{geometry}.

\begin{remark}\remlab{xdotNeg}
Notice, that we do not take $\Pi_{\text{\textnormal{in}},\epsilon}$ with $R_{\text{\textnormal{in}},\epsilon}$ \eqref{Rin} as a neighbourhood of $y=0$: $R_{\text{\textnormal{in}},\epsilon}=R_{\text{\textnormal{in}}}=[-\zeta,\zeta]\times I_{\text{\textnormal{in}}}$, because then we would have $\dot x<0$ for $y<0$ while $\dot x>0$ for $y>0$ cf. \bpwl{}. Then $\mathcal L_\epsilon$ cannot be injective. Instead, one way to realise $\Pi_{\text{\textnormal{in}},\epsilon}$ is to define $R_{\text{\textnormal{in}},\epsilon}$ as follows:
\edNew{\begin{align}
R_{\text{in},\epsilon}=[\epsilon \omega,\zeta]\times I_{\text{in}}.\eqlab{Rin}
\end{align}
Here $I_{\text{in}}$ in \eqref{Rin} is a suitable closed interval on the negative $z$-axis so that $\Pi_{\textnormal{in},0}\cap \Sigma$ is contained inside the funnel but does not include the weak canard. See also \figref{VICanards} (in purple). Furthermore, $\omega$ in \eqref{Rin} satisfies the following inequality
\begin{align*}
 \omega> \omega_0\equiv \phi^{-1}\left(\frac{1-\beta^{-1}c}{1+\beta^{-1}c}\right),
\end{align*}
where $\phi^{-1}:(-1,1)\rightarrow \mathbb R$ is the inverse of $\phi$. With $\phi$ as in \eqref{arctan} we have $\phi^{-1}(s) =\tan\left(\frac{\pi}{2}s\right)$. Therefore the number $\omega_0$ is such that the set $\{(x,y,z)\vert y=\epsilon \omega_0,(x,z)\in \breve \Sigma_{sl}\}$ is the $x$-nullcline of \brpwl{}. In particular, $\dot x>0$ for any $(x,y,z)\in \{(x,y,z)\vert y>\epsilon\omega_0,(x,z)\in \breve \Sigma_{sl}\}$. }


\end{remark}

%
%
%

\subsection{Discussion}\seclab{intuition}
One can interpret the theorem in the following way: the forward orbit $U$ is distinguished amongst all the possible forward orbits leaving $q$. Indeed, the image $\mathcal L_\epsilon(\Pi_{\text{in},\epsilon})$ tends to $u_\text{out} = U\cap \Pi_{\text{out}}$ (see \figref{VISliding}) in the Hausdorff-distance as the regularized system approaches the PWS system ($\epsilon\rightarrow 0^+$). Intuitively, this makes sense because
$U$ is the only \textit{stable} exit from the two-fold of the PWS system. The other forward orbits in \figref{VISliding} are very \textit{fragile} due to the unstable sliding region. Furthermore, if we consider an initial condition of $X^+$ just above $\Sigma_{sl}^+$ close to $q$, then the forward flow follows close to $U$. 

But if we start below $\Sigma_{sl}^-$ and follow $X^-$, then since $l^-$ is invisible, we return to $\Sigma$ and potentially even $\Sigma_{sl}^-$. One could then imagine that this \textit{rotation} could continue indefinitely (\textit{slide} along $\Sigma_{sl}^-$ within funnel region; when reaching $\Sigma_{sl}^+$ follow $X^-$ and return to $\Sigma_{sl}^-$; \textit{slide} along $\Sigma_{sl}^-$ within funnel region; and so on) so that the forward orbit of the regularized system would never leave a small vicinity of the two-fold. A related phenomenon occurs in PWS systems with intersecting discontinuity sets \cite{guglielmi2017a}. The following lemma excludes this behaviour:
\begin{lemma}\lemmalab{sigma0}
Let 
\begin{align*}
\vartheta:\,&\breve \Sigma\cap \{x>0\}\rightarrow \breve \Sigma,\quad 
\vartheta(x,z) = (x_+(x,z),z_+(x,z)),
\end{align*}
be so that $(x_+(x,z),0,z_+(x,z))$ is the first return of $(x,0,z)\in \Sigma\cap \{x>0\}$ to $\Sigma$ under the forward flow of $X^-$. Then 
 \begin{align}
  \vartheta(x,z) = (-x,z+\frac{2\gamma}{b}x),\quad (x,z)\in \breve \Sigma\cap \{x>0\}.\eqlab{eq.sigma0Neg}
 \end{align}
 In particular, $\vartheta$ has a smooth extension to $(x,z)\in \breve \Sigma\cap \{x\ge 0\}$, so that $$(x,0,z)\mapsto (x_+(x,z),0,z_+(x,z)),$$ leaves the invisible fold line $l^-$ fixed: $(0,0,z)\mapsto (0,0,z)$, or simply $\vartheta(0,z)=(0,z)$, for all $z\in \mathbb R$. Furthermore,
 \begin{align}
  D\vartheta(0)\breve v_+ = \begin{pmatrix}
                         -1\\
                        z_1^*
                        \end{pmatrix},\eqlab{Dsigmav}
 \end{align}
 with $\breve v_+$ \eqref{vpm}, the weak direction of the node of \eqref{slidingVectorField}, and where
 \begin{align*}
  z_1^*\equiv -\chi_++\frac{2\gamma}{b}.\end{align*}
The quantity $z_1^*$ satisfies the following inequality
\begin{align}
z_1^*>\chi_-.\eqlab{z1chiN0}
\end{align}
\end{lemma}
\begin{proof}
See \appref{lemma0}.
\end{proof}
A similar local result holds for $X^-$ in \bpws{}$_{y<0}$.
The consequence of \eqref{z1chiN0} is that the half-line $$\{(x,y,z)\in \Sigma \vert\, (x,z) = s D\vartheta(0)\breve v_+=s(-1,z_1^*),\,s\in (0,\infty)\},$$ obtained from \eqref{Dsigmav} as the image of $\gamma^w\cap \Sigma_{sl}^+$, is always contained {\em outside} the funnel of the PWS Filippov system (see \figref{VICanards}). Therefore initial conditions within $\Sigma_{sl}^+$ with $(x,y,z)=s v_+$, $s>0$ small, upon following $X^-$, will return to $\Sigma$ with $(x,y,z)=s(-1,0,z_1^*)$ cf. \eqref{eq.sigma0Neg}. Then from there they follow: 
\begin{itemize}
\item (for $z_1^*> 0$): $X^+$ through crossing. 
\item (for $z_1^*\le 0$): $X_{sl}$ up to the visible fold line $l^+$ and then from there subsequently $X^+$.
\end{itemize}
For $s$ small, in both cases the resulting forward orbit within $\Sigma^+$ will then remain close to $U$. 

\ed{Furthermore, following \cite{krihog}, it is known that there exist slow manifolds $S_{a,\epsilon}$ and $S_{r,\epsilon}$ of \brpwl{} (see also \propref{Fenichel} below) that carry reduced flows that are smoothly $\mathcal O(\epsilon)$-close to the sliding flow. These invariant manifolds intersect transversely along a perturbed weak canard if condition (B) holds (see also \lemmaref{lines} below). This implies, due to the contraction towards the weak canard, that $S_{a,\epsilon}$ eventually aligns itself, in a neighbourhood of the canard on the repelling side, with the weak canard's foliation of unstable fibers. These unstable fibers are due to the existence of unstable fibers of the repelling slow manifold $S_{r,\epsilon}$. Hence, initial conditions within $\Pi_{\text{in},\epsilon}$ will eventually move up or down (we will describe this carefully in cases (a) and (b) below using the position of $\Pi_{\text{in},\epsilon}$ and the eigenvalue ratio $\xi$) and follow $X^+$ or $X^-$ for $0<\epsilon \ll 1$. Combining the above information, about the PWL system in \lemmaref{sigma0} and the weak canard of the regularization, gives the key intuition into why our main theoretical result holds.}

\subsection{Global dynamics}\seclab{globaldynamics}
Suppose we make a mild assumption about the global dynamics of the regularization of the PWS $X=(X^+,X^-)$.
\begin{itemize}
 \item[(C)] There exists a $C^1$ (global) mapping $\mathcal G_\epsilon:\Pi_{\textnormal{out}}\rightarrow \Pi_{\textnormal{in},\epsilon}$, obtained by the forward flow of the regularization, with uniformly bounded derivative: $\vert D\mathcal G_\epsilon\vert \le K$ for some $K>0$ and all $\epsilon>0$ sufficiently small.
\end{itemize}
Then by the Contraction Mapping Theorem, the following result is a simple corollary of \thmref{mainThm}.
\begin{corollary}\corlab{cor}
Suppose (A), (B), (C) and consider the Poincar\'e return mapping 
\begin{align*}
 \mathcal P_\epsilon =  \mathcal L_\epsilon\circ \mathcal G_\epsilon:\,\Pi_{\textnormal{out}}\rightarrow \Pi_{\textnormal{out}},\,(x,z)\mapsto (x_+(x,z),\nu,z_+(x,z)).
\end{align*}
Then there exists an $\epsilon_0>0$ such that for $0<\epsilon\le \epsilon_0$, the mapping $\mathcal P_\epsilon$ has a unique and attracting fixed point, which is $o(1)$-close to $u_\textnormal{out}$. 
\end{corollary}
A similar result is known for the passage through folded nodes in slow-fast systems with one fast and two slow variables \cite[Theorem 4.1]{brons-krupa-wechselberger2006:mixed-mode-oscil}. We will discuss this connection further in \secref{discuss}. In \figref{examples} we illustrate some PWS examples with global returns to the two fold. \corref{cor} shows that the regularization of these systems \textit{can}  (see \secref{discuss} below, in particular assumption (D), for further discussion) produce attracting limit cycles that follow the \textit{singular PWS cycle} (illustrated in \figref{examples} using thick blue lines). 

\begin{figure}[h!] 
\begin{center}
\subfigure[]{\includegraphics[width=.49\textwidth]{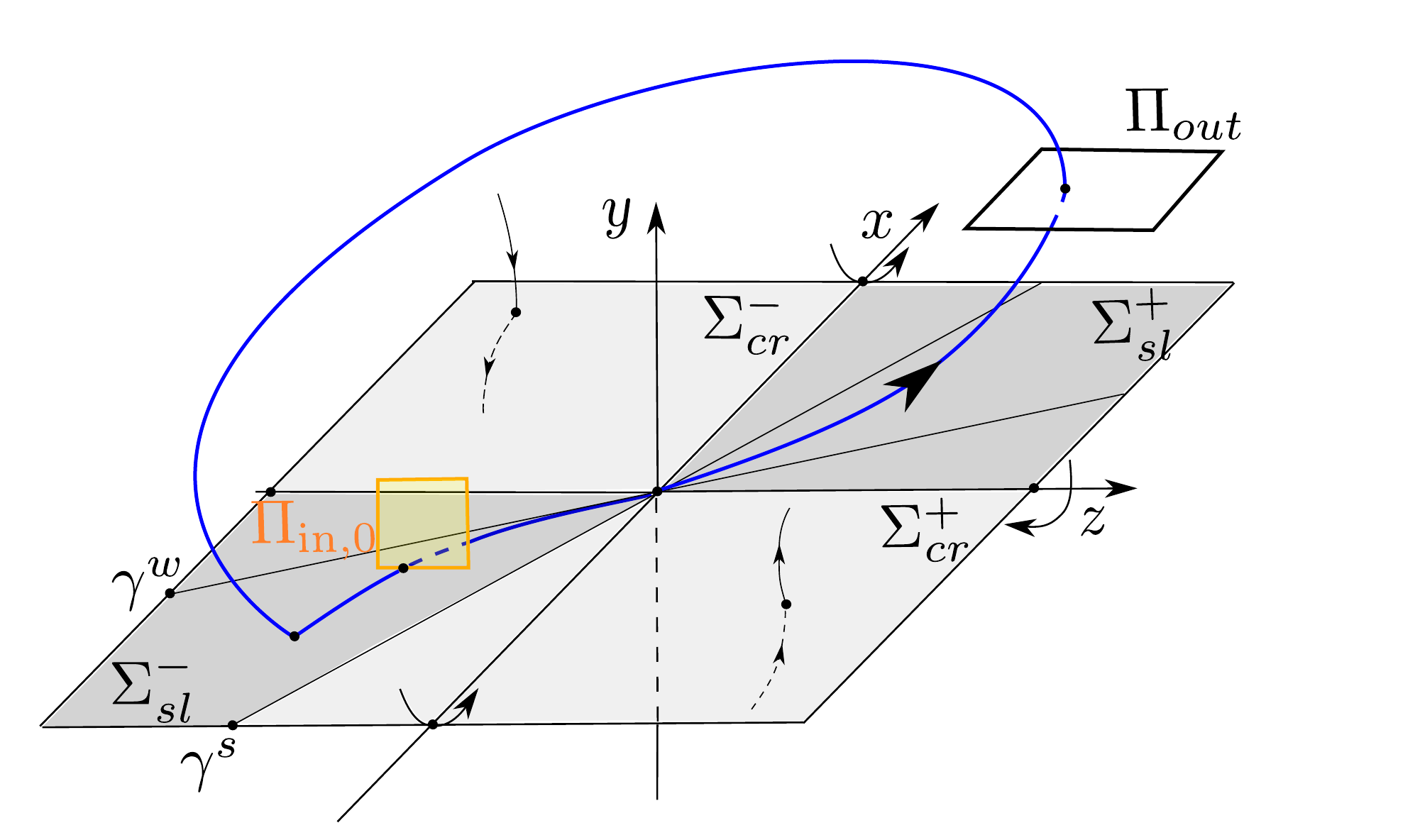}}
\subfigure[]{\includegraphics[width=.49\textwidth]{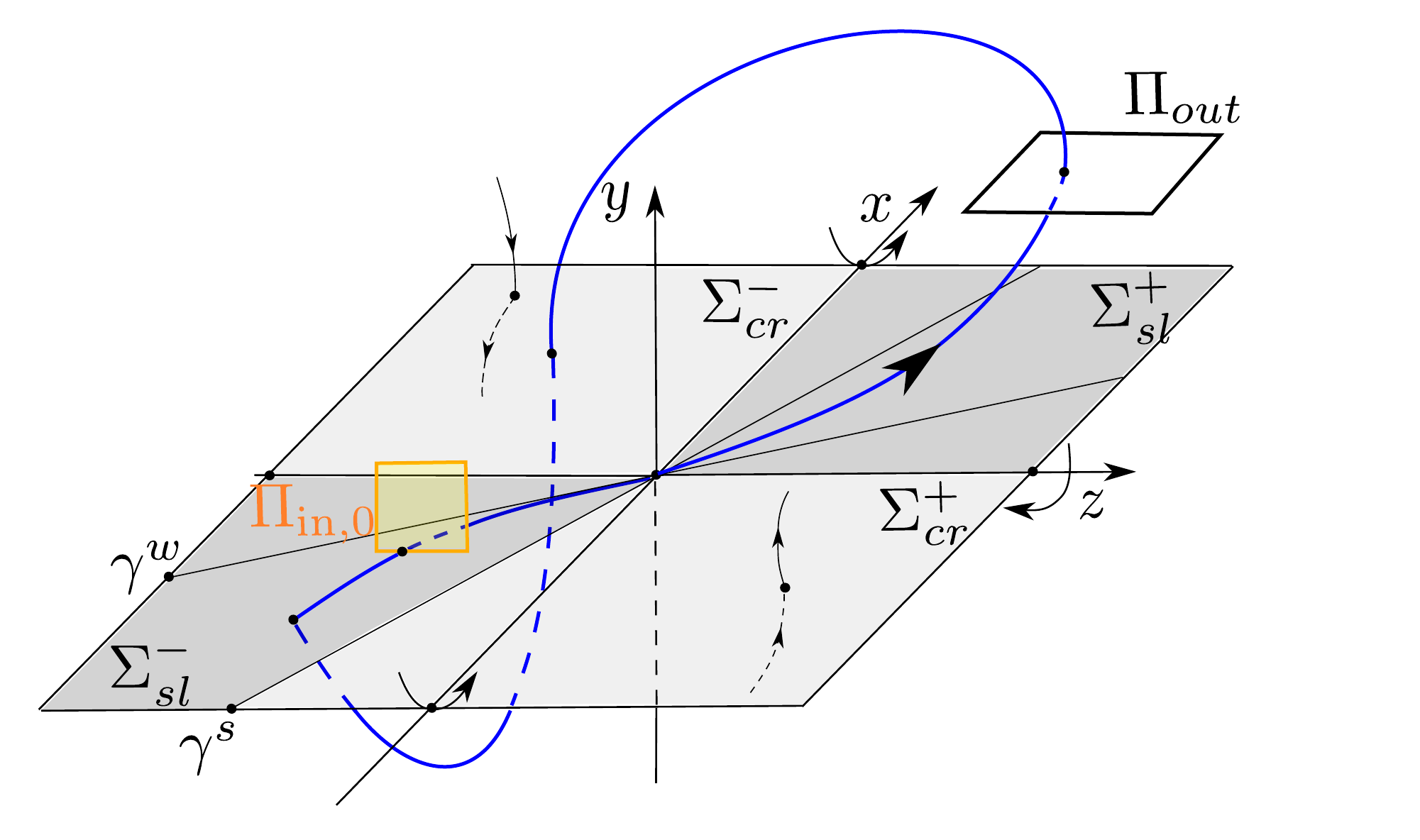}}
\subfigure[]{\includegraphics[width=.49\textwidth]{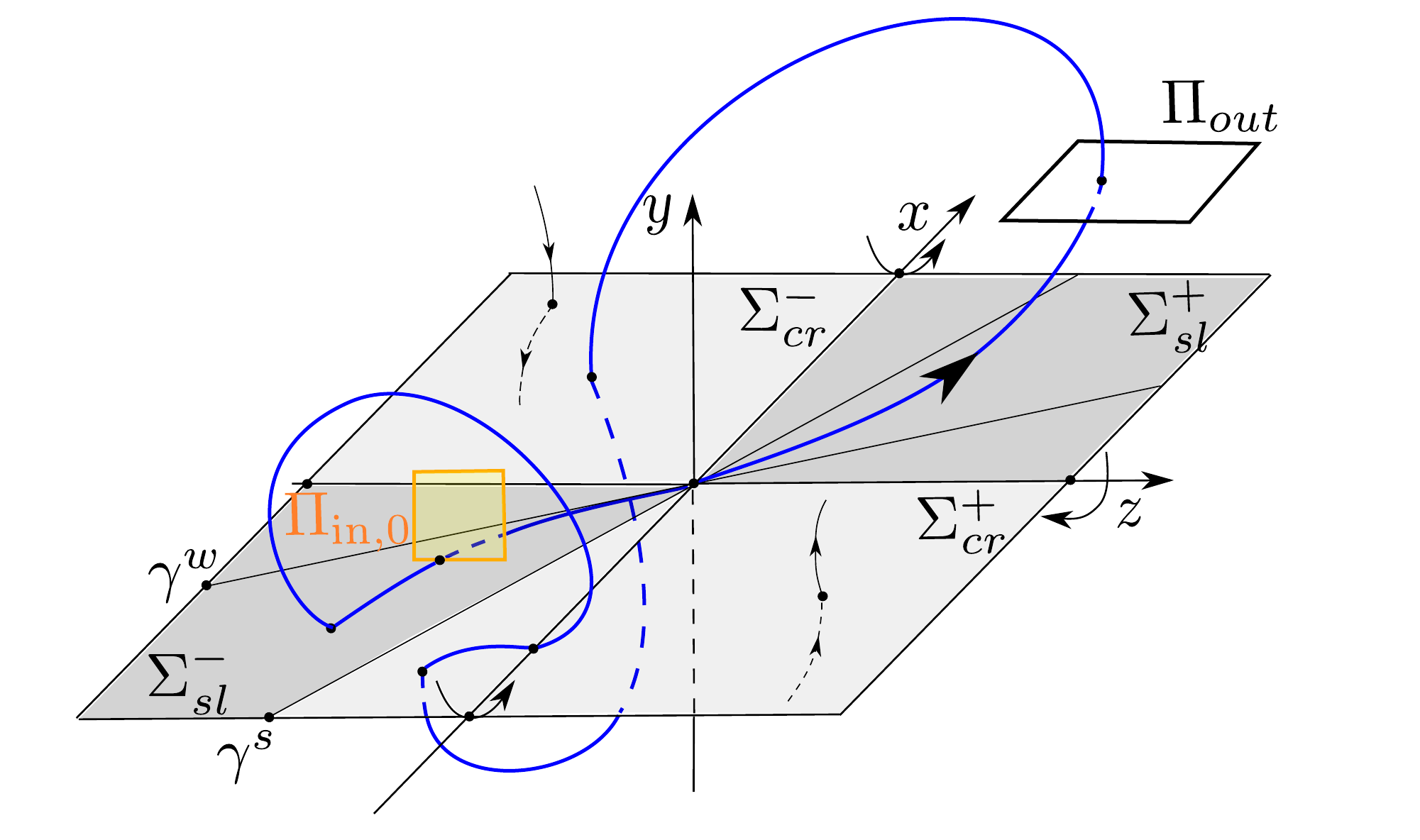}}
\end{center}
 \caption{Different (global) PWS systems with a visible-invisible two-fold, satisfying assumptions (A) and (B), where the segment $U$ returns to $q$ under the forward flow of $X^\pm$ and $X_{sl}$: (a) the segment $U$ intersects the funnel directly; (b) the segment $U$ intersects $\Sigma_{cr}^-$ and then follows $X^-$ up to the funnel; (c) as in (b), but the orbit segment of $X^-$ returns outside the funnel, then follows $X_{sl}$ up until the visible fold line $l^+$ (regularization of which is also described in \cite{krihog}), and then the orbit of $X^+$ returns the funnel to complete the closed cycle. \corref{cor} of \secref{globaldynamics} shows that the regularization of these systems can produce attracting limit cycles.}
\figlab{examples}
\end{figure}


\subsection{Outline of the rest of the paper}
\ed{In \secref{initialblowup}, we blowup our regularized system. In \secref{newDirectional}, we present the directional charts we use to study this blowup. Then in \secref{geometry}, we present the geometry of the regularization using our blowup approach. Here we also demonstrate the use of our directional charts, show the equivalence between stable/unstable sliding and normally hyperbolic, attracting/repelling critical manifolds and how folds and two-folds give rise to loss of hyperbolicity. In the present case, the two-fold becomes a circle $\bar q$ of fully nonhyperbolic critical points of the blown up regularization. These sections are kept fairly general, and the methods used, apply to general regularizations of different PWS systems. We believe that our geometric approach to these kind of problems is new. 

The main result \thmref{mainThm} is then proved in \secref{proof} using an additional blowup of the circle $\bar q$. We complete our paper in \secref{discuss} with a discussion and conclusion section and consider possible extensions of our results. }




%

\section{Blowup of the regularization of a PWS system}\seclab{initialblowup}

To study \brpwl{} (or equivalently \brpwlf{}) we proceed as in \cite{kriBlowup,krihog,krihog2} by considering the following blowup transformation 
\begin{align}
(x,z,\rho ,(\bar y,\bar \epsilon))&\mapsto (x,z,y,\epsilon),\eqlab{blowup0Transformation}
%
%
\end{align}
defined by
\begin{align}
 (y,\epsilon) = \rho (\bar y,\bar \epsilon),\quad (\rho ,(\bar y,\bar \epsilon))\in [0,\infty)\times S^1,\eqlab{blowup0}
\end{align}
where
\begin{align*}
S^1 = \{(\bar y,\bar \epsilon)\vert \bar y^2+\bar \epsilon^2=1\},
\end{align*}
is the unit-circle. 
\ed{The preimage of $y=\epsilon=0$ under \eqref{blowup0Transformation} is $\rho=0,(\bar y,\bar \epsilon)\in S^1$. One therefore says that the transformation \eqref{blowup0Transformation} (or, in fact, the inverse process) \textit{blows up} $y=\epsilon=0$ to the unit circle $(\bar y,\bar \epsilon)\in S^1$, as shown in \figref{InitialCharts}. \eqref{blowup0Transformation} gives a vector-field $\overline X$ on the \textit{blowup space}
\begin{align*}
 \overline P = \{(x,z,\rho,(\bar y,\bar \epsilon))\in \mathbb R^2 \times [0,\infty)\times S^1\},
\end{align*}
by pull-back of the {\em extended fast-time vector-field}:
\begin{align}
\textbf{x}' &= \epsilon X_\epsilon(\textbf x),\eqlab{extended}\\
\epsilon' &=0,\nonumber
 \end{align}
see also \brpwlf{}. Here $\overline X\vert_{\rho=0}\ne 0$ for $\bar y\ne \pm 1$. However, $\overline X\vert_{(\bar y,\bar \epsilon)=(\pm 1,0)}\equiv 0$, recall \eqref{ppp}, by construction. We shall therefore consider the \textit{desingularized} vector-field  $\widetilde X$ defined by 
\begin{align}
\widetilde X = (\pm \bar y^{-1}\bar \epsilon)^{-1} \overline X\quad \text{for} \quad (\bar y,\bar \epsilon)\quad \text{near} \quad (\pm 1,0)\in S^1,\eqlab{Widetilde1}
\end{align}
respectively. We shall see that $\widetilde X$ is well-defined and non-trivial, even at $(\bar y,\bar \epsilon)=(\pm 1,0)$. Notice, that the orbits of $\widetilde X$ and $\overline X$ coincide for $(\bar y,\bar \epsilon)\ne (\pm 1,0)$ and $\bar \epsilon>0$ where the multiplication of $(\pm \bar y^{-1}\bar \epsilon)^{-1}>0$ corresponds to a nonlinear transformation of time (speeding up time near $(\bar y,\bar \epsilon) =(\pm 1,0)$). But for $(\bar y,\bar \epsilon)= (\pm 1,0)$ we now have $\widetilde X\ne 0$ and as a result $\widetilde X$ will have improved hyperbolicity properties. Obviously, we cannot divide $\overline X$ by $\pm \bar y^{-1}\bar \epsilon$ near $\bar y =0$. There we will just study the orbits of $\overline X$, even for $\rho=0$. Following this discussion, we therefore define $\widetilde X$ as 
\begin{align}
\widetilde X \equiv \psi (\bar y,\bar \epsilon)^{-1} \overline X, \quad \forall \quad (\bar y,\bar \epsilon)\in S^1,\eqlab{Widetilde2}
\end{align}
where $\psi:\,(\bar y,\bar  \epsilon)\in S^1\rightarrow [0,\infty)$ is suitable smooth interpolation of the time transformation, which satisfies the following:
\begin{itemize}
 \item[(i)] $\psi(\bar y,\bar \epsilon)>0$ for all $\bar \epsilon>0$;
 \item[(ii)] $\psi(\bar y,\bar \epsilon) = \pm \bar y^{-1}\bar \epsilon$
in a neighbourhood of $(\bar y,\bar \epsilon)=(\pm 1,0)$;
\item[(iii)] $\psi(\bar y,\bar \epsilon) = 1$
in a neighbourhood of $(\bar y,\bar \epsilon)=(0,1)$.
\end{itemize}
Clearly, such a smooth interpolation function $\psi$ exists. In fact, even a piecewise polynomial ($C^k$-smooth $\psi$) will do for our purposes. But the details are not important. 
It is $\widetilde X$ that we shall study in the following.

To describe $\overline P$ and $\widetilde X$, we could use polar variables:
\begin{align}
\bar y=\sin \theta,\,\bar \epsilon=\cos \theta.\eqlab{polar}
\end{align}
But in agreement with general theory \cite{krupa_extending_2001}, it is more useful to consider \textit{directional charts}. }

\subsection{Directional charts}\seclab{newDirectional}
\ed{We will use different directional charts in the sequel. We therefore define these blowup-dependent charts now (in some generality) and introduce our general terminology before we apply these concepts to \eqref{blowup0}. 

Consider therefore
 $\textbf x=(x_1,\ldots,x_n)\in \mathbb R^n$, $\alpha =(\alpha_1,\ldots,\alpha_n)\in \mathbb N^n$ and the following general, weighted (or quasihomogeneous \cite{kuehn2015}), blowup transformation:
 \begin{align}
  \Phi_\alpha:\,[0,\infty)\times S^{n-1}\rightarrow \mathbb R^n:\, (\rho,\bar{\textbf x})\mapsto \textbf x,\quad (\rho,\bar{\textbf x})\mapsto (\rho^{\alpha_1}\bar x_1,\ldots,\rho^{\alpha_n}\bar x_n),\eqlab{generalBlowup}
 \end{align}
Here the pre-image of $\textbf{x} = 0$ is $\{0\}\times S^{n-1}$ where
\begin{align*}
 S^{n-1} = \left\{\bar{\textbf x}=(\bar x_1,\ldots,\bar x_n)\in \mathbb R^n \vert \sum_{i=1}^n \bar x_i^2=1 \right\},
\end{align*}
is the unit $(n-1)$-sphere. The inverse process of \eqref{generalBlowup} therefore blows up $\textbf{x} = 0$ to $S^{n-1}$. The positive integers $\alpha_i\in \mathbb N$ are called the {\it weights} of the blowup, see \cite{kuehn2015}. 

\edNew{\begin{definition}
Let $j\in \{1,\ldots,n\}$ and write $\hat{\textbf{x}}^j=(\hat x_1,\ldots,\hat x_{j-1},\hat x_{j+1},\ldots,\hat x_n)\in \mathbb R^{n-1}$. Then the \textit{directional blowup} in the positive $j$-th direction is a {mapping}:
 \begin{align*}
  \Psi_\alpha^{j}:\,[0,\infty)\times \mathbb R^{n-1}\rightarrow \mathbb R^n, 
  \end{align*}
given by
\begin{align}
(\hat \rho,\hat{\textbf x}^j)&\mapsto \textbf x = (\hat \rho^{\alpha_1} \hat x_1,\ldots, \hat \rho^{\alpha_{j-1}}\hat x_{j-1},\hat \rho^{\alpha_{j}}, \hat \rho^{\alpha_{j+1}} \hat x_{j+1},\ldots,\hat \rho^{\alpha_{n}}\hat x_n).\eqlab{directionPos}
  \end{align}
  \edNewNew{ The \textit{directional chart} ``$\bar x_j=1$'' is then a coordinate chart 
  \begin{align*}
  \Xi^j_\alpha:\, [0,\infty) \times S^{n-1} \rightarrow [0,\infty)\times \mathbb R^{n-1},
  \end{align*}
  such that 
  \begin{align*}
  \Phi_\alpha = \Psi^j_\alpha \circ \Xi^j_\alpha,
  \end{align*} 
    }
  
  Similarly, we define the \textit{directional blowup} in the negative $j$-th direction as the mapping (new symbols compared to \eqref{directionPos})   \begin{align*}
  \Psi_\alpha^{j}:\,[0,\infty)\times \mathbb R^{n-1}\rightarrow \mathbb R^n, 
  \end{align*}
given by 
\begin{align}
(\hat \rho,\hat{\textbf x}^{j})&\mapsto \textbf x = (\hat \rho^{\alpha_1} \hat x_1,\ldots, \hat \rho^{\alpha_{j-1}}\hat x_{j-1},-\hat \rho^{\alpha_{j}}, \hat \rho^{\alpha_{j+1}} \hat x_{j+1},\ldots,\hat \rho^{\alpha_{n}}\hat x_n).\eqlab{directionNeg}
  \end{align}
  \edNewNew{   The \textit{directional chart} ``$\bar x_j=-1$'' is then a coordinate chart 
   \begin{align*}
  \Xi^{j}_\alpha:\, [0,\infty) \times S^{n-1}\rightarrow [0,\infty)\times \mathbb R^{n-1},
  \end{align*}
  such that 
  \begin{align*}
  \Phi_\alpha=\Psi^{j}_\alpha \circ \Xi^{j}_\alpha.
  \end{align*} }
  
  \end{definition}}
  
  \begin{figure}[h!] 
\begin{center}
{\includegraphics[width=.69\textwidth]{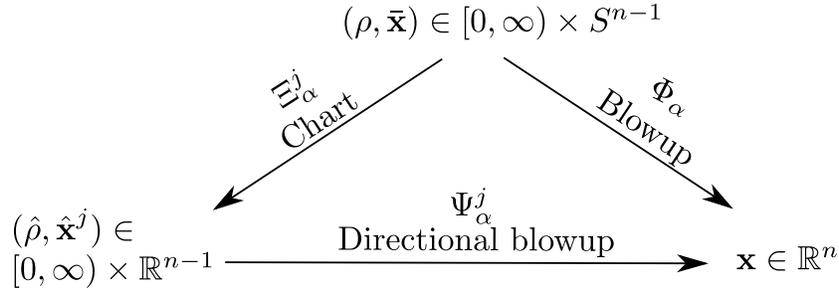}}
\end{center}
 \caption{\ed{Given a blowup and an associated directional blowup (two edges in the diagram), we define the corresponding chart as the mapping (the final, third edge in the diagram) that makes the diagram commute  (on a subset of $[0,\infty)\times S^{n-1}$)}. }
\figlab{commute}
\end{figure}
%
We illustrate the charts in \figref{commute}. Notice that the directional blowups in \eqsref{directionPos}{directionNeg}, respectively, are diffeomorphisms for $\hat \rho>0$. But the preimage of $\textbf x=0$ is $\hat \rho=0,\hat{\textbf{x}} \in \mathbb R^{n-1}$.


Now, straightforward calculations show that the chart $\bar x_j = 1$ is uniquely given by
\begin{align}
(\rho,\bar{\textbf {x}})\mapsto \left\{\begin{array}{cc}
 \hat \rho &= \rho \bar x_j^{1/\alpha_j}\\
 \hat x_i &=\bar x_j^{-\alpha_i/\alpha_j} \bar x_i
 \end{array}\right.,
\eqlab{chartPos}
 \end{align}
for all $i\ne j$, 
and that the associated coordinate patch covers $S^{n-1}\cap \{\bar x_j>0\}$. Similarly, the chart $\bar x_j = -1$ is uniquely given by
\begin{align}
(\rho,\bar{\textbf {x}})\mapsto \left\{\begin{array}{cc}
 \hat \rho &= \rho (-\bar x_j)^{1/\alpha_j}\\
 \hat x_i &=(-\bar x_j)^{-\alpha_i/\alpha_j} \bar x_i
 \end{array}\right.
 ,\eqlab{chartNeg}
\end{align}
for all $i\ne j$ and the associated coordinate patch covers $S^{n-1}\cap \{\bar x_j<0\}$. Therefore the collection of all the charts $\bar x_i = \pm 1$, $i=1,\ldots,n$ provides an atlas on $(\rho,\bar{\textbf x})\in [0,\infty) \times S^{n-1}$ with smooth coordinate changes between charts that overlap. In particular, if $\alpha_i=1$ for all $i$ (in which case the blowup is said to be homogeneous or radial) then the charts $\bar x_j=\pm 1$ simply parametrise $S^{n-1}$ using stereographic projections onto the associated coordinate planes, see e.g. \figref{InitialCharts}.

With slight abuse of notation, we will, as is common in the literature, simply refer to \eqsref{directionPos}{directionNeg} as the (directional) charts $\bar x_j = \pm 1$, respectively (although they are actually the coordinate transformations in the local coordinates of the charts themselves,  see also \figref{commute}). 
Notice that the directional blowups are easy to compute: We just substitute $\bar x_j=\pm 1$ into \eqref{generalBlowup}, see \eqsref{directionPos}{directionNeg}.


%

%

In the context of \eqref{blowup0} we therefore obtain the chart $\bar \epsilon=1$ by setting $\bar \epsilon=1$ in \eqref{blowup0}:
\begin{align*}
 (\hat \rho,\hat y) \mapsto (y,\epsilon) = (\hat \rho \hat y,\hat \rho).
\end{align*}
To avoid too many symbols we eliminate $\hat \rho=\epsilon\ge 0$ and write the chart as
 \begin{align}
\bar \epsilon=1: \quad (\epsilon,\hat y)\mapsto y=\epsilon \hat y,\quad \hat y\in \mathbb R,\,\epsilon \ge 0.\eqlab{yhat}
\end{align}
Notice, that equating \eqref{yhat} with \eqref{polar} gives a coordinate change between the charts with $\hat y =\tan \theta$, $\theta \in (-\pi/2,\pi/2)$. Therefore, in agreement with the general discussion above, the chart $\bar \epsilon=1$ geometrically corresponds to a coordinate plane $(\hat y,\epsilon)$, with $\epsilon\ge 0$, attached to $\bar \epsilon=1$ on $S^1$ as illustrated in \figref{InitialCharts}, where the $\hat y$-coordinate axis parametrises $S^1\cap \{\bar \epsilon>0\}$ using stereographic projection. The chart $\bar \epsilon=1$ therefore also covers $(\bar y,\bar \epsilon)\in S^1\cap \{\bar \epsilon>0\}$.

To cover $(\bar y,\bar \epsilon)\in S^1\cap \{\bar y>0\}$ we introduce the directional chart $\bar y=1$ as follows:
\begin{align*}
(\hat \rho,\hat \epsilon) \mapsto (y,\epsilon ) =(\hat \rho,\hat \rho \hat \epsilon),
\end{align*}
introducing $\hat \epsilon\in [0,\infty)$ and (a new) $\hat \rho\in [0,\infty)$. Again, to avoid too many symbols, we eliminate $\hat \rho=y$ and simply write this ``chart'' as 
\begin{align}
\bar y=1:\quad \quad (y,\hat \epsilon)\mapsto \epsilon = y\hat \epsilon,\quad y\ge 0,\,\hat \epsilon\ge 0, \eqlab{chartEpshat2}
\end{align}
Similarly, we obtain the chart $\bar y=-1$ as
\begin{align}
\bar y=-1:\quad \quad (y,\hat \epsilon)\mapsto  \epsilon = y\hat \epsilon ,\quad y\le 0,\,\hat \epsilon \le  0, \eqlab{chartEpshat}
\end{align}
Geometrically, the charts $\bar y=\pm 1$ are illustrated in \figref{InitialCharts}. They parametrise $S^1\cap \{\bar y\gtrless 0\}$, respectively, using stereographic projections onto the $\hat \epsilon$-coordinate axes attached to $S^1$ at $\bar \epsilon=\pm 1$. 
{For simplicity, we use the same symbols in \eqsref{chartEpshat2}{chartEpshat}, even though the domains are  different ($\hat \epsilon\ge 0$ and $\hat \epsilon\le 0$, respectively). Notice also $(\bar y,\bar \epsilon)=(1,0)\mapsto \hat \epsilon=0$ in $\bar y=1$ while $(\bar y,\bar \epsilon)=(-1,0)\mapsto \hat \epsilon=0$ in $\bar y=-1$.} 
More specifically,
\begin{align}
 \hat \epsilon = \hat y^{-1},\eqlab{hatEpshatY}
\end{align}
and hence by \eqref{polar} $\hat \epsilon = \cot \theta$, $\theta\in (0,\pi/2]$ and $\theta \in [-\pi/2,0)$,
for $\hat y\lessgtr 0$ in \eqsref{chartEpshat2}{chartEpshat}, respectively, 
with $\hat y$ in \eqref{yhat}.

The charts \eqref{yhat}, \eqref{chartEpshat2} and \eqref{chartEpshat} now cover the half-circle $S^1\cap \{\bar \epsilon\ge 0\}$ completely. This is the relevant part of the circle since we are only interested in $\epsilon\ge 0$. 

Finally, we note that
\begin{align*}
\hat \epsilon = \pm \bar y^{-1} \bar \epsilon,
\end{align*}
in \eqsref{chartEpshat2}{chartEpshat}. This follows from the details of the actual charts themselves, see e.g. the general expressions \eqsref{chartPos}{chartNeg}. Therefore we obtain $\widetilde X$, see \eqref{Widetilde1}, in the charts \eqsref{chartEpshat2}{chartEpshat}, by simply dividing the local vector-fields by $\pm \hat \epsilon$, respectively. Here the local vector-fields are obtained from the extended vector-field \eqref{extended} by applying the substitutions:
\begin{align*}
 (y,\epsilon)\mapsto (y,\hat \epsilon)  = (y,y^{-1} \epsilon),\quad y\gtrless 0,
\end{align*}
given in \eqref{chartEpshat2} and \eqref{chartEpshat}. In contrast, we obtain, following (iii) above, a local version of $\widetilde X$ in chart $\bar \epsilon=1$ by simply applying the substitution:
\begin{align*}
 (y,\epsilon)\mapsto (\hat y,\epsilon)  = (\epsilon^{-1} y,\epsilon),\quad \epsilon \ge 0,
\end{align*}
given in \eqref{yhat}, to \eqref{extended}, without doing any subsequent division of the right hand side. 
(In practice, since we are only interested in orbits, without explicit reference to time, we therefore do not worry about the specifics of the function $\psi$ in \eqref{Widetilde2} that defines the actual time transformation to get to $\widetilde X$ from $\overline X$.)  See \secref{BarY1} for further details. 



\begin{figure}[h!] 
\begin{center}
{\includegraphics[width=.9\textwidth]{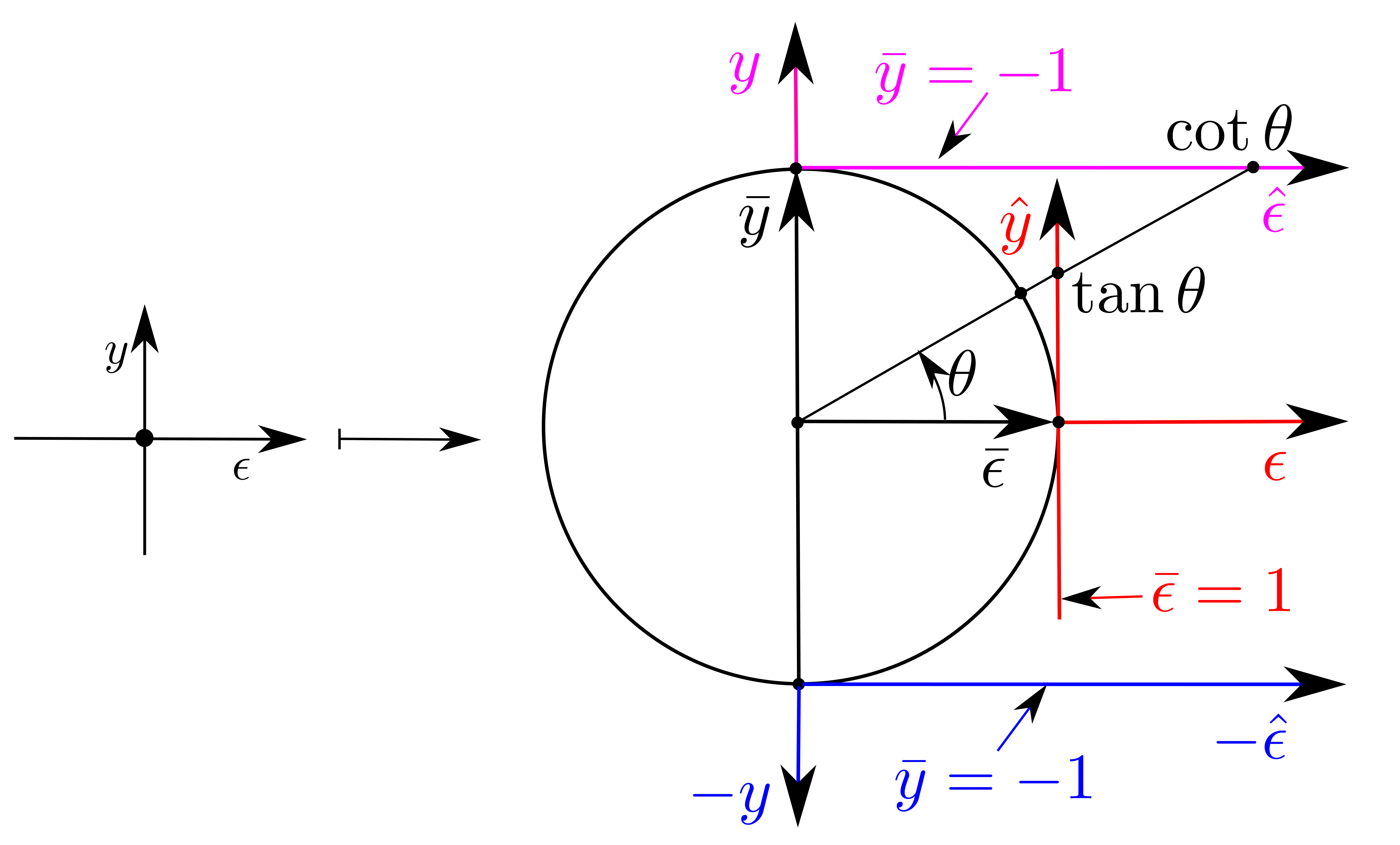}}
\end{center}
 \caption{Blowup of $y=\epsilon=0$ to the unit circle $\bar y^2+\bar \epsilon^2=1$. Charts $\bar \epsilon=1$ \eqref{yhat}, $\bar y=1$ \eqref{chartEpshat2} and $\bar y=-1$ \eqref{chartEpshat} in the blowup transformation \eqref{blowup0}. $\bar y\gtrless 0$ correspond to the south and the north semicircles, respectively. }
\figlab{InitialCharts}
\end{figure}

We follow the notation convention that all geometric objects obtained in any of these charts will be given a hat. We will often switch between charts using \eqref{hatEpshatY} but we believe it is clear from the context what variables are used. A geometric object obtained in the charts, say $\widehat M$, will be given a bar, $\overline M$, in the blowup variables $(x,z,\rho,(\bar y,\bar \epsilon))\in \mathbb R^2 \times [0,\infty)\times S^1$. Furthermore, an object, say $\widehat G$, obtained in $\bar \epsilon=1$ will often only be partially covered ($\hat y\gtrless 0$, respectively) by the charts $\bar y=\pm 1$. For simplicity, we will, however, continue to denote the subset of $\widehat G$ that is covered by the charts $\bar y=\pm 1$ by the same symbol.
Furthermore, by applying \eqref{blowup0Transformation} to an object $\overline M$, then we obtain a set, which denote by $M$, in the $(x,y,z,\epsilon)$-space. In this way, one says that $\overline M$ has been {\it blown down} to $M$.}

\section{Geometry of $\widetilde X$}\seclab{geometry}
\ed{
Now we present the $\epsilon=0$ dynamics of $\widetilde X$, obtained from the extended vector-field \eqref{extended}, with $X_\epsilon$ as in \eqref{Xeps}, using the truncated, PWL system $X^\pm(\textbf x)$ given in \bpwl{}, and the blowup and desingularization, see \eqref{blowup0}, \eqref{Widetilde1} and \eqref{Widetilde2}. 
\edNew{We recall that a smooth manifold of critical points is said to be normally hyperbolic if the linearization of any point only has as many eigenvalues with zero real part as the dimension of the tangent space. Similarly, we say that a critical point is partially hyperbolic if its linearization has hyperbolic directions. In contrast, we say that a set is fully nonhyperbolic if the linearization of any point only has eigenvalues with zero real part.} Then working in the charts $\bar \epsilon=1$ and $\bar y=\pm 1$, as $\epsilon\rightarrow 0^+$, we obtain the following result.
\begin{proposition}\proplab{WidetildeX}
 The following sets:
\begin{align*}
 \overline q:\quad &x=\rho=z=0,\,(\bar y,\bar \epsilon)\in S^1\cap \{\bar \epsilon \ge 0\},\\
 \overline l^+:\quad &x\in \mathbb R,\,\rho =z=0,\,(\bar y,\bar \epsilon)=(1,0)\in S^1,\\
 \overline l^-:\quad &z\in \mathbb R,\,x=\rho =0,\,(\bar y,\bar \epsilon)=(-1,0)\in S^1,
\end{align*}
given in the blowup variables $(x,z,\rho,(\bar y,\bar \epsilon))\in \mathbb R^2\times [0,\infty)\times S^1$, are sets of fully nonhyperbolic critical points of $\widetilde X$. The critical sets
\begin{align*}
 \overline M^+:\quad &(x,z)\in \mathbb R^2,\,\rho = 0,\,(\bar y,\bar \epsilon) = (1,0)\in S^1,\\
 \overline M^-:\quad &(x,z)\in \mathbb R^2,\,\rho = 0,\,(\bar y,\bar \epsilon) = (-1,0)\in S^1,
\end{align*}
are, on the other hand, normally hyperbolic for $x\ne 0$ and $z\ne 0$, respectively, each being of saddle-type. Also 
$$\overline M^+\cap \{z=0\}=\overline l^+, \quad \overline M^-\cap \{x=0\}=\overline l^-.$$

Furthermore, for  $(x,z)\in \breve \Sigma_{sl}^\pm$ the graphs
\begin{align*}
 \overline S_{a}:\quad & \bar \epsilon^{-1} \bar y=h(x^{-1}z),\,\quad \text{for}\quad (x,z)\in \breve \Sigma_{sl}^{-},\\
 \overline S_{r}:\quad & \bar \epsilon^{-1} \bar y=h(x^{-1}z),\,\quad \text{for}\quad (x,z)\in \breve \Sigma_{sl}^{+},
\end{align*}
are normally hyperbolic and attracting/repelling critical manifolds. The function $h:(0,\infty)\rightarrow \mathbb R$ is defined in \eqref{haFunction}.
These manifolds carry reduced, slow flows  which on $(x,z)$ coincide with the Filippov sliding flow on $\Sigma_{sl}^\pm$. 
\end{proposition}
As a corollary, $\overline S_{a}\cup \overline q \cup \overline S_{r} $ contains 
\begin{align}
 \bar \gamma^s:&\quad \bar \epsilon^{-1} \bar y = h(-\chi_-),\,(x,z) = \breve v_-s,\,s\in \mathbb R,\eqlab{barGammaWs}\\
 \bar \gamma^{w}:&\quad \bar \epsilon^{-1} \bar y = h(-\chi_+),\,(x,z) = \breve v_+s,\,s\in \mathbb R,\nonumber
\end{align}
recall \eqref{vpm} and \eqref{chipm}. 

We illustrate the geometry in \figref{Gamma0}. Upon blowing down (i.e. by applying the mapping \eqref{blowup0Transformation}) and returning to the original $(x,y,z)$-variables, $\overline q$, $\overline l^\pm$, $\overline S_{a}$, $\overline S_r$ collapse to the two-fold $q=(0,0,0)$, the fold lines $l^\pm$, and the stable and unstable sliding regions $\Sigma_{sl}^\mp$, respectively, for $\epsilon=0$. Furthermore, $\bar \gamma^s$ and $\bar \gamma^w$ in \eqref{barGammaWs} collapse to $\gamma^s$ and $\gamma^w$ in \eqsref{gammas}{gammaw}.

\begin{figure}[h!] 
\begin{center}
{\includegraphics[width=0.8\textwidth]{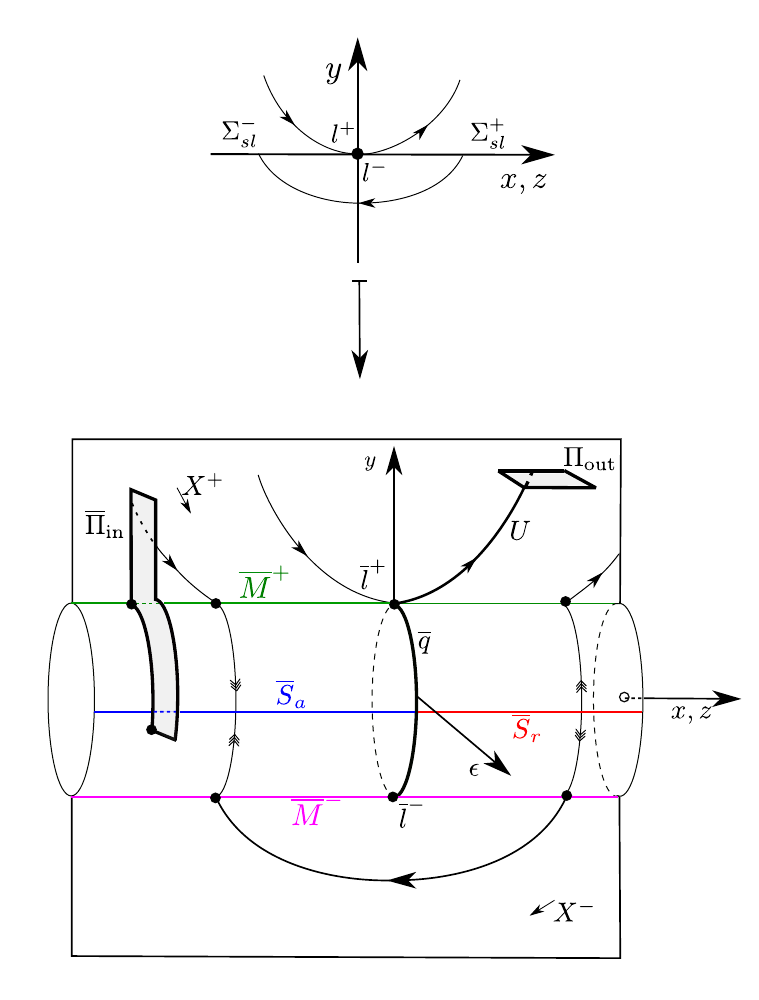}}
\end{center} 
 \caption{Schematic illustration of the blowup geometry: the nonhyperbolic critical points $\overline q$ and $\overline l^\pm$ of $\overline X$, the normally hyperbolic critical manifolds $\overline S_{a,r}$ and the blowup of the section $\Pi_{\text{\textnormal{in}},\epsilon}$ (see \secref{Piin}).}
\figlab{Gamma0}
\end{figure}

\propref{WidetildeX} is proven through a set of calculations that are easy to carry out in the three charts $\bar \epsilon=1$, see \eqref{yhat}, $\bar y=1$, see \eqref{chartEpshat2}, and $\bar y=-1$, see \eqref{chartEpshat}. 
We consider each of these charts in turn.}
 
\subsection{The chart $\bar \epsilon=1$: A slow-fast system}\seclab{barEps1}
Inserting $\hat y = \epsilon^{-1} y$, cf. \eqref{yhat}, into \brpwlf{} gives the following set of equations:
\begin{align}
x' &=\epsilon \left(\beta^{-1} c(1+\phi(\hat y))-(1-\phi(\hat y))\right),\eqlab{slowFastEquations1}\\
\hat y'&=b {z} (1+\phi(\hat y)) -\beta  x(1-\phi(\hat y)),\nonumber\\
 z'&=\epsilon\left((1+\phi(\hat y))+b^{-1} \gamma(1-\phi(\hat y))\right),\nonumber
\end{align}
in terms of the fast time $\tau$. Here $\hat y\in \mathbb R,\,\epsilon\ge 0$ and obviously $\epsilon'=0$. This is now a standard slow-fast system. The $\hat y$ variable is fast with $\mathcal O(1)$ velocities in general whereas $x$ and $z$ are slow variables with $\mathcal O(\epsilon)$ velocities. In slow-fast theory, system \eqref{slowFastEquations1} is called {\it the fast system}, whereas
\begin{align}
\dot x &=\beta^{-1} c(1+\phi(\hat y))-(1-\phi(\hat y)),\eqlab{slowFastEquations10}\\
\epsilon \dot{\hat y}&=b {z} (1+\phi(\hat y)) -\beta  x(1-\phi(\hat y)),\nonumber\\
 \dot z&=(1+\phi(\hat y))+b^{-1} \gamma(1-\phi(\hat y)),\nonumber
\end{align}
is called {\it the slow system}. 

The {\em layer problem} is given by the limiting system \eqref{slowFastEquations1}$_{\epsilon=0}$:
\begin{align}
x' &=0,\eqlab{layer}\\
\hat y'&=b {z} (1+\phi(\hat y)) -\beta  x(1-\phi(\hat y)),\nonumber\\
 z'&=0.\nonumber
\end{align}
Note that $x$ and $z$ are constant in \eqref{layer}. 

The {\em reduced problem} is given by the limiting system \eqref{slowFastEquations10}$_{\epsilon=0}$: 
\begin{align}
\dot x &=\beta^{-1} c(1+\phi(\hat y))-(1-\phi(\hat y)),\eqlab{reduced}\\
0&=b {z} (1+\phi(\hat y)) -\beta  x(1-\phi(\hat y)),\nonumber\\
 \dot z&=(1+\phi(\hat y))+b^{-1} \gamma(1-\phi(\hat y)).\nonumber
\end{align}
Note that $\hat y$ is slaved in \eqref{reduced}. 

In this chart, we have known results that are collected together in \propref{criticalManifold} as follows. 

\begin{proposition}\proplab{criticalManifold}
\cite[Theorem 5.1, Proposition 5.4]{krihog}
The critical manifold $$\widehat S_0 = \widehat S_a \cup \widehat S_r \cup \hat q$$ of \eqref{slowFastEquations1}$_{\epsilon=0}$ is a union of the smooth graphs:
 \begin{align}
 \widehat S_{a}:\quad &\hat y=h(x^{-1}z),\,\quad \text{for}\quad (x,z)\in \breve \Sigma_{sl}^{-},\eqlab{h0Function}\\
 \widehat S_{r}:\quad &\hat y=h(x^{-1}z),\,\quad \text{for}\quad (x,z)\in \breve \Sigma_{sl}^{+},\nonumber
 \end{align}
 and the line:
 \begin{align}
\hat q & = \{(x,\hat y,z)\in \mathbb R^3 \vert \hat y\in \mathbb R,\,x=0,\,z=0\}.\eqlab{hatq}
 \end{align}
  On $\widehat S_{a,r}$ the motion of the slow variables $x$ and $z$ is described by the reduced problem \eqref{reduced} which coincides with the sliding equations \eqref{slidingEqns}. Also $\widehat S_{a,r}$ are both normally hyperbolic, $\widehat S_a$ being attracting, $\widehat S_r$ being repelling while $\hat q$ is fully nonhyperbolic. 
\end{proposition}
\begin{proof}
 We obtain the set $\widehat S_0$ as the set of equilibria of \eqref{layer}. Hyperbolicity is determined by linearization of \eqref{layer}. The remainder of the proof then follows from simple calculations. In particular, the linearization about any point in $\hat q$ has only zero as an eigenvalue.
\end{proof}
\ed{We sketch the critical manifold $\widehat S_0$ in \figref{SaSrpNew}. By Fenichel's theory \cite{fen1,fen2}, we have the following:
\begin{proposition}\proplab{Fenichel}
Consider any compact submanifold (with boundary) of the normally hyperbolic critical manifold $\widehat S_{a}$ ($\widehat S_r$) in the $(x,\hat y,z)$-space. Then this invariant sub-manifold perturbs $\mathcal O(\epsilon)$-smoothly to a locally invariant \textnormal{slow} manifold $\widehat S_{a,\epsilon}$ ($\widehat S_{r,\epsilon}$) for $0<\epsilon \ll 1$ sufficiently small. $\widehat S_{a,\epsilon}$ ($\widehat S_{r,\epsilon}$) carries a slow flow that is smoothly $\mathcal O(\epsilon)$-close to the sliding equations. 
\end{proposition}
Note that we cannot control $\widehat S_{a,\epsilon}$ and $\widehat S_{r,\epsilon}$ close to $\hat q$ by Fenichel's theory. To do so we need to apply a separate blowup (see \eqref{blowup1} below). 

\begin{figure}[h!] 
\begin{center}
{\includegraphics[width=\textwidth]{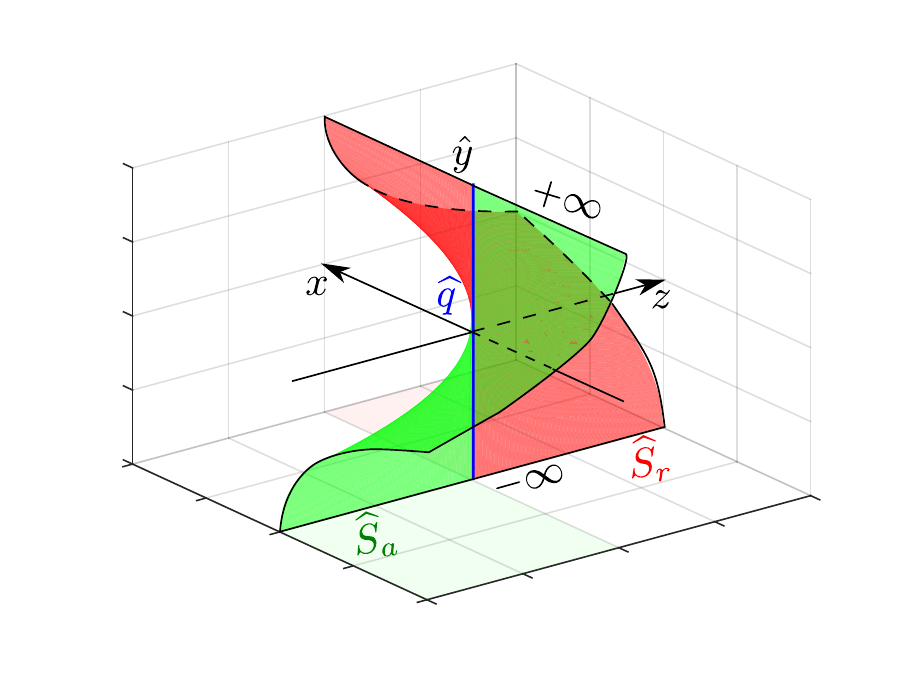}}
\end{center} 
 \caption{\ed{Sketch of the critical manifold $\widehat S_0$. The surfaces $\widehat S_{a}$ and $\widehat S_r$ are graphs over $\Sigma_{sl}^-$ and $\Sigma_{sl}^+$, respectively. But for fixed $z\ne 0$ and $x\rightarrow 0$ then $\hat y\rightarrow -\infty$ on each surface $\widehat S_{a}$ or $\widehat S_r$. Similarly, for fixed $x\ne 0$ and $z\rightarrow 0$ then $\hat y\rightarrow \infty$ on each surface $\widehat S_{a}$  or $\widehat S_r$. 
 We obtain the line $\hat q$ of fully nonhyperbolic critical points as the intersection of the closure of the two sub-manifolds $\widehat S_{a}$ and $\widehat S_r$}.}
\figlab{SaSrpNew}
\end{figure}


The objects $\hat q$, $\widehat S_{a}$ and $\widehat S_r$ above are all subsets of $(x,\hat y,z)$-space. But since the blowup \eqref{blowup0} works on the extended space $(x,\hat y,z,\epsilon)$, we will henceforth also view them as $\epsilon=0$ sections of this extended space $(x,\hat y,z,\epsilon)$. For example, we will have to follow this viewpoint in the charts $\bar y=\pm 1$. For simplicity, we will use the same symbol for the objects in the extended space (e.g. ``$\hat q=\hat q\times\{0\}$''). Notice also that the $2D$ slow manifolds $\widehat S_{a,\epsilon}$ and $\widehat S_{r,\epsilon}$ are $\epsilon=\text{const}.$ sections of $3D$ center-like manifold in the extended space $(x,\hat y,z,\epsilon)$ (see \lemmaref{lines} where we extend the slow manifolds through this viewpoint, following \cite{krupa_extending_2001}). }

As mentioned above, although $\hat q$, $\widehat S_{a}$ and $\widehat S_{r}$ are only partially covered by the charts $\bar y=\pm 1$ we will nevertheless denote the subset of the objects that are covered by the same symbols.

\subsection{The chart $\bar y=1$}\seclab{BarY1}
Inserting \eqref{chartEpshat2} into \eqref{extended}, we obtain the following equations, using \eqref{phipm}:
\begin{align}
x'&= \hat \epsilon y (\beta^{-1} c(1+\phi_+(\hat \epsilon))-(1-\phi_+(\hat \epsilon))),\nonumber\\
y'&= \hat \epsilon y (b {z} (1+\phi_+(\hat \epsilon)) -\beta  x(1-\phi_+(\hat \epsilon))),\nonumber\\
 z'&=\hat \epsilon y  ((1+\phi_+(\hat \epsilon ))+b^{-1} \gamma(1-\phi_+(\hat \epsilon ))),\nonumber\\
\hat \epsilon' & =-\hat \epsilon^2 (b {z} (1+\phi_+(\hat \epsilon)) -\beta  x(1-\phi_+(\hat \epsilon))),\nonumber
\end{align}
where $(y,\hat \epsilon) \in [0,\infty)^2$. \ed{Recall definition of $\phi_+$ in \eqref{phipm}. This system is the local form of $2\overline X$. Now, in agreement with \eqref{Widetilde1}, we see that $\hat \epsilon = \bar y^{-1}\bar \epsilon$ is a common factor of these equations. To study $\widetilde X$ in this chart, we therefore divide out this common factor by rescaling time to} obtain the following system
\begin{align}
x'&= y (\beta^{-1} c(1+\phi_+(\hat \epsilon))-(1-\phi_+(\hat \epsilon))),\eqlab{xyzhatEps0Pos}\\
y'&= y (b {z} (1+\phi_+(\hat \epsilon)) -\beta  x(1-\phi_+(\hat \epsilon))),\nonumber\\
 z'&=y  ((1+\phi_+(\hat \epsilon ))+b^{-1} \gamma(1-\phi_+(\hat \epsilon ))),\nonumber\\
\hat \epsilon' & =-\hat \epsilon (b {z} (1+\phi_+(\hat \epsilon)) -\beta  x(1-\phi_+(\hat \epsilon))),\nonumber
\end{align}
that we study in the sequel.  
\begin{remark}
The two sets $\{y=0\}$ and $\{\hat \epsilon=0\}$ are each invariant for \eqref{xyzhatEps0Pos}.
Within $\{\hat \epsilon=0\}$ we recover the vector-field $X^+$ of \bpwl{}$_{y>0}$ from  \eqref{xyzhatEps0Pos}:
\begin{align}
 \dot x &=\beta^{-1} c,\eqlab{XPnew}\\
 \dot y&=bz,\nonumber\\
 \dot z &=1,\nonumber
\end{align}
after further division of the right hand side by $2y> 0$, using $\phi_+(0)=1$.

Within $\{y=0\}$ we have $x'=z'=0$ and
\begin{align*}
 \hat \epsilon' & =-\hat \epsilon (b {z} (1+\phi_+(\hat \epsilon)) -\beta  x(1-\phi_+(\hat \epsilon))),
\end{align*}
in agreement with the layer problem \eqref{layer}, using \eqref{hatEpshatY} and \eqref{phipm}. 
In particular, the set defined by $b {z} (1+\phi_+(\hat \epsilon)) -\beta  x(1-\phi_+(\hat \epsilon))=0$ coincides with $\widehat S_{0}\cap \{\hat y>0\}$ of \propref{criticalManifold} under the coordinate transformation \eqref{hatEpshatY}, having the same hyperbolicity properties. Here 
\begin{align}
 \hat q:\,x=z=y=0,\, \hat \epsilon\ge 0,\eqlab{newq}
\end{align}
 is a line of critical points of the layer problem \eqref{layer}. The smooth graphs $\widehat S_{a,r}$ are now given by
\begin{align*}
 \widehat S_{a,r}:\, \hat \epsilon = h_+(x^{-1}z),\,\quad (x,y,z)\in \Sigma_{sl}^\pm \cap \{x^{-1}z \in (0,b^{-1} \beta)\},
\end{align*}
respectively, 
where $h_+:(0,b^{-1}\beta)\rightarrow [0,\infty)$ is defined by
\begin{align*}
  h_+(s)= \phi_+^{-1}\left(\frac{1-\beta^{-1} b s}{1+\beta^{-1} b s}\right).
\end{align*}
Notice that 
\begin{align*}
 h_+(0) =0,\quad h_+'(0) = \beta^{-1}\pi b.
\end{align*}
\end{remark}

In the chart $\bar y=1$, along the intersection $\{y=\hat \epsilon=0\}$ of the invariant sets $\{y=0\}$ and $\{\hat \epsilon=0\}$, we obtain the following.
\begin{lemma}\lemmalab{hatMP}
 The set $\widehat M^+ \equiv \{y=\hat \epsilon =0\}$ is a set of critical points of \eqref{xyzhatEps0Pos}. It is of saddle-type for $z\ne 0$: The linearization about any point in $\widehat M^+$ has only two non-trivial eigenvalues $\pm 2bz$ with associated eigenvectors
 \begin{align*}
  \begin{pmatrix}
   \beta^{-1} c\\ 
bz\\ 
1\\
0
  \end{pmatrix},\begin{pmatrix}
0\\
0\\
0\\
1
  \end{pmatrix},
 \end{align*}
respectively. The line 
\begin{align}
\hat l^+=\widehat M^+\cap \{z=0\},\eqlab{hatlPlus}
\end{align}
is fully nonhyperbolic: The linearization about any point in $\hat l^+$ has only zero eigenvalues. It becomes $l^+$ upon returning to the $(x,y,z)$-variables.
\end{lemma}
\begin{proof}
 Simple calculations.
\end{proof}

The details for $\bar y=-1$ are similar to those in $\bar y=1$ and are therefore postponed to \appref{appHALLO}. This completes the proof of \propref{WidetildeX}.

\subsection{The blowup of the section $\Pi_{\text{\textnormal{in}},\epsilon}$}\seclab{Piin}
\ed{

Now, as with $\hat q$, $\widehat S_a$, etc. above in the chart $\bar \epsilon=1$, the plane $\Pi_{\text{\textnormal{in}},\epsilon}$ becomes a $\epsilon=\text{const}.\in (0,\epsilon_0]$ section of 
\begin{align*}
 \Pi_{\textnormal{in}} \equiv \{(x,y,z,\epsilon)\in \mathbb R^4 \vert (x,y,z)\in \Pi_{\text{in},\epsilon},\,\epsilon \in (0,\epsilon_0]\},
\end{align*}
with $0<\epsilon_0\ll 1$. Upon blowup \eqref{blowup0}, $\Pi_{\text{in}}$ becomes $\overline{\Pi}_{\textnormal{in}} \subset \overline P$ which in \figref{Gamma0} extends from $(\bar y,\bar \epsilon) = (1,0)$ down to include a small neighborhood of the critical manifold at $\bar \epsilon^{-1} \bar y = h(-\delta^{-1} z)$ within $x=-\delta$. Notice that this is always possible by taking $\omega$ sufficiently close to $\omega_0$ in $R_{\text{in},\epsilon}$ since $\dot x>0$ on $\overline S_{a}$ for $x=-\delta$ and $z\in I_{\text{in}}$, see \remref{xdotNeg}.} 
\section{Proof of \thmref{mainThm}}\seclab{proof}

 

We shall now prove \thmref{mainThm}. We work with the desingularized vector-field $\widetilde X$ in the phase space 
\begin{align*}
 \overline P = \left\{(x,z,\rho,(\bar y,\bar \epsilon))\in\mathbb R^2 \times [0,\infty)\times S^1\right\}.
\end{align*}
Here
\begin{align}
 \bar q=\left\{x=z=\rho=0,\,(\bar y,\bar \epsilon)\in S^1\right\},\eqlab{barp}
\end{align}
is the circle of nonhyperbolic critical points of $\widetilde X$ illustrated in \figref{Gamma0}. Recall also \propref{WidetildeX}. Now, we perform a further blowup of this circle by applying the method of Dumortier and Roussarie \cite{dumortier_1991,dumortier_1993,dumortier_1996} (see also \cite{gomory1955a,takens1974a}),  in the formulation of Krupa and Szmolyan \cite{krupa_extending_2001, krupa_extending2_2001} for singular perturbed problems. 
\subsection{Blowup of $\bar q$}
We apply the following blowup transformation 
\begin{align}
(x,\rho,z,(\bar y,\bar \epsilon))\mapsto (r,(\bar x,\bar \rho,\bar z),(\bar y,\bar \epsilon)),\eqlab{blowupTransformation}
\end{align}
defined by
\begin{align}
 x = r \bar x,\,\rho = r^2 \bar \rho,\,z=r\bar z,\quad r\ge 0, (\bar x,\bar \rho,\bar z)\in S^2=\{\bar x^2+\bar \rho^2+\bar z^2=1\}.\eqlab{blowup1}
\end{align} 
Putting this together with \eqref{blowup0} we have
\begin{align*}
 (r,(\bar x,\bar \rho,\bar z),(\bar y,\bar \epsilon)) \mapsto (x,y,z,\epsilon),
\end{align*}
defined by
\begin{align*}
 x = r\bar x,\,y=r^2 \bar \rho \bar y,\,z=r\bar z,\,\epsilon = r^2 \bar \rho\bar \epsilon,\quad r \ge 0,\,(\bar x,\bar \rho,\bar z)\in S^2,\,(\bar y,\bar \epsilon)\in S^1.
\end{align*}
The transformation \eqref{blowupTransformation} blows up $\bar q$, the circle of nonhyperbolic critical points \eqref{barp}, to $ \overline{\overline q}$, a circle of spheres $S^1\times S^2$: 
\begin{align}\eqlab{barbarp}
 \overline{\overline q}:\quad r=0,\,(\bar x,\bar \rho,\bar z)\in S^2,(\bar y,\bar \epsilon)\in S^1,
\end{align}
as illustrated in \figref{Gamma0barq}. The double-bar indicates that the two-fold $q$ has been blown up twice. 
Henceforth we consider the following phase space
\begin{align*}
 \overline{\overline P}=\left\{(r,(\bar x,\bar \rho,\bar z),(\bar y,\bar \epsilon))\in [0,\infty)\times S^2 \times S^1\right\}.
\end{align*}
\eqref{blowupTransformation} {\em pulls back} $\widetilde X$ on $\overline P$ to a vector-field $\overline{\widetilde X}$ on $\overline{\overline{P}}$. Here $\overline{\widetilde{X}}\vert_{r=0}=0$. However, the exponents of $r$ in \eqref{blowup1} have been chosen so that 
\begin{align*}
 \widetilde{\widetilde {X}} = r^{-1}\overline{\widetilde{X}},
\end{align*}
is well-defined and non-trivial, following \cite{krupa_extending_2001}.  It is $\widetilde{\widetilde X}$ that we study in the sequel. 

\begin{figure}[h!] 
\begin{center}
{\includegraphics[width=.8\textwidth]{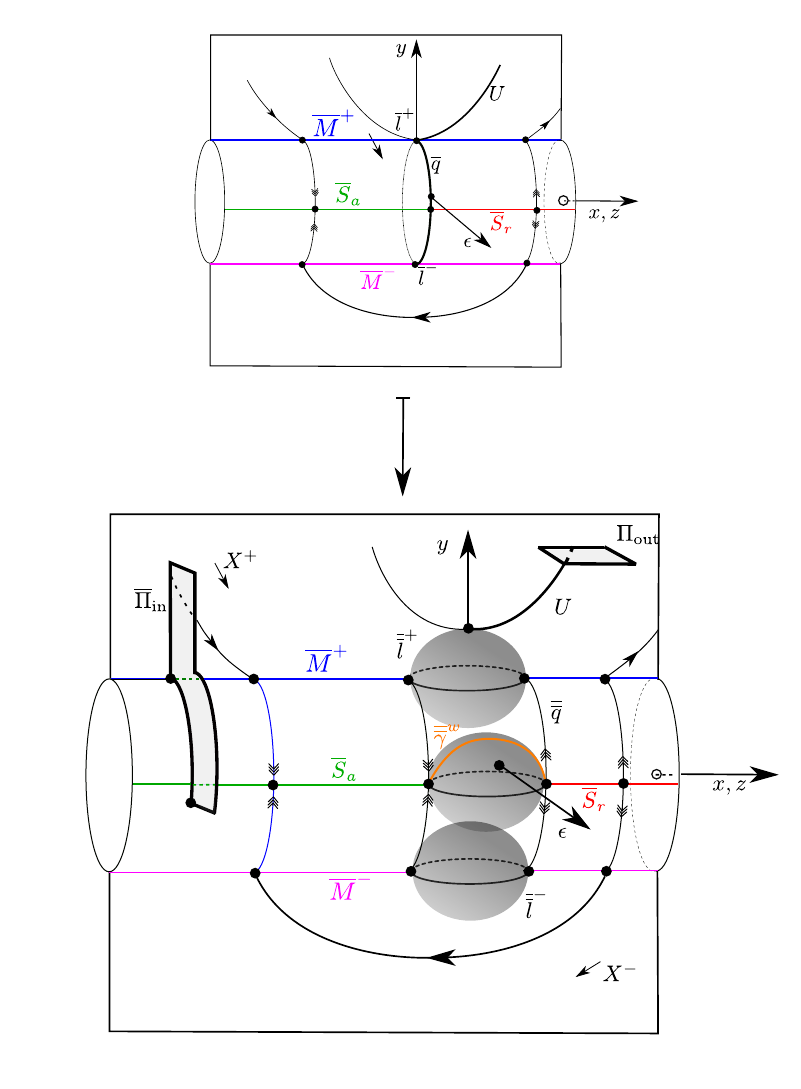}}
\end{center}
 \caption{Schematic illustration of the blowup of $\bar q$, a circle of nonhyperbolic critical points defined in \eqref{barp}, to $\overline{\overline q}$, a circle of spheres $S^1\times S^2$ defined in \eqref{barbarp}.}
\figlab{Gamma0barq}
\end{figure}

To describe $\overline{\overline{P}}$ and $\widetilde{\widetilde{X}}$ we use the charts $\bar \epsilon =1$ \eqref{yhat}, $\bar y = 1$ \eqref{chartEpshat2} and $\bar y=-1$  \eqref{chartEpshat}. This covers the relevant part of the circle $(\bar y,\bar \epsilon)\in S^1$ with $\bar \epsilon\ge 0$. Notice that in each of the charts $\bar \epsilon =1$, $\bar y = 1$ and $\bar y=-1$, we obtain subsets of $\overline{\overline q}$ as a cylinder of spheres $\overline{\widehat q}$, see also \figref{barEpsEq1Charts}. In $\bar \epsilon=1$ this cylinder $\overline{\widehat q}$ is double infinite $\hat y\in \mathbb R$. Only $\bar y=\pm 1$ is not covered in this chart. In each of the charts $\bar y=\pm 1$, $\overline{\widehat q}$ is infinite $\hat \epsilon \in [0,\infty)$ and $\hat \epsilon \in (-\infty,0]$, respectively. To parametrise and cover the relevant part of the sphere $(\bar x,\bar \rho,\bar \epsilon)\in S^2$, we use the following directional charts
\begin{align*}
 \kappa_1:\quad (r_1,\rho_1,z_1)\mapsto \left\{\begin{array} {cc} x &= -r_1\\
 \rho &= r_1^2 \rho_1\\
 z&=r_1 z_1
 \end{array}\right.,\\
 \kappa_2:\quad  (r_2,x_2,z_2) \mapsto \left\{\begin{array} {cc}  x &= r_2x_2\\
 \rho &= r_2^2\\
 z&=r_2 z_2
 \end{array}\right.,\\
  \kappa_3:\quad (r_3,\rho_3,z_3)\mapsto \left\{\begin{array} {cc} x &= r_3\\
 \rho &= r_3^2 \rho_3\\
 z&=r_3 z_3
 \end{array}\right.,
\end{align*}
for $r_i\ge 0$,
obtained by setting $\bar x=-1$, $\bar \rho=1$, and $\bar x=1$, respectively, in \eqref{blowup1}, as suggested by \cite{krupa_extending_2001}.

Within chart $\bar \epsilon =1$ \eqref{yhat} where $y=\epsilon \hat y,\,\epsilon\ge 0,\,\hat y\in \mathbb R$ the charts $\kappa_i$ become 
\begin{align}
 (\bar \epsilon=1,\kappa_1):\quad  (r_1,\hat y,z_1,\epsilon_1) &\mapsto \left\{ \begin{array}{cc} x &= -r_1 \\
                                                                               y &=r_1^2\epsilon_1 \hat y\\
                                                                               z &=r_1 z_1\\
                                                                               \epsilon &=r_1^2\epsilon_1
                                                                              \end{array}\right.,\quad r_1\ge 0,\,\epsilon_1\ge 0,\eqlab{kappa1} \\
 (\bar \epsilon=1,\kappa_2):\quad (r_2,x_2,\hat y,z_2)&\mapsto \left\{ \begin{array}{cc}  
 x &= r_2x_2\\
 y &= r_2^2 \hat y\\
 z &=r_2 z_2\\
 \epsilon &= r_2^2
 \end{array}\right.,\quad r_2\ge 0, \eqlab{kappa2}\\
 (\bar \epsilon=1,\kappa_3):\quad (r_3,\hat y,z_3,\epsilon_3) &\mapsto
 \left\{ \begin{array}{cc}  x&= r_3\\
          y &=r_3^2\epsilon_3 \hat y\\
          z &=r_3 z_3\\
          \epsilon &=r_3^2\epsilon_3,
         \end{array}\right.,\quad r_3\ge 0,\,\epsilon_3\ge 0.\eqlab{kappa3}
\end{align}
We refer to these charts as $(\bar \epsilon=1,\kappa_i)$, $i=1,2,3$. In \cite{krihog}, following the general terminology in \cite{krupa_extending_2001}, we referred to these charts as the {\it entry chart}, the {\it scaling chart} (since $r_2=\sqrt{\epsilon}$), and the {\it exit chart}, respectively.  When the charts overlap we can change coordinates as follows:
\begin{align}
 (\bar \epsilon=1,\kappa_1) &\rightarrow (\bar \epsilon=1,\kappa_2):\quad x_2=-1/\sqrt{\epsilon_1},\,r_2=r_1\sqrt{\epsilon_1},\,z_2=z_1/\sqrt{\epsilon_1},\eqlab{kappa21}\\
 (\bar \epsilon=1,\kappa_2)&\rightarrow(\bar \epsilon=1,\kappa_3):\quad r_3 = r_2 x_2,\,z_3 = x_2^{-1} z_2,\,\epsilon_3 = x_2^{-2},\eqlab{kappa32}
\end{align}
defined for $\epsilon_1>0$ and $x_2>0$, respectively. Their inverses can easily be computed from these expressions. Notice that $(\bar \epsilon=1,\kappa_1)$ and $(\bar \epsilon=1,\kappa_3)$ cannot overlap.

Similarly in charts $\bar y = \pm1 $, given by \eqsref{chartEpshat2}{chartEpshat} where $\epsilon = y\hat \epsilon$, $(y,\hat \epsilon)\in \overline{\mathbb R}_{\pm}$, respectively, we obtain
\begin{align}
 (\bar y=\pm 1,\kappa_1):\quad (r_1,y_1,z_1,\hat \epsilon) &\mapsto \left\{\begin{array}{cc} x &= -r_1\\
                                                                            y&=r_1^2y_1\\
                                                                            z&=r_1z_1\\
                                                                            \epsilon &=r_1^2y_1 \hat \epsilon
                                                                           \end{array}\right.,\quad r_1 \ge 0,\,y_1\gtreqless 0,\eqlab{K1New}\\
 (\bar y=\pm 1,\kappa_2):\quad (r_2,x_2,z_2,\hat \epsilon,) &\mapsto \left\{\begin{array}{cc} x &= r_2x_2\\
                                                                            y&=\pm r_2^2\\
                                                                            z &=r_2z_2\\
                                                                            \epsilon &=\pm r_2^2 \hat \epsilon,
                                                                           \end{array}\right.,\quad r_2 \ge 0,\eqlab{K2New}\\
 (\bar y=\pm 1,\kappa_3):\quad (r_2,x_2,z_2,\hat \epsilon) &\mapsto \left\{\begin{array}{cc} x &= r_3\\ 
                                                                            y &= r_3^2y_3 \\ z&=r_3 z_3\\ \epsilon & = r_3^2y_3 \hat \epsilon
                                                                           \end{array}\right., \quad r_3\ge 0,\,y_3\gtreqless 0\eqlab{K3New}
\end{align}
referring to these chart as $(\bar y=\pm 1,\kappa_i)$, $i=1,2,3$, henceforth\footnote{Unfortunately, the coordinates $r_2,\,x_2,\,z_2$ in \eqref{K2New} are different from the same symbols used in \eqref{kappa2}. But since we never need to go from $(\bar \epsilon=1,\kappa_2)$ to $(\bar y=\pm 1,\kappa_2)$ we believe confusion should not occur.}. 
The coordinate changes between the charts 
are given by
\begin{align}
(\bar y=\pm 1,\kappa_1) &\rightarrow (\bar y=\pm 1,\kappa_2) :\quad x_2=-1/\sqrt{\pm y_1},\,r_2=r_1\sqrt{\pm y_1},\,z_2=z_1/\sqrt{\pm y_1},\eqlab{kappa21New}\\
 (\bar y=\pm 1,\kappa_2) &\rightarrow (\bar y=\pm 1,\kappa_3):\quad r_3 = r_2 x_2,\,z_3 = x_2^{-1} z_2,\, y_3 = \pm x_2^{-2},\eqlab{kappa32New}
\end{align}
defined for $y_1\gtrless 0$ and $x_2>0$, respectively, and their inverses. 

Notice that the coordinate change from $(\bar \epsilon=1,\kappa_i)$ to $(\bar y=\pm 1,\kappa_i)$, for either $i=1$ or $i=3$, is obtained from
\begin{align}
y_1 = \hat y\epsilon_1,\quad y_3 = \hat y\epsilon_3,\quad \hat \epsilon = \hat y^{-1}.\eqlab{cc}
\end{align}
Fixing appropriate compact subsets of the coordinates in each chart, these charts then completely cover the relevant part of the cylinder $\overline{\overline q}\cap \{\bar \rho\ge 0,\,\bar \epsilon\ge 0,\,\bar x\ne 0\}$. 

We illustrate the $(\bar \epsilon=1,\kappa_i)$-charts and the $(\bar y=1,\kappa_i)$-charts in \figref{barEpsEq1Charts} (a) and (b), respectively.  The $(\bar y=-1,\kappa_i)$ chart is just a reflection of  the $(\bar y=1,\kappa_i)$ chart.
\begin{figure}[h!] 
\begin{center}
\subfigure[]{\includegraphics[width=.495\textwidth]{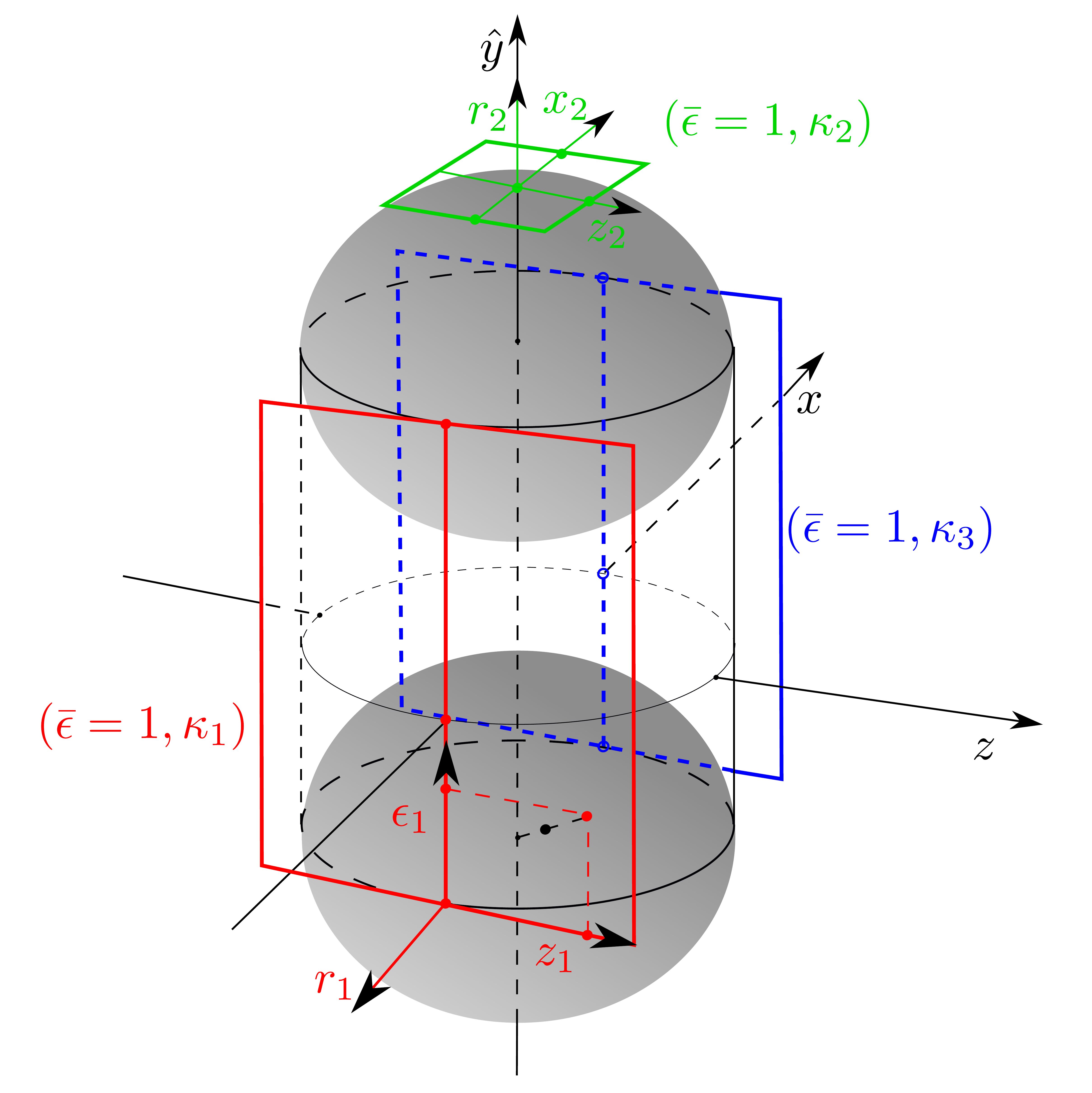}}
\subfigure[]{\includegraphics[width=.495\textwidth]{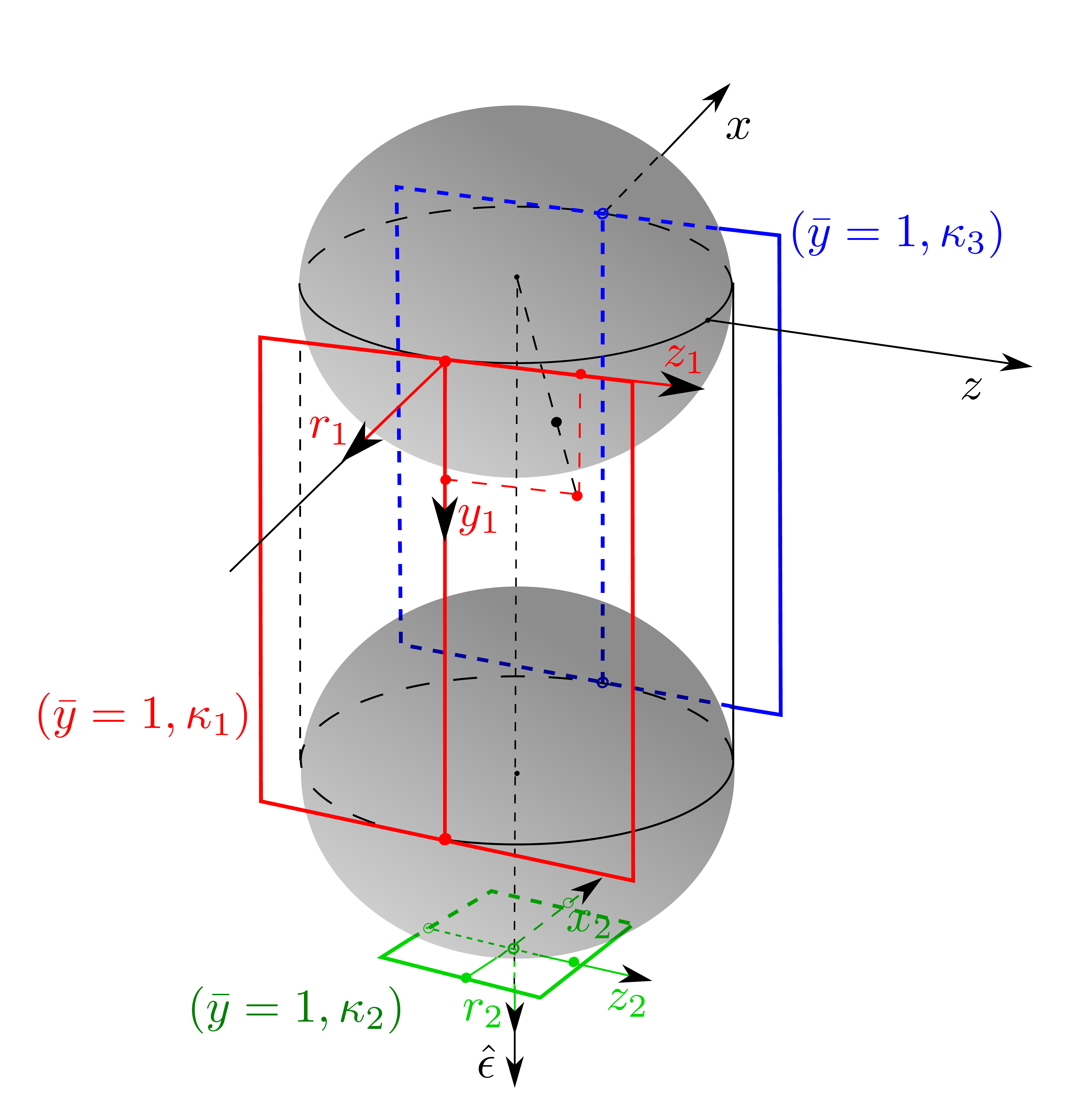}}
\end{center}
 \caption{\ed{(a): Illustration of the charts $(\bar \epsilon=1,\kappa_i)$, see \eqref{kappa1}, \eqref{kappa2} and \eqref{kappa3} and the associated coordinates. The \textit{entry chart} $(\bar \epsilon=1,\kappa_1)$, with coordinates $(\hat y,r_1,z_1,\epsilon_1)$, covers $\bar x<0$ of the spheres $\{\hat y=\text{const}.\in \mathbb R\}$ whereas the \text{scaling chart} $(\bar \epsilon=1,\kappa_2)$, with the coordinates $(x_2,\hat y,z_2,r_2)$, and the \text{exit chart} $(\bar \epsilon=1,\kappa_3)$, with the coordinates $(\hat y,r_3,z_3,\epsilon_3)$ (not shown but similar to $(r_1,\hat y,z_1,\epsilon_1)$), cover $\bar \epsilon>0$ and $\bar x>0$, respectively. (b): Illustration of the charts $(\bar y=1,\kappa_i)$, see \eqref{K1New}, \eqref{K2New} and \eqref{K3New} and the associated coordinates. The chart $(\bar y=1,\kappa_1)$, with coordinates $(r_1,y_1,z_1,\hat \epsilon)$, covers $\bar x<0$ of the spheres $\{\hat \epsilon=\text{const}.\in [0,\infty)\}$ whereas the chart $(\bar \epsilon=1,\kappa_2)$, with coordinates $(r_2,x_2,z_2,\hat \epsilon)$, and the chart $(\bar \epsilon=1,\kappa_3)$, with coordinates $(r_3,y_3,z_3,\hat \epsilon)$ (not shown but similar to $(r_1,y_1,z_1,\hat \epsilon)$),  cover $\bar \epsilon>0$ and $\bar x>0$, respectively}. } 
\figlab{barEpsEq1Charts}
\end{figure}

Note that we follow the (standard) convention that a geometric object obtained in $(\bar \epsilon=1,\kappa_i)$ or $(\bar y=\pm 1,\kappa_i)$ will be given a hat and a subscript $i$. Such an object, say $\widehat G_i$, will be denoted by $\overline{\widehat G}$ in the blowup variables \eqref{blowup1} of either of the charts $\bar \epsilon=1,\,\bar y=\pm 1$ used to cover $(\bar y,\bar \epsilon)\in S^1$ (each of these spaces are illustrated in \figref{barEpsEq1Charts}). In the full blowup space $\overline{\overline P}$, this object will be denoted by $\overline{\overline G}$. As above, objects, say $\widehat G_i$, obtained in $(\bar \epsilon=1,\kappa_i)$ will frequently only be partially covered ($\hat y\gtrless 0$, respectively) by the charts $(\bar y=\pm 1,\kappa_i)$. For simplicity, we will, however, (again) continue to denote the subset of $\widehat G_i$ that is covered by the charts $(\bar y=\pm 1,\kappa_i)$ using the same symbol.

\ed{We now present the following result (compare with \propref{WidetildeX}):
\begin{proposition}\proplab{WidetildeWidetildeX}
The following sets:
\begin{align*}
 \overline {\overline S}_a:\quad \bar \epsilon^{-1} \bar y = h(\bar x^{-1}\bar z),\,\bar \rho = 0,\,r\ge 0,\,\bar x<0,\,\bar z<0,\\
 \overline {\overline S}_r:\quad \bar \epsilon^{-1} \bar y = h(\bar x^{-1}\bar z),\,\bar \rho = 0,\,r\ge 0,\,\bar x>0,\,\bar z>0,
\end{align*}
obtained in the blowup variables $(x,z,\rho,(\bar y,\bar \epsilon))\in \mathbb R^2\times [0,\infty)\times S^1$, are normally hyperbolic (even for $r\ge 0$) critical manifolds of $\widetilde{\widetilde X}$. In particular, the sets on $\overline{\overline q}$:
\begin{align*}
 \overline{\overline L}_a&=  \overline {\overline S}_a\cap \{r=0\},\\
 \overline{\overline L}_r&=  \overline {\overline S}_r\cap \{r=0\},
\end{align*}
are partially hyperbolic: The linearization about any point in $\overline{\overline L}_a$ ($\overline{\overline L}_r$) has one single non-zero (negative/positive, respectively) eigenvalue.

Similarly,
\begin{align*}
\overline {\overline M}^+:\quad &(\bar y,\bar \epsilon) = (1,0),\,(\bar x,\bar z)\in S^1,\,r\ge 0,\,\bar \rho=0,\\
\overline {\overline M}^-:\quad &(\bar y,\bar \epsilon) = (-1,0),\,(\bar x,\bar z)\in S^1,\,r\ge 0,\,\bar \rho=0,
\end{align*}
are sets of normally hyperbolic (even for $r=0$) critical points for $\bar x\ne 0$ and $\bar z\ne 0$, respectively, each being of saddle-type. However, 
\begin{align*}
\overline {\overline l}^+ &= \overline {\overline M}^+\cap \{\bar z=0\}, \\
\overline {\overline l}^- &=\overline {\overline M}^-\cap \{\bar z=0\},
\end{align*}
are sets of fully nonhyperbolic critical points. 
\end{proposition}
\begin{proof}
The result follows from computations done in the charts below. 
\end{proof}}

On the blowup of $\bar q$, we therefore have (through the partial hyperbolicity of $\overline{\overline L}_{a,r}$) gained hyperbolicity for the vector field $\widetilde{\widetilde X}$. In \cite{krihog} we used this to extend the Fenichel's slow manifold using center manifold theory. We review these results in the following section.

\subsection{Slow manifolds and results from \cite{krihog}}\seclab{sec:slow}
\ed{We now use the blowup \eqref{blowup1} to extend the slow manifold $\widehat S_{a,\epsilon}$ ($\widehat S_{r,\epsilon}$) in \propref{Fenichel} up close to the fold. Technically we do this by working in the chart $(\bar \epsilon=1,\kappa_1)$ ($(\bar \epsilon=1,\kappa_3)$). But to ease the comparison with \propref{Fenichel}, we blow down the result and present the following Lemma in the $(x,\hat y,z)$-variables.
\begin{lemma}\lemmalab{lines}
Consider $\mu>0$ sufficiently small. Then there exist 
unique slow manifolds $\widehat S_{a,\epsilon}$ and $\widehat S_{r,\epsilon}$ in the $(x,\hat y,z)$-space that extend as perturbations of $\widehat S_{a}$ and $\widehat S_r$ up to $x = -\mu^{-1}\sqrt{\epsilon}$ and $x = \mu^{-1}\sqrt{\epsilon}$, respectively, for all $0<\epsilon\ll 1$, in the following way: 

Let $I_a\subset \mathbb R_-,\,I_r\subset \mathbb R_+$ be suitably large intervals and fix $k>0$. 
Then, $\widehat S_{a,\epsilon}$ is the section $\epsilon=\textnormal{const}.$ of the following, locally invariant, $3D$-surface in the extended space $(x,\hat y,z,\epsilon)$-space obtained as the image of the embedding: 
\begin{align}
[0,k] \times [0,\mu^2] \times I_a\ni  (r_1,\epsilon_1,z_1)\mapsto \left\{\begin{array}{cc} 
x &\hspace{-6.75cm}{=-r_1}\\
   \hat y &= h(-z_1)+\epsilon_1 b (\chi_+-z_1)(\chi_--z_1) m(z_1,\epsilon_1)\\
  z & \hspace{-6.75cm}{=r_1 z_1}\\
  \epsilon & \hspace{-6.75cm}{=r_1^2 \epsilon_1}
  \end{array}\right.,
  \eqlab{Saeps1}
 \end{align}
with $m(\cdot , \cdot)$ smooth.

Similarly, $\widehat S_{r,\epsilon}$ is the section $\epsilon=\textnormal{const}.$ of the following, locally invariant, $3D$-surface in the $(x,\hat y,z,\epsilon)$-space
obtained as the image of the embedding: 
\begin{align}
[0,k] \times [0,\mu^2] \times I_r\ni  (r_3,\epsilon_3,z_3)\mapsto \left\{\begin{array}{cc} 
x &\hspace{-6.75cm}{=r_3}\\
   \hat y &= h(z_3)+\epsilon_1 b (\chi_++z_3)(\chi_-+z_3) m(-z_3,\epsilon_1)\\
  z & \hspace{-6.75cm}{=r_3 z_3}\\
  \epsilon & \hspace{-6.75cm}{=r_3^2 \epsilon_3}
  \end{array}\right..
  \eqlab{Sreps1}
 \end{align}

Let $\widehat S_{a,\epsilon}^*$ ($\widehat S_{r,\epsilon}^{-*}$) denote the forward (backward) flow of $\widehat S_{a,\epsilon}$ ($\widehat S_{r,\epsilon}$). Then $\widehat S_{a,\epsilon}^*$ and $\widehat S_{r,\epsilon}^{-*}$ intersect transversally along the invariant line:
\begin{align}
 \hat{\gamma}^s:&\quad \hat y = h(-\chi_-),\, (x,z) = \breve v_-s,\,s\in \mathbb R.\eqlab{hatgammas}
\end{align}
The intersection is also transverse along the invariant line:
\begin{align}
 \hat{\gamma}^w:&\quad \hat y = h(-\chi_+),\, (x,z) = \breve v_+s,\,s\in \mathbb R,\eqlab{hatgammaw}           
\end{align}
if and only if (B) (see \eqref{nonResonance2}) holds.  


\end{lemma}
\begin{remark}\remlab{weakStrong}
We refer to $\hat \gamma^w$ in \eqref{hatgammaw} ($\hat \gamma^s$ in \eqref{hatgammas}) as the \textit{weak} (\textit{strong}) canard since its $xz$-projection is tangent, at the two-fold, to the weak (strong) eigendirection $\breve v_+$ ($\breve v_-$) associated with the desingularized sliding equations, see \eqref{slidingVectorField}. See also \cite{szmolyan_canards_2001}. 
\end{remark}
\begin{proof}
 The extension of Fenichel's slow manifold $\widehat S_{a,\epsilon}$ follows from \cite[Proposition 7.4]{krihog} where we use center manifold theory in the $(\bar \epsilon=1,\kappa_1)$-chart. For completeness, we include some details here.
 
 First, we insert \eqref{kappa1} into \eqref{slowFastEquations1} and divide the right hand side by the common factor $r_1$. This gives the following local form of $2\widetilde{\widetilde X}$:
\begin{align}
 \dot r_1 &=-r_1\epsilon_1 G(\hat y),\eqlab{BarEps1Kappa1Eqns}\\
 \dot{\hat y} &=bz_1(1+\phi(\hat y))+\beta (1-\phi(\hat y)),\nonumber\\
 \dot z_1 &=\epsilon_1 (H(\hat y)+z_1 G(\hat y)),\nonumber\\
 \dot \epsilon_1 &= 2\epsilon_1^2 G(\hat y).\nonumber
\end{align}
Here we have introduced the functions 
\begin{align}
G(\hat y) &= \beta^{-1} c(1+\phi(\hat y))- (1-\phi(\hat y)),\quad H(\hat y) = (1+\phi(\hat y))+b^{-1}\gamma (1-\phi(\hat y)).\eqlab{GHFunctions}
\end{align}
The line
\begin{align}
 \widehat L_{a,1}:\quad r_1=\epsilon_1=0,\,\hat y=h(-z_1),z_1\in \mathbb R_-,\eqlab{La1}
\end{align}
(which is the local version of $\overline{\overline L}_a$ in \propref{WidetildeWidetildeX}), 
is a set of partially hyperbolic critical points of \eqref{BarEps1Kappa1Eqns}: The linearization about any point $\widehat L_{a,1}$ gives three zero eigenvalues and one negative $bz_1\phi'-\beta \phi'$. Now, let $\mu$, $k$ and $I_a$ be as in \lemmaref{lines}. Then by the partial hyperbolicity of $\widehat L_{a,1}\vert_{z_1\in I_a}$, we obtain a unique center manifold \cite{car1}
\begin{align*}
 W^c(\widehat L_{a,1}\vert_{z_1\in I_a}):\quad \hat y &= h(-z_1)+\epsilon_1 (bz_1^2 +(\gamma-c) z_1 +\beta ) m(z_1,\epsilon_1),\\
 &(r_1,\epsilon_1,z_1)\in [0,k] \times [0,\mu^2] \times I_a,
\end{align*}
after straightforward calculations. It is unique in the sense that it does not depend upon $r_1$ and the center manifold is overflowing for the $r_1=0$ subsystem. See \cite{krihog}. 
By restricting this manifold to the invariant set $r_1^3\epsilon_1=\epsilon$, we obtain our $\widehat S_{a,\epsilon,1}$. 
Finally, the extended repelling slow manifold, $\widehat S_{r,\epsilon}$, is obtained by applying the time-reversible symmetry \eqref{timeRev}. 

Next, $\hat \gamma^s$ and $\hat \gamma^w$ are clearly invariant for the flow of \eqref{slowFastEquations1} for all $\epsilon \ge 0$. By the form of $\widehat S_{a,\epsilon}$ and $\widehat S_{r,\epsilon}$ it also easily follows (setting $z_1=\chi_\pm$ and $z_3=-\chi_\pm$ in \eqsref{Saeps1}{Sreps1}, respectively) that  $\widehat S_{a,\epsilon}^*$ and $\widehat S_{r,\epsilon}^{-*}$ contain these lines. 
For the transversality of the intersection of $\widehat S_{a,\epsilon}$ and $\widehat S_{r,\epsilon}$ along $\gamma^s$ we refer to \cite{krihog}. The details are not important for the present paper. On the other hand, the transversality along $\widehat \gamma^w$  when (B) holds, the details of which is important in the following, will follow from \lemmaref{twist} below. For further details, we again refer to \cite{krihog}.
\end{proof}

Consider the chart $(\bar \epsilon=1,\kappa_2)$. Here we obtain the following local form of $2\widetilde{\widetilde X}$ 
%
%
 \begin{align}
\dot x_2 &=\beta^{-1} c(1+\phi(\hat y))-(1-\phi(\hat y)),\eqlab{ss}\\
\dot{\hat y} &=b {z}_2(1+\phi(\hat y)) -\beta x_2(1-\phi(\hat y)),\nonumber\\
 \dot z_2&=1+\phi(\hat y)+b^{-1} \gamma (1-\phi(\hat y)),\nonumber
\end{align}
by inserting \eqref{kappa2} into \eqref{slowFastEquations1} and dividing the right hand side by the common factor $r_2=\sqrt{\epsilon}>0$. Also $\dot r_2=0$. But (as in \secref{barEps1}) we shall simply view \eqref{ss} in $(x_2,\hat y,z_2)$-space with $r_2\ge 0$ as a perturbation parameter. 
 In this way, we obtain, using \lemmaref{lines}, the following local expressions for $\overline{\overline \gamma}^w$, $\overline{\overline S}_{a,\epsilon}$ and $\overline{\overline S}_{r,\epsilon}$:
 \begin{align}
 \gamma_2^w: \quad &\left(x_2,\hat y,z_2 \right)=\left(x_2,h(-\chi_+),-\chi_{+}x_2\right),\,x_2\in \mathbb R,\eqlab{Q21App}\\
 \widehat S_{a,\epsilon,2}:\quad &\hat y = h(x_2^{-1} z_2)+\mathcal O(x_2^{-2}),\eqlab{Ca2}
\end{align}
for
\begin{align*}
 (x_2,z_2)\in \left\{(x_2,z_2)\in \mathbb R^2\vert z_2 =-z_1 x_2,\,z_1 \in I_a,\,x_2 \in [-k/\sqrt{\epsilon},-\mu^{-1}]\right\},
\end{align*}
and 
\begin{align*}
 \widehat S_{r,\epsilon,2}:\quad &\hat y = h(x_2^{-1} z_2)+\mathcal O(x_2^{-2}),\nonumber
 \end{align*}
  for 
  \begin{align*}
(x_2,z_2)\in \left\{(x_2,z_2)\in \mathbb R^2\vert z_2 =z_3 x_2,\,z_3 \in I_r,\,x_2 \in [\mu^{-1},k/\sqrt{\epsilon}]\right\}.  
  \end{align*}
Notice that the expressions for $\widehat S_{a,\epsilon,2}$ and $\widehat S_{r,\epsilon,2}$ are independent of $r_2=\sqrt{\epsilon}$. (In \cite{krihog}, we refer to the corresponding $\epsilon=0$ objects as $C_{a,2}$ and $C_{r,2}$, respectively).
  
We now describe how the tangent space $T\widehat S_{a,\epsilon,2}$ of $\widehat S_{a,\epsilon,2}$ twists upon forward flow application of the variation of \eqref{ss} along $\gamma_2^w$. For this we first replace time by $x_2$ by dividing the equations for $\hat y$ and $z_2$ by $\dot x_2$ and then linearize about the solution $\gamma_2^w:\,\hat y = h(-\chi_+),\,z_2=-\chi_+x_2,\,x_2\in \mathbb R$. 
This gives the following set of variational equations:
\begin{align}
 \frac{d u}{dx_2} &= -{\lambda_{+}}^{-1}{\beta} \left(\varphi ux_2 + bv\right)\eqlab{var}\\
 \frac{dv}{dx_2} &=  -b^{-1} \lambda_{+}^{-1}  \lambda_{-}\varphi  u.\nonumber
\end{align}
Here 
\begin{align*}
\varphi \equiv \frac12 \beta {(1-\beta^{-1} b\chi_{+})^2\phi'\left(\frac{1+\beta^{-1}b \chi_{+}}{1-\beta^{-1}b \chi_{+}}\right)}.
\end{align*}
Notice that $\varphi>0$.
\begin{lemma}\lemmalab{twist}
Suppose (A) and (B). Let $n=\lfloor \xi\rfloor\in \mathbb N$ so that $n<\xi<n+1$. 
 Then the tangent space of $\widehat S_{a,\epsilon}^*$ twists along $\hat{\gamma}^w$ in the following way:

 Consider the scaling chart ($\bar \epsilon=1,\kappa_2)$ and the coordinates $(x_2,\hat y,z_2)$. Then for $\mu$ sufficiently small, define
 the following sections 
 \begin{align}
 \widehat{\Gamma}_{\textnormal{in},2}:\quad &x_2 = -\mu^{-1},(\hat y,z_2) \in \widehat R_{\textnormal{in},2},\nonumber\\
 \widehat{\Gamma}_{\textnormal{out},2}:\quad &x_2 = \mu^{-1},(\hat y,z_2) \in \widehat R_{\textnormal{out},2}.\eqlab{GammaOut2}
 \end{align}
    Here $\widehat R_{\textnormal{in},2}$ and $\widehat R_{\textnormal{out},2}$ are suitably large rectangles in the $(\hat y,z_2)$-plane such that both sections are transverse to $\hat{\gamma}_2^{w}$. Then the tangent vector of $\widehat S_{a,\epsilon,2}\cap \widehat \Gamma_{\textnormal{in},2}$ at $\hat{\gamma}_2^w\cap \widehat \Gamma_{\textnormal{in},2}$: 
\begin{align}
 \varpi_{\textnormal{\textnormal{in}}} = \begin{pmatrix}
                               0\\
                              \frac{2 \beta^{-1} b}{(1-\beta^{-1}b \chi_+)^2\phi'\left(\phi^{-1}\left(\frac{1+\beta^{-1} {b} \chi_+}{1-\beta^{-1} {b} \chi_+} \right)\right)}\mu+\mathcal O(\mu^2)\\
                               1\\
                                \end{pmatrix},\eqlab{vin}
\end{align}
 is under the forward flow of the variational equations \eqref{var}, transformed to a tangent vector $\varpi_{\textnormal{\textnormal{out}}}=\varpi_{\textnormal{\textnormal{out}}}(\mu)$ of $\widehat S_{a,\epsilon,2}^*\cap \widehat \Gamma_{\textnormal{out},2}$ at $\hat{\gamma}_2^w\cap \widehat \Gamma_{\textnormal{out},2}$, which is transverse to the tangent space of $\widehat S_{r,\epsilon,2}\cap \widehat \Gamma_{\textnormal{out},2}$ at $\hat{\gamma}_2^w\cap \widehat \Gamma_{\textnormal{out},2}$, satisfying
 \begin{align}
 \overline{\varpi}_{\textnormal{out}}\equiv \frac{\varpi_{\textnormal{out}}(\mu)}{\vert \varpi_{\textnormal{out}}(\mu)\vert} = (0,\chi+o(1),o(1)),\eqlab{varpiout}
\end{align}
as $\mu\rightarrow 0$, 
where
\begin{align}
 \chi = \left\{\begin{array}{c}
 -1 \quad \text{if $n$ is odd}\\
               1 \quad \text{if $n$ is even}               
              \end{array}
\right. \eqlab{zetaSign}
\end{align}
  
\end{lemma}

\begin{proof}
The result follows from \cite[Lemma 7.8]{krihog} and the fact that  \eqref{var} can be written as the Weber equation
\begin{align}
 \frac{d^2v}{d\tilde x_2} - \tilde x_2 \frac{dv}{d\tilde x_2}+\xi v = 0,\eqlab{weber}
\end{align}
by replacing $x_2$ by 
\begin{align*}
 \tilde x_2 = (-\varphi \beta \lambda_+^{-1})^{1/2} x_2,
\end{align*}
and eliminating $u$. The twisting is then determined by the zeros the solution of \eqref{weber} corresponding to displacements along $T\widehat S_{a,\epsilon,2}$ for $x_2\rightarrow -\infty$. See also \cite{szmolyan_canards_2001} (twisting along weak canard for the classical folded node in $\mathbb R^3$). The form of $\varpi_{\textnormal{in}}$ in \eqref{vin} is obtained by differentiating the expression in \eqref{Ca2} with respect to $z_2$. 
%
%
\end{proof}

\begin{figure}[h!] 
\begin{center}
{\includegraphics[width=.9\textwidth]{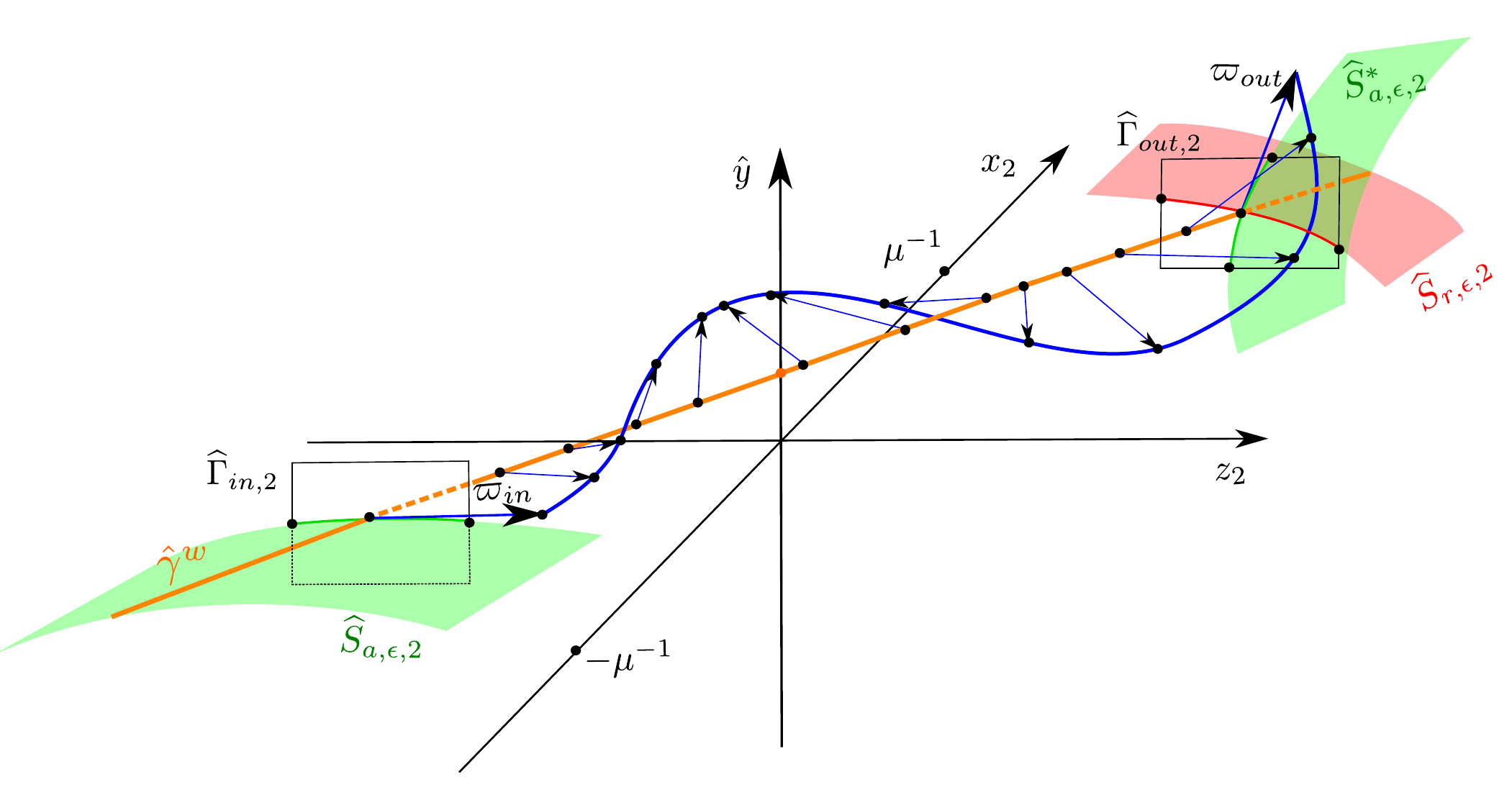}}
\end{center}
 \caption{\ed{Twist of slow manifolds along the weak canard described in \lemmaref{twist} (here $n$ even). The blue vectors are tangent vectors to $\widehat S_{a,\epsilon,2}^*$ and are described by the variational equations \eqref{var}}. }
\figlab{SarEps2}
\end{figure}

}

\subsection{Outline of the proof of \thmref{mainThm}}

\figref{SarEps2} illustrates the consequences of \lemmaref{twist}: Using the coordinates $(x_2,\hat y,z_2)$ in the scaling chart $(\bar \epsilon=1,\kappa_2)$, we see that points on the attracting slow manifold $\widehat S_{a,\epsilon,2}$ that are close to, but on either side of, the weak canard $\hat{\gamma}^w_2$ will be displaced along opposite directions once reaching the vicinity of the repelling slow manifold $\widehat S_{r,\epsilon,2}$. In particular, for the case illustrated in \figref{SarEps2} for $n$ even, initial conditions on $\widehat S_{a,\epsilon,2}$ displaced from the weak canard along the direction defined by $\varpi_{in}$, will under the forward flow eventually be above the repelling slow manifold $\widehat S_{r,\epsilon,2}$. On the other hand, initial conditions displaced in the opposite direction will eventually be below $\widehat S_{r,\epsilon,2}$. (Also, other way around if $n$ is odd. )
The proof of \thmref{mainThm} therefore naturally divides into two separate cases: moving upwards, which we shall call case (a), and moving downwards, abbreviated case (b). In reference to \lemmaref{twist}, we define the cases formally as follows:
\begin{definition}\defnlab{caseiii}
Let $n=\lfloor \xi\rfloor$ (as in \lemmaref{twist}) be  the greatest integer less than $\xi>1$ where $\xi = \lambda^{-1}_+\lambda_-$ is the ratio of eigenvalues 
defined in \eqref{nonResonance2}. Then cases (a) and (b) are defined as follows:
\begin{itemize}
 \item Case (a): 
 \begin{itemize} 
 \item $\Pi_{\text{in},0}\cap \Sigma$ is between $\gamma^w$ and $\gamma^s$ and $n$ is even,
 \end{itemize} \textit{or} 
 \begin{itemize}
 \item $\Pi_{\text{in},0}\cap \Sigma$ is between $l^-$ and $\gamma^w$ and $n$ is odd.
 \end{itemize}
 \item Case (b):  \begin{itemize} 
 \item $\Pi_{\text{in},0}\cap \Sigma$ is between $\gamma^w$ and $\gamma^s$ and $n$ is odd,
 \end{itemize}
 \textit{or} 
 \begin{itemize}
 \item $\Pi_{\textnormal{in},0}\cap \Sigma$ is between $l^-$ and $\gamma^w$ and $n$ is even.
 \end{itemize}
\end{itemize}
\end{definition}

In \secref{passage}, we begin the detailed proof of \thmref{mainThm} by describing the initial passage through $q$. During this phase, there is a contraction towards the weak canard $\gamma^w$, that ultimately produces the contraction of the map $\mathcal L_\epsilon$. 
The position of $\Pi_{\text{in},0}\cap \Sigma$ relative to $l^-$, $\gamma^w$ and $\gamma^s$ will determine the directions of this initial contraction towards $\gamma^w$.
When combined with the rotation of the tangent spaces, described by \lemmaref{twist}, this contraction allows us to separate cases (a) and (b), as defined in \defnref{caseiii}, by carrying the forward flow of $\Pi_{\text{in},\epsilon}$ towards $\hat y=\pm \delta^{-1}$, respectively, in chart $(\epsilon=1,\kappa_3)$.  See \propref{step1} and its proof.

Then, by working in charts $(\bar y=\pm 1,\kappa_3)$, we follow orbits from $\hat y= \delta^{-1}$ ($\hat \epsilon=\delta$ in $\bar y=1$ by \eqref{hatEpshatY}) for case (a) in \secref{casea} and from $\hat y= -\delta^{-1}$ ($\hat \epsilon=-\delta$ in $\bar y=-1$ by \eqref{hatEpshatY}) case (b) in \secref{caseb}.  We will then successively identify certain hyperbolic segments of $\widetilde{\widetilde X}$ on $\overline{\overline q}$ that guide the forward flow. In the blowup space $\overline{\overline{P}}$, we denote these segments by $\overline{\overline{Q}}^{j,a}$ and $\overline{\overline{Q}}^{j,b}$, $j=1,2,\ldots$, for case (a) and case (b), respectively. In both cases we end up on $\bar y =1$ within the stable manifold of a hyperbolic point $\overline{\overline{u}}$ of $\widetilde{\widetilde{X}}$. This point has a $1D$ unstable manifold $\overline{\overline{U}}$ which upon blowing down becomes the special orbit $U$.

We complete the proof by perturbing away from these segments at $\epsilon=0$ using standard local hyperbolic methods of dynamical systems theory in the appropriate charts $(\bar \epsilon=1,\kappa_i)$ and $(\bar y=\pm 1,\kappa_i)$. This allows us, upon blowing back down, to successfully guide the forward flow of $\Pi_{\text{in},\epsilon}$ close to $U$, as described in our main theorem.

In further details, the result is that, as $\epsilon\rightarrow 0$, the forward flow of $\Pi_{\textnormal{in},\epsilon}$ under $\widetilde{\widetilde X}$ converges in the Hausdorff distance to a set that within $\overline{\overline q}$ is the union of $$\overline{\overline{\gamma}}^w\cap \{r=0\}=\left\{(r,(\bar x,\bar \rho,\bar z),(\bar y,\bar \epsilon))\in \overline{\overline P}\vert r=0,\,\bar \epsilon^{-1} \bar y = h(-\chi_+),\,\bar x^{-1} \bar z = -\chi_+\right\},$$ see also \eqref{barGammaWs},  and the following case-dependent segments
\begin{itemize}
\item for case (a): $\overline{\overline{Q}}^{1,a}$, $\overline{\overline{Q}}^{2,a}$;
\item for case (b): 
\begin{itemize}
 \item ($z_1^*\le 0$): $\overline{\overline{Q}}^{1,b}$, $\overline{\overline{Q}}^{2,b}$, $\overline{\overline{Q}}^{3,b}$, $\overline{\overline{Q}}^{4,b}$, $\overline{\overline{Q}}^{5,b}$;
 \item ($z_1^*>0$): $\overline{\overline{Q}}^{1,b}$, $\overline{\overline{Q}}^{2,b}$, $\overline{\overline{Q}}^{3,b}$, $\overline{\overline{Q}}^{4,b}$.
\end{itemize}
%
\end{itemize}
We illustrate these segments in \figref{singularCycleA} for case (a) and in \figref{singularCycleB} for case (b) with $z_1^*<0$. As opposed to \figref{Gamma0barq} we now represent $\overline{\overline q}$ as a semicircle of $2D$-disks (using $S^2\cap \{\bar \rho\ge 0\}\cong D^2$). There are three important discs: at $\hat y=\bar \epsilon^{-1} \bar y = h(-\chi_+)$, which contains $\overline{\overline \gamma}^w\cap \{r=0\}$, and is illustrated in \figref{singularCycleA} and \figref{singularCycleB} using orange (boundaries); at $\bar y =1$, which contains $\overline{\overline Q}^{2,a}$ in case (a) and $\overline{\overline Q}^{5,b}$ in case (b), and is illustrated using cyan (boundaries); and finally, at $\bar y=-1$, which contains $\overline{\overline Q}^{2,b}$ in case (b), and is illustrated using purple (boundaries). All other segments are subsets of the face $\bar \epsilon=0$ of $\overline{\overline q}$. In particular, $\overline{\overline Q}^{4,b}$ belongs to the set $\overline{\overline L}_a$, recall \propref{WidetildeWidetildeX}. Notice again the point $\overline{\overline u}$ on the top disc $\bar y=1$ of $\overline{\overline q}$. In the blowup space, $U$ is the $1D$ unstable manifold of $\overline{\overline u}$. 

\ed{
For an additional depiction, in \figref{projectiveA} (case (a)) and \figref{projectiveB} (case (b)) we use $\bar z/\bar x$ as an axis to illustrate a forward orbit of a point on $\Pi_{in}$ under the flow of $\widetilde{\widetilde X}$ on $\overline{\overline P}$ in the limit $\epsilon\rightarrow 0$. Notice that $z_1=-\bar z/\bar x$ when $\bar x<0$ and $z_3=\bar z/\bar x$ when $\bar x>0$, see \eqref{kappa1} and \eqref{kappa3}. Therefore  we do not see the spheres, that occur as a result of the second blowup \eqref{blowup1} and shown in  \figref{singularCycleA} and  \figref{singularCycleB} as discs, in this projection. In fact, $\overline S_a$ and $\overline{S}_r$ coincide in this projection. (More precisely, $\overline S_a$, $\overline{\overline L}_a$, $\overline S_r$ and $\overline{\overline L}_r$ all coincide). Also $\overline{\gamma}^{s}$ and $\overline{\gamma}^w$ project to single points with $\bar z/\bar x = -\chi_\pm$, respectively, see \eqref{barGammaWs}. (Obviously, these properties only hold for the regularization of the PWL system \bpwl{}; they do not hold in general for the regularization of the PWS system \bpws{}.) But, on the other hand, we see the role of the dynamics outside $\overline{\overline{q}}$ more clearly in \figref{projectiveA} and \figref{projectiveB}. This was hidden in \figref{singularCycleA} and  \figref{singularCycleB}. 

In \figref{projectiveA} and \figref{projectiveB} we see the following: First from a point in $\Pi_{\text{in},0}$ we follow the forward flow of $X^+$ inside $\bar y=1$ until we reach $\rho=0$ and the normally hyperbolic set $\overline M^+$. From there, we follow a heteroclinic connection, connecting our landing point on $\overline M^+$ with a base point on $\overline S_a$ through a stable critical fiber. Subsequently we then follow the slow flow (sliding equations, recall \propref{WidetildeX}) on $\overline S_a$ (green segment) and contract towards $\overline \gamma^w$ (which appears as a point with $\bar z/\bar x = -\chi_+$ in our projection). Then, after having moved across the sphere along $\overline{\overline{\gamma}}^w\cap \{r=0\}$ at $\bar \epsilon^{-1} \bar y  = h(-\chi_+)$ (orange in \figref{singularCycleA} and \figref{singularCycleB}) we then follow the segments $\overline{\overline{Q}}^{j,a}$ and $\overline{\overline{Q}}^{j,b}$, $j=1,2,\ldots$ illustrated in \figref{singularCycleA} in case (a) and \figref{singularCycleB} in case (b), respectively. We re-use the colours of these segments in \figref{singularCycleA} and \figref{singularCycleB} in the present figures. Notice in particular, that $\overline{\overline Q}^{2,b}$ (cyan) in \figref{projectiveB}, illustrating case (b), approaches $\bar z/\bar x=z_1^*$ on $\bar y=-1$, cf. \lemmaref{sigma0}, which is outside the funnel region (to the left of $\bar z/\bar x=-\chi_-$, the value of $\bar z/\bar x$ at $\overline \gamma^s$, in this projection). We therefore move towards the fold at $\bar z=0$ (by following the slow flow on $\overline {\overline L}_a$).  }

\begin{figure}[h!] 
\begin{center}
{\includegraphics[width=.8\textwidth]{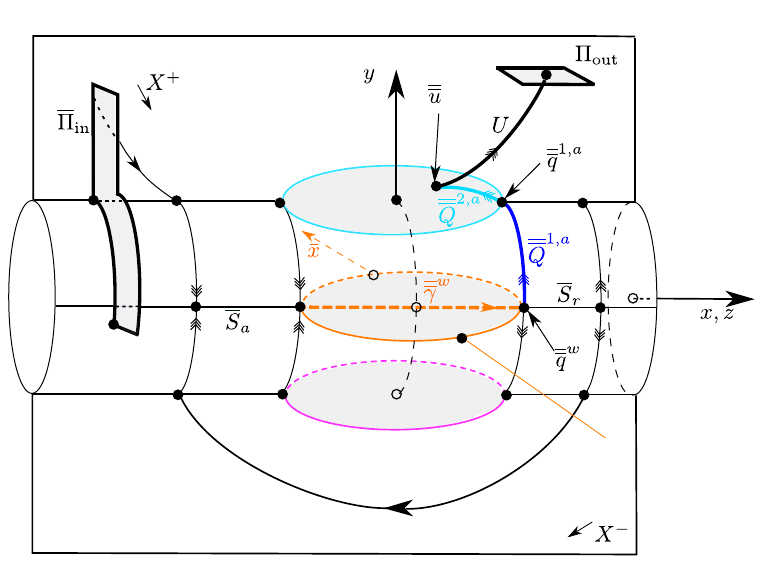}}
\end{center}
 \caption{\ed{Schematic projection of case (a). In this figure the semicircle of hemispheres $\overline{\overline q}$ is represented as a semicircle of discs. $\overline{\overline Q}^{2,a}\subset \{\bar y=1\}$ (cyan) is asymptotic to the hyperbolic point $\overline{\overline u}$. $U$ of $X_+$ is then the $1D$ unstable manifold of $\overline{\overline u}$ which completes the singular orbit}. }
\figlab{singularCycleA}
\end{figure}
\begin{figure}[h!] 
\begin{center}
{\includegraphics[width=.8\textwidth]{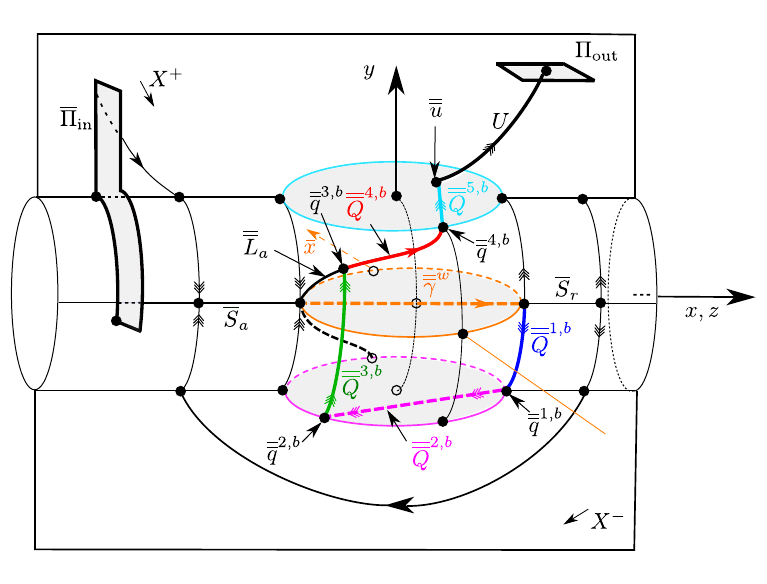}}
\end{center}
 \caption{\ed{Schematic projection of case (b), with $z_{1}^*<0$. In contrast to case (a) of \figref{singularCycleA}, $\overline{\overline Q}^{2,b}$ and $\overline{\overline Q}^{3,b}$ (green) re-injects the dynamics into the region corresponding to \textit{stable sliding} (where $\bar x<0,\,\bar z<0$). Recall also \secref{intuition} and \lemmaref{sigma0}. Here \textit{stable sliding} is formally extended onto $\overline{\overline q}$ through the partially hyperbolic line $\overline{\overline L}_a \subset \{\bar \epsilon=0\}$ of equilibria for $\widetilde{\widetilde X}$ (see also \propref{WidetildeWidetildeX} and \eqref{La1} below). $\overline{\overline Q}^{3,b}$ is then a critical, stable fiber of $\overline{\overline L}_a$ with base point $\overline{\overline q}^{3,b}$. In agreement with \secref{intuition}, the slow forward flow of $\overline{\overline q}^{3,b}$ on $\overline{\overline L}_a$, giving rise to the orbit segment $\overline{\overline Q}^{4,b}$ (red), then approaches $\overline{\overline q}^{4,b}$ at $\bar y=1,\,\bar z=0$ on the ``visible'' fold $\overline{\overline l}^+$. At $\overline{\overline q}^{4,b}$, the line $\overline{\overline{L}}_a$ is actually fully nonhyperbolic (due to the lack of hyperbolicity of $\overline{\overline l}^+$, see \propref{WidetildeWidetildeX}). But using a separate blowup here (see details in \appref{z1StarNeg}) we obtain separate hyperbolic segments, that upon blowing down gives $\overline Q^{5,b}\subset \{\bar y=1\}$ (cyan). This orbit then finally connects to $U$ through the hyperbolic point $\overline{\overline u}$ as in case (a)}. }
\figlab{singularCycleB}
\end{figure}

\begin{figure}[h!] 
\begin{center}
{\includegraphics[width=.7\textwidth]{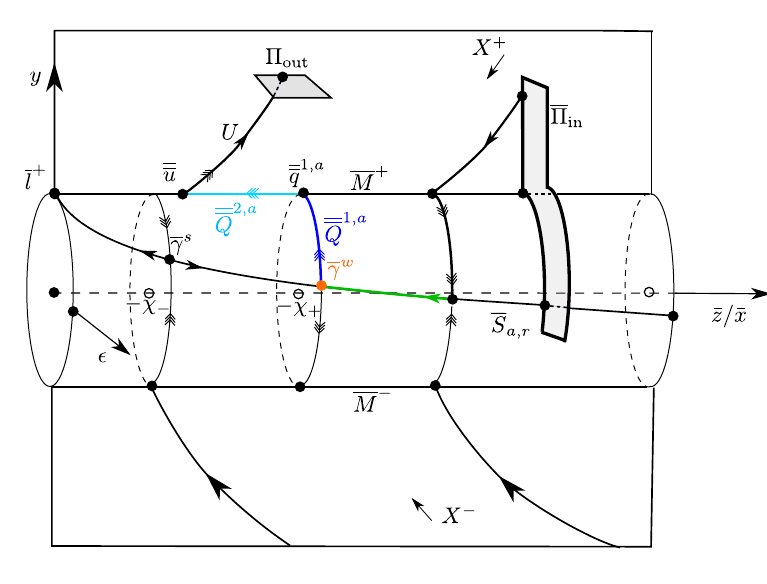}}
\end{center}
 \caption{\ed{Schematic projection of case (a) using $\bar z/\bar x$ as an axis. See text and caption of \figref{singularCycleA} for more details}. }
\figlab{projectiveA}
\end{figure}
\begin{figure}[h!] 
\begin{center}
{\includegraphics[width=.7\textwidth]{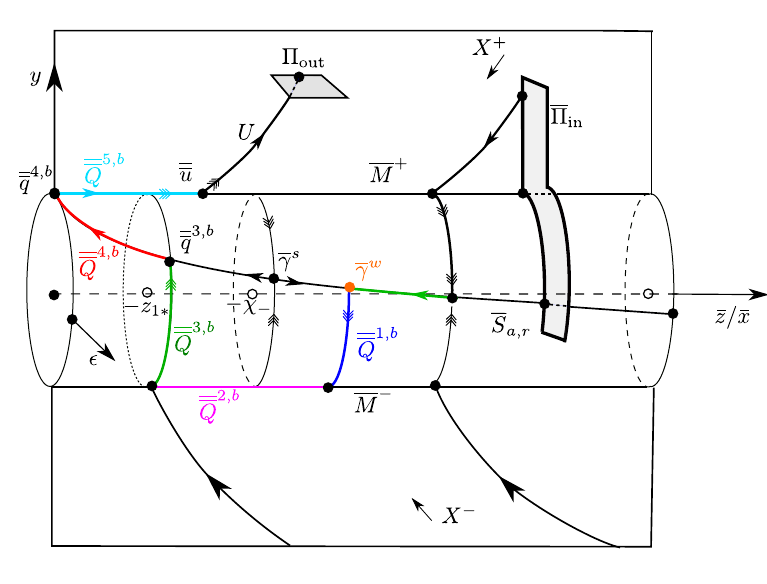}}
\end{center}
 \caption{\ed{Schematic projection of case (b) using $\bar z/\bar x$ as an axis. See text and caption of \figref{singularCycleB} for more details. Also compare with \figref{projectiveA}}.
 }
\figlab{projectiveB}
\end{figure}

\subsection{Initial passage through $q$}\seclab{passage}
In this section we describe a mapping from $\Pi_{\text{in},\epsilon}$ to $\{\hat y=\pm \delta^{-1}\}$ with $\delta>0$ sufficiently small, using the exit chart $(\bar \epsilon = 1,\kappa_3)$ given in \eqref{kappa3}. 

Define the case-dependent section $\widehat \Lambda_{\textnormal{out},3}$ as follows: 
 \begin{align}
  \textnormal{case (a)}:\quad &\widehat \Lambda_{\textnormal{out},3}:\quad \hat y = \delta^{-1},(r_3,z_3,\epsilon_3) \in \widehat R_{\textnormal{out},3}\cap \{r_3\ge 0,\,\epsilon_3\ge 0\},\eqlab{LambdaOut3}\\
  \textnormal{case (b)}:\quad &\widehat \Lambda_{\textnormal{out},3}:\quad \hat y = -\delta^{-1},(r_3,z_3,\epsilon_3) \in \widehat R_{\textnormal{out},3}\cap \{r_3\ge 0,\,\epsilon_3\ge 0\},\nonumber
  \end{align}
 where $\widehat R_{\textnormal{out},3}$ is now a small neighbourhood of $(0,-\chi_+,0)$. We illustrate $\overline{\widehat \Lambda}_{\textnormal{out}}$ in \figref{barEpsEq1A} (a) and (b) for case (a) and (b), respectively, in the $\bar \epsilon=1$ chart of the initial blowup \eqref{blowup0}.
 \begin{figure}[h!] 
\begin{center}
\subfigure[Case (a)]{\includegraphics[width=.495\textwidth]{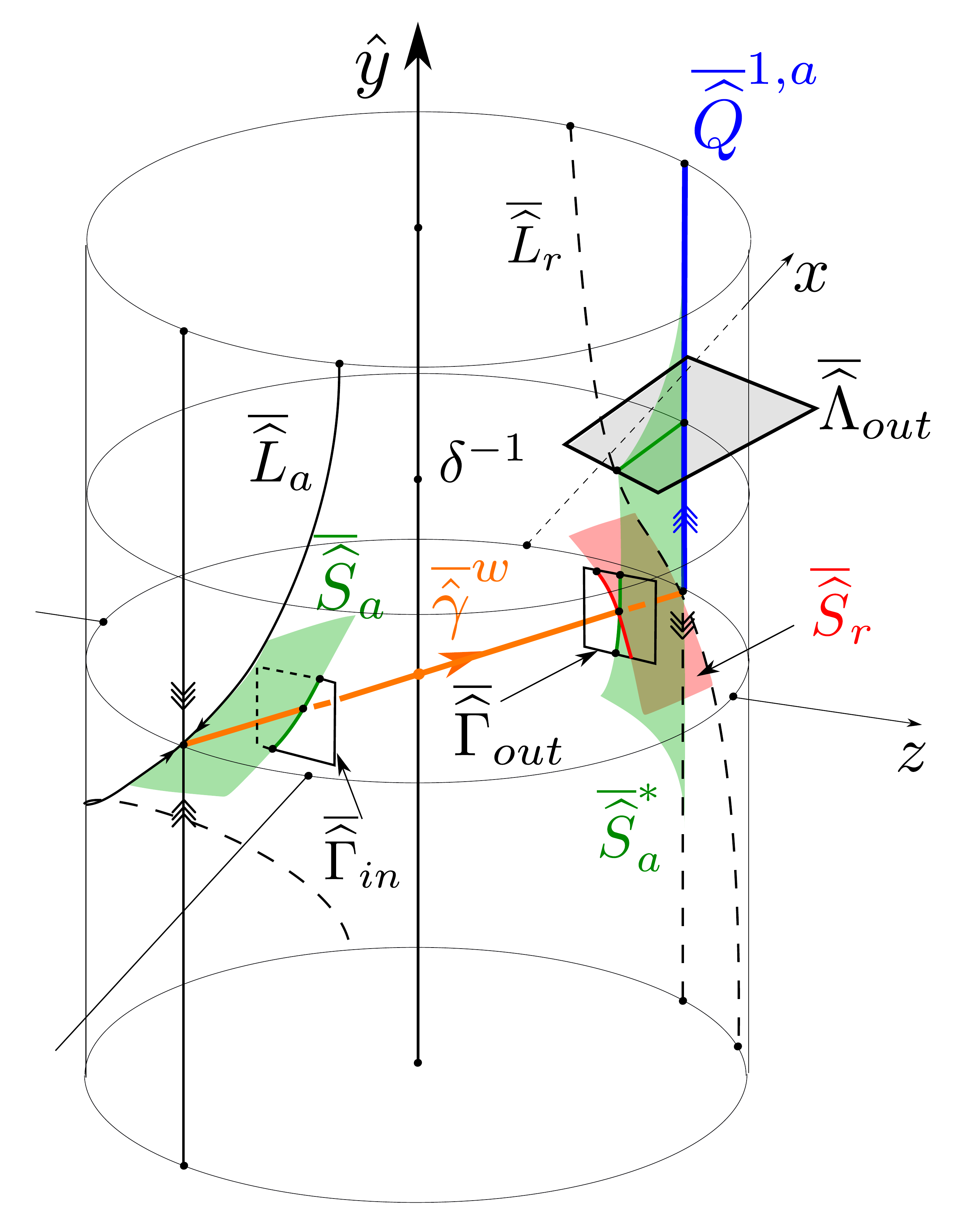}}
\subfigure[Case (b)]{\includegraphics[width=.495\textwidth]{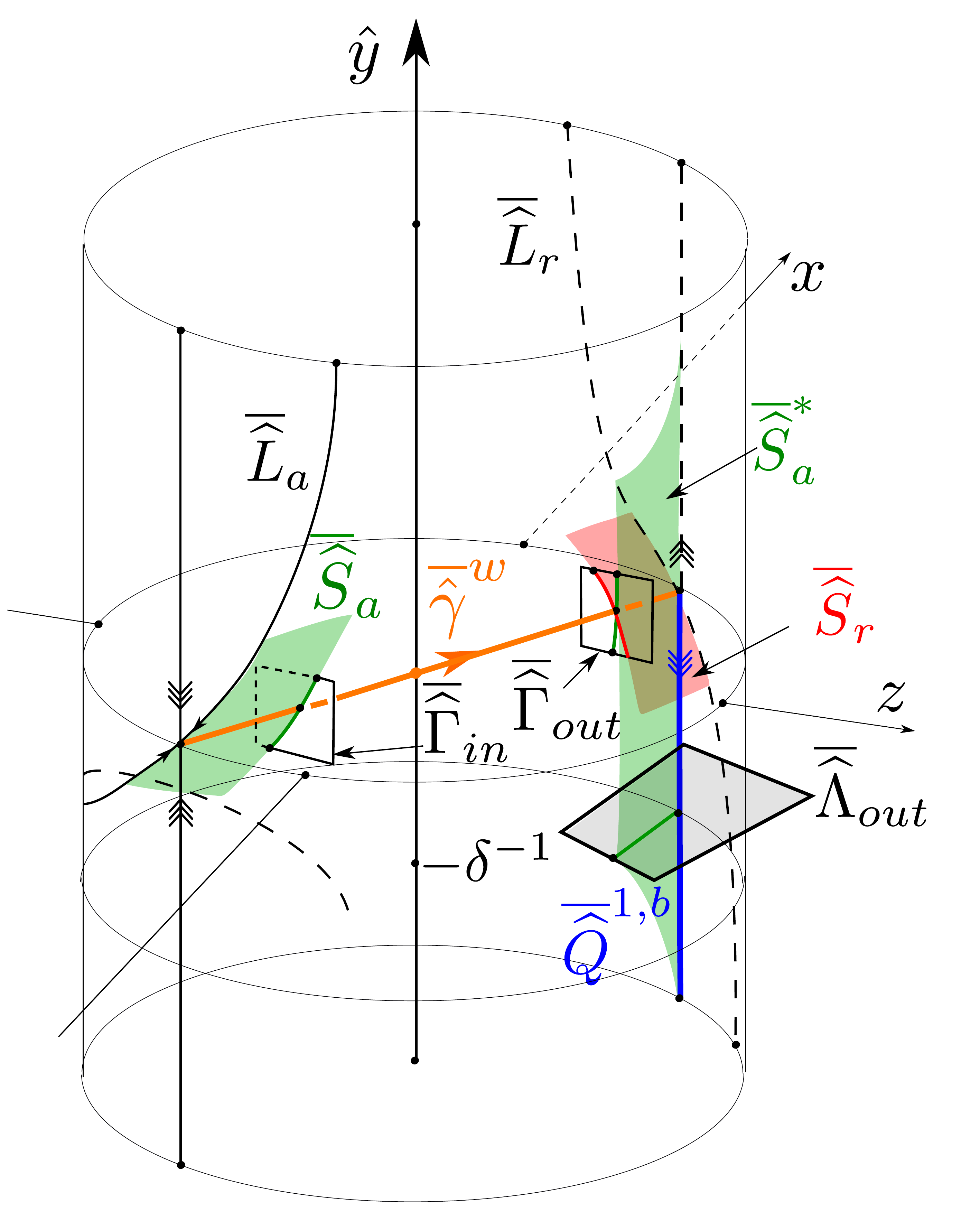}}
\end{center}
 \caption{\ed{Schematic representation of important objects relevant to the proof of \propref{step1} in chart $\bar \epsilon=1$ of blowup \eqref{blowup0}. Notice, in particular, that $\overline{\widehat \Lambda}_{\text{out}}$, which is covered by chart $(\bar \epsilon=1,\kappa_3)$ and where it is denoted by $\widehat \Lambda_{\text{out},3}$, see \eqref{LambdaOut3}, is case-dependent: It is a subset of $\{\hat y=\delta^{-1}\}$ in case (a), see Fig. (a), and a subset of $\{\hat y=-\delta^{-1}\}$ in case (b), see Fig. (b)}. }
\figlab{barEpsEq1A}
\end{figure}
    Then we have the following.
  \begin{proposition}\proplab{step1}
The mapping
\begin{align}
\widehat{\mathcal L}_{3,\epsilon}:\quad \Pi_{\textnormal{in},\epsilon}\rightarrow \widehat \Lambda_{\textnormal{out},3},\,(y,z) \mapsto (r_3,z_3,\epsilon_3) = (r_{3+}(y,z),z_{3+}(y,z),\epsilon_{3+}(y,z)),\eqlab{hatLeps}
\end{align}
where $r_{3}^2\epsilon_{3}=\epsilon$, 
obtained by the forward flow, is well-defined. 
In fact, $\widehat{\mathcal L}_{3,\epsilon}$ is contractive for $\epsilon>0$ sufficiently small, in the following sense. Fix any $\upsilon \in (0,1)$, $K>0$ sufficiently large and consider the set
\begin{align*}
 \widehat D_3 = \left\{(r_3,\hat y,z_3,\epsilon_3)\in \widehat{\Lambda}_{\textnormal{out},3} \vert r_3^2 \epsilon_3 = \epsilon,\, K^{-1} \ln^{-1} \epsilon^{-1} \le \textnormal{max}\,(\vert z_3+\chi_+\vert,\epsilon_3)\le K  \ln^{-1} \epsilon^{-1}\right\},
\end{align*}
contained within $\widehat{\Lambda}_{\textnormal{out},3}$. Then $\widehat{\mathcal L}_{3,\epsilon}(\Pi_{\textnormal{in},\epsilon})\subset \widehat D_3 $. 
Furthermore,
the eigenvalues of the Jacobian 
 \begin{align*}
D\begin{pmatrix}
  \hat{z}_{3+}\\
  \epsilon_{3+}
 \end{pmatrix}
 (y,z),
 \end{align*}
 are $\mathcal O(e^{-C/\epsilon})$ and  $\mathcal O(\epsilon^{(\xi-1) \upsilon/2})$, where $\xi$ is defined in assumption (B). 
\end{proposition}
\begin{proof}
From the definition of $\Pi_{\text{in},\epsilon}$ in \eqref{Pin} all orbits initially contract towards $S_{a,\epsilon}$. Therefore we work in the entry chart $(\bar \epsilon =1,\kappa_1)$ and the scaling chart $(\bar \epsilon=1,\kappa_2)$. $\widehat S_{a,\epsilon}$ as given in \eqref{Saeps1} will be covered in both of these charts. We denote this unique slow manifold by $\widehat S_{a,\epsilon,1}$ and $\widehat S_{a,\epsilon,2}$ in the charts $(\bar \epsilon =1,\kappa_1)$ and $(\bar \epsilon=1,\kappa_2)$, respectively. Then we have the following.
\begin{lemma}\lemmalab{contract}
 Consider the mapping 
\begin{align}
\widehat{\mathcal L}_{2,\epsilon}:\quad \Pi_{\textnormal{in},\epsilon}\rightarrow \widehat{\Gamma}_{\textnormal{in},2},\,(y,z) \mapsto (x_2,\hat y,z_2) = (-\mu^{-1},\hat y_+(y,z),z_{2+}(y,z)),\eqlab{Leps2}
\end{align}
with $r_2=\sqrt{\epsilon}$,
obtained by the forward flow. Then $\widehat{\mathcal L}_{2,\epsilon}$ is contractive for $\epsilon>0$ sufficiently small, in the following sense. Fix any $\upsilon \in (0,1)$. Then the image $\widehat{\mathcal L}_{2,\epsilon}(\Pi_{\textnormal{in},\epsilon})$ is a $\mathcal O(e^{-C/\epsilon})$-thin strip around $\widehat S_{a,\epsilon,2}\cap \widehat{\Gamma}_{\textnormal{in},2}$ with $C>0$. The width of the strip is $\mathcal O(\epsilon^{(\xi-1) \upsilon/2})$. 
In particular, let $A_{\hat \gamma_2^w}$ be 
  an annulus in $\widehat \Gamma_{\textnormal{in},2}$, centred around $\hat \gamma_2^w\cap \widehat{\Gamma}_{\textnormal{in},2}$, with inner radius $K^{-1} \epsilon^{(\xi-1) /(2\upsilon)}$ and outer radius $K \epsilon^{(\xi-1) \upsilon/2}$, with $K$ sufficiently large. 
  Moreover, consider $N_{\widehat S_{a,\epsilon,2}}$ as the $C_1^{-1} e^{-C_1/\epsilon}$-neighbourhood of $\widehat S_{a,\epsilon,2}\cap \widehat{\Gamma}_{\textnormal{in},2}$ within $\widehat{\Gamma}_{\textnormal{in},2}$:
  \begin{align*}
   N_{\widehat S_{a,\epsilon,2}} = \{(x_2,\hat y,z_2)\in \widehat \Gamma_{\textnormal{in},2}\vert \textnormal{dist}((x_2,\hat y,z_2),\widehat S_{a,\epsilon,2}\cap \widehat{\Gamma}_{\textnormal{in},2})\le C_1^{-1} e^{-C_1/\epsilon}\},
  \end{align*}
with $C_1>0$ sufficiently small. Then for all $\epsilon>0$ sufficiently small, 
\begin{align}
\widehat{\mathcal L}_{2,\epsilon}(\Pi_{\textnormal{in},\epsilon}) \subset A_{\hat{\gamma}_2^w}\cap N_{\widehat S_{a,\epsilon,2}}.\eqlab{bound}
\end{align}
Furthermore,
the Jacobian 
 \begin{align*}
D\begin{pmatrix}
  \hat y_+\\
  \hat z_{2+}
 \end{pmatrix}
 (y,z),
 \end{align*}
 has one eigenvalue of $\mathcal O(e^{-C/\epsilon})$, $C>0$, and another one of $\mathcal O(\epsilon^{(\xi-1) \upsilon/2})$, the estimates being uniform in $(y,z)$. 
\end{lemma}

\begin{proof}
Consider the local form of $2\widetilde{\widetilde X}$ in \eqref{BarEps1Kappa1Eqns}. Then on 
\begin{align*}
\widehat S_{a,\epsilon,1}:\quad \hat y &= h(-z_1)+\epsilon_1 (bz_1^2 +(\gamma-c) z_1 +\beta ) m(z_1,\epsilon_1),\quad (r_1,\epsilon_1,z_1)\in [0,k] \times [0,\mu^2] \times I_a,
 \end{align*}
 cf. \lemmaref{lines}, we obtain the reduced problem 
\begin{align}
 \dot z_1 &=(c-bz_1)^{-1} \left(bz_1^2+(c-\gamma)z_1+\beta\right)\left(1+\epsilon_1 n_1(z_1,\epsilon_1)\right),\eqlab{reducedSaeps1}\\
  \dot \epsilon_1&=2\epsilon_1 ,\nonumber
\end{align}
with $n_1$ smooth, after division of the right hand side by 
$G(\hat y)\epsilon_1$. This quantity is positive inside the funnel and therefore corresponds to a nonlinear transformation of time. Since the $r_1$-equation decouples, $r_1$ can be determined by the conservation of $\epsilon$: $r_1^2 = \epsilon_1^{-1} \epsilon$. The point $z_1=\chi_+<0$, $\epsilon_1=0$ is a saddle for \eqref{reducedSaeps1}, whose linearization has the eigenvalues
\begin{align*}
 1-\xi,\,2. 
\end{align*}
Notice $1-\xi<0$ by (A). 
The result then follows by (a) simple estimation through Gronwall's inequality (or simply using general results on Silnikov boundary value problems), (b) the exponential contraction towards $\widehat S_{a,\epsilon,1}$, (c) applying the coordinate transformation \eqref{kappa21}. 
\end{proof}
To complete the proof of \propref{step1}, we subsequently describe the mapping 
\begin{align}
 \widehat{\mathcal L}_{2,\epsilon}(\Pi_{\text{in},\epsilon})\subset \widehat{\Gamma}_{\textnormal{in},2}\rightarrow \widehat{\Gamma}_{\textnormal{out},2},\eqlab{mappingk2}
\end{align}
from $\widehat{\mathcal L}_{2,\epsilon}(\Pi_{\text{in},\epsilon})$ to $\widehat{\Gamma}_{\textnormal{out},2}$, see \eqref{GammaOut2},
in the chart $(\bar \epsilon=1,\kappa_2)$. 
By \lemmaref{contract}, we can do this by considering the variational equations \eqref{var} about $\widehat \gamma_2^w$. Indeed, the image $\widehat{\mathcal L}_{2,\epsilon}(\Pi_{\text{in},\epsilon})$ is close to the tangent space of $\widehat S_{a,\epsilon,2}\cap \widehat \Gamma_{\text{in},2}$ for $0<\epsilon\ll 1$. By \eqref{reducedSaeps1}, the case when $\Pi_{\text{in},0}\cap \Sigma$ is between $\hat \gamma^w$ and $\hat \gamma^s$, then corresponds to variations in the positive direction of $\varpi_{\text{in}}$, recall \eqref{vin} and see \figref{SarEps2}. By \lemmaref{twist}, in particular \eqref{zetaSign}, the forward flow of $\widehat{\mathcal L}_{2,\epsilon}(\Pi_{\text{in},\epsilon})$ therefore intersects $\widehat{\Gamma}_{\textnormal{out},2}$ below (above) the unique slow manifold $\widehat S_{r,\epsilon,2}$, see \eqref{Sreps1} for $x_2=1/\sqrt{\epsilon_3}\ge \mu^{-1}$, when $n$ is even (odd), respectively. On the other hand, the case when $\Pi_{\text{in},0}\cap \Sigma$ is between $l^-$ and $\hat \gamma^w$, corresponds to variations in the negative direction of $\varpi_{\text{in}}$ and the forward flow of $\widehat{\mathcal L}_{2,\epsilon}(\Pi_{\text{in},0})$ therefore intersects $\widehat{\Gamma}_{\textnormal{out},2}$ above (below) the unique slow manifold $\widehat S_{r,\epsilon,2}$ when $n$ is even (odd), respectively. Under the $O(1)$-time application of the forward flow in chart $\kappa_2$, the image of \eqref{mappingk2} remains cf. \eqref{bound} sufficiently (for the proceeding arguments to follow through) bounded away from the weak canard and $\widehat S_{r,\epsilon,2}$ at $\widehat{\Gamma}_{\textnormal{out},2}$. 
In the chart $(\bar \epsilon=1,\kappa_3)$, we obtain the following equations:
 \begin{align}
 \dot r_3 &=r_3\epsilon_3 G(\hat y),\eqlab{BarEps1Kappa3Eqns}\\
 \dot{\hat y} &=bz_1(1+\phi(\hat y))-\beta (1-\phi(\hat y)),\nonumber\\
 \dot z_3 &=\epsilon_3 (H(\hat y)-z_3 G(\hat y)),\nonumber\\
 \dot \epsilon_3 &= -2\epsilon_3^2 G(\hat y),\nonumber
\end{align}
recall \eqref{GHFunctions}.
The line
\begin{align}
\widehat L_{r,3}:\quad r_3=\epsilon_3 = 0,\,\hat y = h(z_3),\, z_3 \in (0,\infty),
\end{align}
(the local form of $\overline{\overline L}_r$ in \propref{WidetildeWidetildeX})
is partially hyperbolic and, as in chart $\kappa_1$, this produces a unique center manifold:
\begin{align*}
 W^c(\widehat L_{r,3}\vert_{z_3\in I_r}):\quad \hat y &= h(z_3)+\epsilon_3 (bz_3^2 -(\gamma-c) z_3 +\beta ) m(-z_3,\epsilon_3),\\
 &(r_3,\epsilon_3,z_3)\in [0,k] \times [0,\mu^2] \times  I_r,
\end{align*}
with $\mu$, $k$ and $I_r$ as in \lemmaref{lines}, 
which upon restriction to the invariant set $r_3^2\epsilon_3=\epsilon$ gives $\widehat S_{r,\epsilon,3}$ \eqref{Sreps1}. Here $\hat \gamma_3^w$ intersects $\widehat L_{r,3}$ in 
\begin{align}
 \hat q_3^w:\quad r_3=0,\epsilon_3=0,\,z_3 = -\chi_+,\,\hat y=h(-\chi_+).\eqlab{hatqw3}
\end{align}
for $r_3=0$. The unstable manifold of \eqref{hatqw3} for \eqref{BarEps1Kappa3Eqns} is the union of the two sets
\begin{align}
 \widehat Q_3^{1,a}:\quad &r_3=\epsilon_3=0,\,z_3=-\chi_+,\,\hat y\ge h(-\chi_+),\eqlab{hatQ31a}\\
 \widehat Q_3^{1,b}:\quad &r_3=\epsilon_3=0,\,z_3=-\chi_+,\,\hat y\le h(-\chi_+).\eqlab{hatQ31b}
\end{align}
Using the initial conditions at $\widehat{\Gamma}_{\textnormal{out},3}\subset \{\epsilon_3=\mu^2\}$ (using the coordinate change in \eqref{kappa32}) it is then relatively easy to finish the proof of \propref{step1}, e.g. by using estimation after transformation into Fenichel-like normal form (by straightening out unstable fibers), and follow the forward flow up/down to the section $\widehat{\Lambda}_{\text{out},3}\subset \{\hat y=\pm \delta^{-1}\}$ in cases (a) and (b), respectively. The result shows that the forward orbits in $(\bar \epsilon=1,\kappa_3)$ follow the union of ${\hat \gamma}_3^w$ within $\{r_3=0\}$ and: $\widehat Q_3^{1,a}$ in case (a) or $\widehat Q_3^{1,b}$ in case (b) within $\{r_3=\epsilon_3=0\}$ as $\epsilon\rightarrow 0$.
\end{proof}


\subsection{Case (a)}\seclab{casea}
We now focus attention on case (a) of \defnref{caseiii}, the simpler of the two cases. We shall be able to complete the proof of \thmref{mainThm} for case (a) in this section. We start at $\widehat{\mathcal L}_{3,\epsilon}(\Pi_{in})\subset \{\hat y=\delta^{-1}\}$. Note that $\hat y$ is increasing on $\widehat Q_3^{1,a}\backslash\{\hat y=h(-\chi_+)\}$. Therefore to follow $\widehat Q_3^{1,a}$ forward we move from chart $(\bar \epsilon =1, \kappa_3)$ to chart $(\bar y=1,\kappa_3)$ using the coordinate change \eqref{cc}.
In chart $(\bar y=1,\kappa_3)$, we obtain the following equations from \eqref{xyzhatEps0Pos}:
\begin{align}
 \dot r_3 &=r_3 y_3 J(\hat \epsilon),\eqlab{BarY1Kappa3Eqns}\\
 \dot y_3 &=y_3 \left( K_3(z_3,\hat \epsilon)-2y_3 J(\hat \epsilon)\right),\nonumber\\
 \dot z_3 &=y_3 \left(L(\hat \epsilon)-z_3 J(\hat \epsilon)\right),\nonumber\\
 \dot{\hat \epsilon}&=-\hat \epsilon K_3(z_3,\hat \epsilon),\nonumber
\end{align}
where
\begin{align*}
 J(\hat \epsilon) &=\frac12 \beta^{-1} c(1+\phi_+(\hat \epsilon)) - \frac12 (1-\phi_+(\hat \epsilon)),\\
 K_3(z_3,\hat \epsilon)&=\frac12 bz_3 (1+\phi_+(\hat \epsilon)) - \frac12 \beta (1-\phi_+(\hat \epsilon)),\\
 L(\hat \epsilon) &=\frac12 (1+\phi_+(\hat \epsilon))+\frac12 b^{-1}\gamma (1-\phi_+(\hat \epsilon)),
\end{align*}
so that
\begin{align}
 J(0) &=\beta^{-1} c,\quad 
 K_3(z_3,0)=bz_3,\quad
 L(0) =1,\eqlab{JK3L0}
\end{align}
using $\phi_+(0)=1$. In the $(\bar y=1,\kappa_3)$-chart, $\widehat Q_3^{1,a}$ from \eqref{hatQ31a} becomes
\begin{align*}
 \widehat Q_3^{1,a}:\quad r_3=y_3=0,\,z_3=-\chi_+,\,\text{and}\,&\hat \epsilon 
 \in [0,h(-\chi_+)^{-1} ] \,\, \text{for}\,\,h(-\chi_+)>0, \\
 \text{or}\,&\hat \epsilon \in [0,\infty) \,\, \text{for}\,\,h(-\chi_+)\le 0,
\end{align*}
(extending it to $\hat \epsilon=0$). 
Then we have
\begin{lemma}\lemmalab{caseAFinalLemma}
The set 
\begin{align*}
 \widehat M_3^+:\quad r_3\ge 0,\,y_3 = 0,\,z_3\in  (0,\infty),\,\hat \epsilon = 0,
\end{align*}
is a set of critical points of \eqref{BarY1Kappa3Eqns} of saddle-type. The linearization about any point $(r_3,y_3,z_3,\hat \epsilon)\in \widehat M_3^+$ has only two non-zero eigenvalues
\begin{align*}
 \pm b z_3,
\end{align*}
with corresponding eigenvectors:
\begin{align}
 \begin{pmatrix}
  \beta^{-1} c r_3\\
  b z_3\\
  1-\beta^{-1} cz_3\\
  0
 \end{pmatrix},\,\begin{pmatrix}
  0\\
  0\\
  0\\
  1
 \end{pmatrix},\eqlab{M3PlusEigVec}
\end{align}
respectively.
Let 
\begin{align*}
 \hat q_3^{1,a}=(0,0,-\chi_+,0)\in \widehat M_3^+\cap \{r_3=0\}.
\end{align*}
Then $\widehat Q_3^{1,a}$ is the stable manifold of $\hat q_3^{1,a}$. On the other hand, the unstable manifold of $\hat q_3^{1,a}$ is 
\begin{align}
\widehat Q_3^{2,a} &= \bigg\{(r_3,y_3,z_3,\hat \epsilon)\vert r_3=\hat \epsilon=0,\, z_3 = c^{-1} \beta\left( -1+e^{-\beta^{-1} cs} \left(1 +\beta^{-1} c\chi_+\right)  \right) ,\nonumber\\
 &y_3 =  \frac12 c^{-2} b\beta^2 \left( \left(1+2\beta^{-1} c\chi_+ \right) e^{2\beta^{-1} cs}+1  - 2\left( 1+ \beta^{-1} c\chi_+ \right) e^{-\beta^{-1} cs}\right),\nonumber\\
 &\textnormal{for}\,s\in [0,\infty)\bigg\},\eqlab{widehatQ32p}
\end{align}
tangent to the vector $(0,
  -b\chi_+,
  1+\beta^{-1} c\chi_+,  0)^T$, see \eqref{M3PlusEigVec} with $r_3=0$, at $\hat q_3^{1,a}$. 

$\widehat Q_3^{2,a}$ is contained inside the $3D$ stable manifold of the hyperbolic equilibrium
\begin{align}
 \widehat u_3 = (0,\frac12 c^{-2}b \beta^2, c^{-1}\beta,0).\eqlab{widehatu3}
\end{align}
 The linearization of $\widehat u_3$ has $3$ negative eigenvalues: $-c^{-1} b\beta,-c^{-1} b\beta,-\frac12 -c^{-1} b\beta$ and $1$ positive eigenvalue $\frac12 c^{-1} b\beta$. The associated $1D$ unstable manifold is
\begin{align}
\widehat{U}_3 = \{(r_3,y_3,z_3,\hat \epsilon) \vert \hat \epsilon = 0,\,r_3\ge 0,\, y_3 = \frac12 c^{-2}b \beta^2,\,z_3= c^{-1}\beta\}.\eqlab{widehatU3}
\end{align}
%

\end{lemma}
\begin{proof}
 Straightforward calculations. In particular, we obtain $\widehat Q_3^{2,a}$ by setting $r_3=\hat \epsilon=0$ in \eqref{BarY1Kappa3Eqns}:
 \begin{align}
  \dot y_3 &=K_3(z_3)-2y_3 J(0)=bz_3 -2y_3 \beta^{-1}c,\eqlab{y3z3Eqns}\\
 \dot z_3 &=L(0)-z_3 J(0)=1-z_3\beta^{-1}c,\nonumber
 \end{align}
 using \eqref{JK3L0}, after division of the right hand side by $y_3$. Solving this linear system with the initial conditions
 \begin{align*}
  y_3(0) = 0,\,z_3(0) = -\chi_+,
 \end{align*}
 gives the parametrization  in \eqref{widehatQ32p} by the time $s$ in \eqref{y3z3Eqns}. Letting $s\rightarrow \infty$ gives \eqref{widehatu3}. 
\end{proof}
We illustrate the results in \lemmaref{caseAFinalLemma} in \figref{M3Plus}. Here we use an artistic sketch of the four dimensions. See caption for details.

\begin{figure}[h!] 
\begin{center}
{\includegraphics[width=0.6\textwidth]{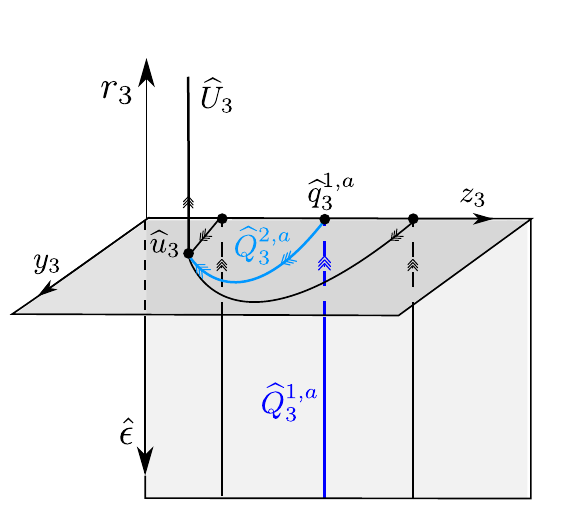}}
\end{center} 
 \caption{\ed{Schematic figure of the result in \lemmaref{caseAFinalLemma}. $\widehat Q_3^{1,a}$ (blue) is a subset of the $r_3=y_3=0$ plane (light grey), whereas $\widehat Q_3^{2,a}$ (cyan) is a subset of the $r_3=\hat \epsilon=0$ plane (darker grey). The point $\widehat u_3$ is hyperbolic with a $1D$ unstable manifold $\widehat U_3$ contained within the subset defined by $\hat \epsilon=0$}. }
\figlab{M3Plus}
\end{figure}

By blowing back down to chart $\bar y=1$ and the variables $(x,y,z,\hat \epsilon)$, we realize (see \appref{lemma1}) that $\widehat U_3$ becomes
$\widehat U$, as desired.  Then it is possible to guide the image $\widehat{\mathcal L}_{3,\text{out}}$ along $\widehat{Q}_3^{1,a}$, $\widehat{Q}_3^{2,a}$ and finally $\widehat U_3$ for $0<\epsilon\ll 1$. This is done by estimation of two transition maps: One near the normally hyperbolic set $\widehat M_3^+$ and one near $\widehat u_3$. Near $\widehat M_3^+$ we perform the estimation using a Fenichel-like normal form (by straightening out unstable fibers) and the fact that $\epsilon=r_3^2y_3\hat \epsilon$ in $(\bar y=1,\kappa_3)$. On the other hand, near the hyperbolic, but resonant, point $\widehat u_3$ we perform the estimation using a $C^1$ linearization in a neighbourhood of $\widehat u_3$ within the $\hat \epsilon=r_3=0$ subsystem (similar to \cite[Proposition 2.11]{krupa_extending2_2001}). Finally, upon blowing back down, we complete the proof of \thmref{mainThm} in case (a). We omit the details.


\subsection{Case (b)}\seclab{caseb}
We now turn our attention to case (b) of \defnref{caseiii} starting at $\widehat{\mathcal L}_{3,\epsilon}(\Pi_{in})\subset \{\hat y=-\delta^{-1}\}$. For this we note that $\hat y$ decreases along $\widehat Q_3^{1,b}\backslash\{\hat y=h(-\chi_+)\}$, see \eqsref{hatQ31b}{BarY1Kappa3Eqns}. Therefore to follow $\widehat Q_3^{1,b}$ forward we move from chart $(\bar \epsilon =1, \kappa_3)$ to chart $(\bar y=-1,\kappa_3)$ using the coordinate change in \eqref{cc}.

\subsubsection*{Chart $(\bar y=-1,\kappa_3)$}
Here we obtain the following equations from \eqref{xyzhatEps0Neg}:
\begin{align}
\dot r_3 &=-r_3 y_3 \Psi(\hat \epsilon),\eqlab{BarYN1Kappa3Eqns}\\
\dot y_3&=-y_3(\Delta_3(z_3,\hat \epsilon)-2y_3\Psi(\hat \epsilon)),\nonumber\\
\dot z_3&=-y_3(\Omega(\hat \epsilon)-z_3 \Psi(\hat \epsilon)),\nonumber\\
\dot{\hat \epsilon} &=\hat \epsilon \Delta_3(z_3,\hat \epsilon),\nonumber
\end{align}
where
\begin{align*}
 \Psi(\hat \epsilon) &=\frac12 \beta^{-1}c (1+\phi_-(\hat \epsilon))-\frac12 (1-\phi_-(\hat \epsilon)),\\
 \Delta_3(z_3,\hat \epsilon)&=\frac12 bz_3 (1+\phi_-(\hat \epsilon))-\frac12 \beta (1-\phi_-(\hat \epsilon)),\\
 \Omega(x,z,\hat \epsilon)&=\frac12 (1+\phi_-(\hat \epsilon))+\frac12 b^{-1}\gamma (1-\phi_-(\hat \epsilon)),
\end{align*}
so that
\begin{align}
 \Psi(0) &= -1,\quad \Delta(z_3,0)=-\beta,\quad \Omega(0) = b^{-1}\gamma,\eqlab{JK3L0New}
\end{align}
using that $\phi_-(0)=-1$. In the $(\bar y=-1,\kappa_3)$-chart, $\widehat Q_3^{1,b}$ from $(\bar \epsilon=1,\kappa_3)$, see \eqref{hatQ31b}, becomes
\begin{align*}
 \widehat Q_3^{1,b}:\quad r_3=y_3=0,\,z_3=-\chi_+,\,\text{and}\,&\hat \epsilon 
 \in [h(-\chi_+)^{-1} ,0] \,\, \text{for}\,\,h(-\chi_+)<0, \\
 \text{or}\,&\hat \epsilon \in (-\infty,0] \,\, \text{for}\,\,h(-\chi_+)\ge  0,
\end{align*}
(extending it to $\hat \epsilon=0$). Notice that $\widehat Q_3^{1,b}$ from $(\bar \epsilon=1,\kappa_3)$ is only partially covered by the $(\bar y=-1,\kappa_3)$-chart for $h(-\chi_+)\ge  0$. But as promised, we will continue to use the same symbol for this object in the new chart. 
\begin{lemma}
The set 
\begin{align*}
 \widehat M_3^-:\quad r_3\ge 0,\,y_3 = 0,\,z_3\in \mathbb R,\,\hat \epsilon = 0,
\end{align*}
is a set of critical points of \eqref{BarYN1Kappa3Eqns} of saddle-type. The linearization about any point $(r_3,y_3,z_3,\hat \epsilon)\in \widehat M_3^-$ has only two non-zero eigenvalues
\begin{align*}
 \pm \beta,
\end{align*}
with corresponding eigenvectors:
\begin{align}
 \begin{pmatrix}
  -r_3\\
  -\beta \\
 z_3+b^{-1}\gamma\\
  0
 \end{pmatrix},\,\begin{pmatrix}
  0\\
  0\\
  0\\
  1
 \end{pmatrix},\eqlab{M3MinusEigVec}
\end{align}
respectively.
Let 
\begin{align}
 \hat q_3^{1,b}=(0,0,-\chi_+,0)\in \widehat M_3^-\cap \{r_3=0\}.\eqlab{hatq31b}
\end{align}
Then $\widehat Q_3^{1,b}$ is the stable manifold of $\hat q_3^{1,b}$. On the other hand, the unstable manifold of $\hat q_3^{1,b}$ is 
\begin{align}
 \widehat Q_3^{2,b} &= \bigg\{(r_3,y_3,z_3,\hat \epsilon)\vert r_3=\hat \epsilon=0,\,z_3 = -b^{-1} \left(\gamma-(\gamma-b\chi_+)\sqrt{-2\beta^{-1} y_3+1}\right),\,y_3\le 0\bigg\},\eqlab{widehatQ32n}
\end{align}
tangent to the vector $(0,-\beta, -\chi_++\gamma,0)^T$, see \eqref{M3MinusEigVec} with $r_3=0$, at $\hat q_3^{1,b}$. 

%

\end{lemma}
\begin{proof}
 Straightforward calculations. In particular, we obtain $\widehat Q_3^{2,b}$ by setting $r_3=\hat \epsilon=0$ in \eqref{BarY1Kappa3Eqns}:
 \begin{align}
  \dot y_3 &=\Delta_3(z_3,0)-2y_3\Psi(0)=-\beta + 2y_3,\eqlab{y3z3Eqnsn}\\
 \dot z_3 &=\Omega(0)-z_3 \Psi(0)=b^{-1}\gamma+z_3,\nonumber
 \end{align}
 using \eqref{JK3L0New}, after division of the right hand side by $-y_3$. Solving this linear system with the initial conditions
 \begin{align*}
  y_3(0) = 0,\,z_3(0) = -\chi_+,
 \end{align*}
 gives the parametrization  in \eqref{widehatQ32n} upon elimination of time.
\end{proof}

We sketch the dynamics within $r_3=0$ in \figref{M3Neg}, illustrating the segments $\widehat Q_3^{1,b}$ and $\widehat Q_3^{2,b}$. 
\begin{figure}[h!] 
\begin{center}
{\includegraphics[width=0.6\textwidth]{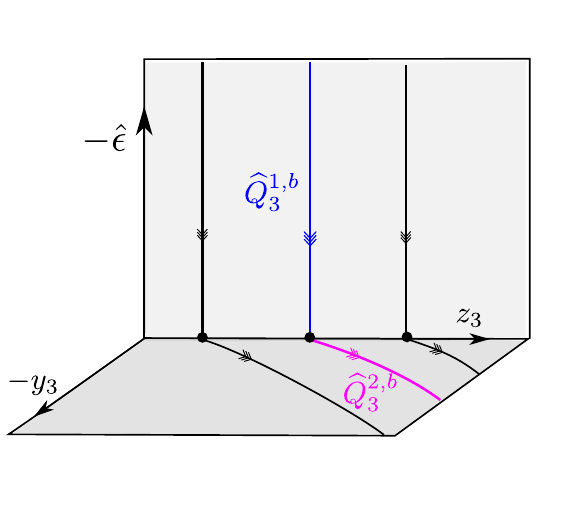}}
\end{center} 
 \caption{\ed{Dynamics within $r_3=0$}. }
\figlab{M3Neg}
\end{figure}

Now, $y_3$ decreases unboundedly along $\widehat Q_3^{2,b}$ in \eqref{widehatQ32n}. Working in chart  $(\bar y=-1,\kappa_2)$, we can follow $\overline{\widehat Q}^{2,b}$ into chart $(\bar y=-1,\kappa_1)$.

\subsubsection*{Chart $(\bar y=-1,\kappa_1)$}
 In this chart, $\overline{\widehat Q}^{2,b}$ becomes
\begin{align}
 \widehat Q_1^{2,b} &= \bigg\{(r_1,y_1,z_1,\hat \epsilon)\vert r_1=\hat \epsilon=0,\,z_1 = -b^{-1} \left(\gamma-(\gamma-b\chi_+)\sqrt{2\beta^{-1} y_1+1}\right),\,y_1\le 0\bigg\},\eqlab{widehatQ12n}
\end{align}
(extending it to $y_1=0$).
%
%
 We then have
\begin{lemma}
 In chart $(\bar y=-1,\kappa_1)$, the set 
\begin{align*}
 \widehat M_1^-:\quad r_1\ge 0,\,y_1 = 0,\,z_1\in \mathbb R,\,\hat \epsilon = 0,
\end{align*}
is a set of critical points of \eqref{BarYN1Kappa3Eqns} of saddle-type. The linearization about any point $(r_1,y_1,z_1,\hat \epsilon)\in \widehat M_1^-$ has only two non-zero eigenvalues
\begin{align*}
 \mp \beta,
\end{align*}
with corresponding eigenvectors:
\begin{align}
 \begin{pmatrix}
  \beta r_1\\
  -\beta \\
  z_1-b^{-1}\gamma\\
  0
 \end{pmatrix},\,\begin{pmatrix}
  0\\
  0\\
  0\\
  1
 \end{pmatrix},\eqlab{M1MinusEigVec}
\end{align}
respectively.
Let 
\begin{align}
 \hat q_1^{2,b}=(0,0,z_1^*,0)\in \widehat M_1^-\cap \{r_1=0\}.\eqlab{hatq12b}
\end{align}
Then $\widehat Q_1^{2,b}$ \eqref{widehatQ12n} is the stable manifold of $\hat q_1^{2,b}$, tangent to the vector $(0,-\beta,z_1^*-b^{-1}\gamma,0)^T$, see \eqref{M1MinusEigVec} with $r_1=0$, at $\hat q_1^{2,b}$. On the other hand, the (local) unstable manifold of $\hat q_1^{2,b}$ is 
\begin{align}
 \widehat Q_1^{3,b} &= \bigg\{(r_1,y_1,z_1,\hat \epsilon)\vert r_1=y_1= 0,\,z_1 = z_1^*,\,\hat \epsilon\in [-\theta,0]\},\eqlab{Q13b}
\end{align}
$\theta>0$ sufficiently small, 
tangent to the vector $(0,0,0,1)^T$, see \eqref{M1MinusEigVec}, at $\hat q_1^{2,b}$. 
\end{lemma}
\begin{proof}
 In $(\bar y=-1,\kappa_1)$, we have from \eqref{xyzhatEps0Neg} that
\begin{align}
\dot r_1 &=r_1 y_1 \Psi(\hat \epsilon),\eqlab{BarYN1Kappa1Eqns}\\
\dot y_1&=-y_1(\Delta_1(z_1,\hat \epsilon)+2y_1\Psi(\hat \epsilon)),\nonumber\\
\dot z_1&=-y_1(\Omega(\hat \epsilon)+z_1 \Psi(\hat \epsilon)),\nonumber\\
\dot{\hat \epsilon} &=\hat \epsilon \Delta_1(z_1,\hat \epsilon),\nonumber
\end{align}
where
\begin{align*}
 \Delta_1(z_1,\hat \epsilon)&=\frac12 bz_1 (1+\phi_-(\hat \epsilon))+\frac12 \beta (1-\phi_-(\hat \epsilon)),\\
\end{align*}
so that
\begin{align*}
 \Delta_1(z_1,0)=\beta,
\end{align*}
using that $\phi_-(0)=-1$. We then obtain the result through simple calculations. 
\end{proof}
\begin{remark}
 Notice that $\overline{\overline Q}^{2,b}$ is a heteroclinic connection on the sphere $\overline{\overline q}\cap \{(\bar y,\bar \epsilon)=(-1,0)\}$, connecting $\hat q_3^{1,b}$ \eqref{hatq31b} in chart $(\bar y=-1,\kappa_3)$ with $\hat q_1^{2,b}$ \eqref{hatq12b} in $(\bar y=-1,\kappa_1)$. This connection is in agreement with \lemmaref{sigma0} and \eqref{Dsigmav}. Indeed, the linear mapping $D\vartheta(0)$ in \lemmaref{sigma0}, maps $v_+$ to $(-1,\chi_+)^T$. In the projective coordinates, $z_3=x^{-1} z,\,x>0$ and $z_1=-x^{-1}z,\,x<0$, this assignment becomes $z_3=-\chi_+\mapsto z_1 = z_1^*$, or simply $\hat q_3^{1,b}\mapsto \hat q_1^{2,b}$ for $r=0$.  
\end{remark}

We sketch the dynamics within $r_1=0$ in \figref{M1Neg}, illustrating the segments $\widehat Q_1^{2,b}$ and $\widehat Q_1^{3,b}$. 
\begin{figure}[h!] 
\begin{center}
{\includegraphics[width=0.6\textwidth]{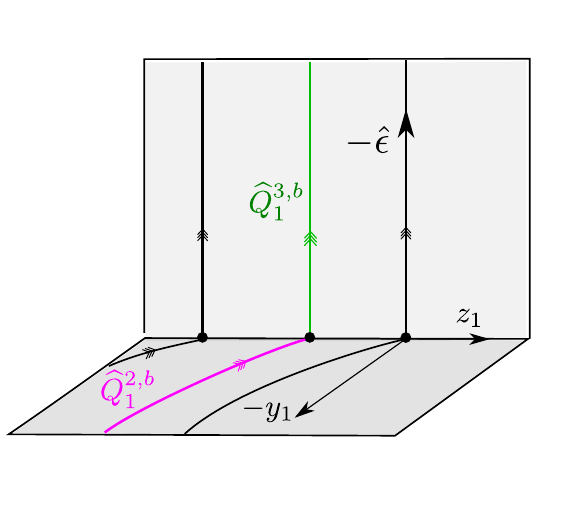}}
\end{center} 
 \caption{\ed{Dynamics within $r_1=0$}. }
\figlab{M1Neg}
\end{figure}

The variable $\hat \epsilon\le 0$ decreases along $\widehat Q_1^{3,b}$. 
Following the discussion proceeding \lemmaref{sigma0}, the extension of this manifold by the forward flow now depends on the sign of $z_1^*$. The details of the separate cases $z_1^*<0$, $z_1^*=0$ and $z_1^*>0$ are given in \appref{z1StarPos}, \appref{z1Star0} and \appref{z1StarNeg}, respectively. The case $z_1^*>0$ is similar to case (a). The case $z_1^*\le 0$ is more involved due to the nonhyperbolicity of 
\begin{align*}
 \hat q_1^{4,b}:\quad r_1=y_1 = z_1=\hat \epsilon=0,
\end{align*}
in chart $(\bar y=1,\kappa_1)$. Using the coordinates in the chart $(\bar y=1,\kappa_1)$, this point corresponds to the intersection of the nonhyperbolic line (of visible folds) $\bar l^+$ (see \eqref{hatlPlus} and \lemmaref{hatMP}) with the $\bar x<0$ subset of the blown up two-fold $\overline{\overline q}$. Therefore we will have to blowup this point in \appref{z1StarPos} to obtain a complete, hyperbolic, singular picture. 


\section{Discussion and conclusion}\seclab{discuss}
In this paper, we consider the PWS visible-invisible two-fold in the truncated, piecewise linear, normal form \bpwl{} satisfying assumption (A) as a singular limit of the regularized system \brpwl{}. We restrict attention to the regularization function $\phi(s) = \frac{2}{\pi}\arctan(s)$ and assume a non-degeneracy condition (B). Then our main result \thmref{mainThm} states that, as the regularized system tends to the PWS system, there is a distinguished forward trajectory $U$, shown in \figref{VISliding}, among all the candidates leaving the two-fold. 

Our approach to the problem is new, because we combine two separate blowups\footnote{Consecutive blowups can be used to study other singular perturbation phenomenon in different regularizations of piecewise smooth systems, see e.g. \cite{kriBlowup}.}. The first blowup \eqref{blowup0} resolves the singularity at $y=0,\,\epsilon=0$. Then we obtain the two-fold as a circle of nonhyperbolic critical points in the blowup space. 
The second blowup \eqref{blowup1} is in the sense Dumortier, Roussarie, Krupa and Szmolyan \cite{dumortier_1993,dumortier_1996,krupa_extending_2001} used to study nonhyperbolic critical points. We blow up the circle of nonhyperbolic critical points to a circle of spheres. By selecting appropriate weights associated with the blowup, we use desingularization to gain hyperbolicity. 

It is possible to obtain our main result \thmref{mainThm} if we relax assumption (A), and replace it with $c-\gamma \le \sqrt{(c-\gamma)^2-4b\beta}$. In this case $\chi_-\ge 0$ in \eqref{chipm} and so we always have $z_1^*>0$ in \eqref{z1chiN0}. Then there is no strong canard $\gamma^s$ for the PWS system and {\em any} orbit of $X_{sl}$ is tangent to $v_+$ at $q$. In contrast, both $z_1^*>0$ and $z_1^*\le 0$ are possible for the more complicated case given by assumption (A).

The non-degeneracy condition (B) is independent of the regularization function $\phi(s)$. Together with the position of $R_{\text{in},\epsilon}$ in relation to the span of $v_+$, the parameter $\xi$ determines whether the forward flow of the regularization follows $X^+$ directly beyond $q$ (case (a) of \defnref{caseiii}) or whether a twist occurs where the forward orbit first follows $X^-$ before returning to $\Sigma^+$ and $X^+$ (case (b) of \defnref{caseiii}). The case $\xi \in \mathbb N$, which is excluded by (B), is at the boundary of these two separate cases. Here additional (secondary) canards appear (see \cite{krihog}) which complicates the analysis further.

Our result can be extended in a number of ways. For example, the result holds true for other regularization functions, including the Sotomayor and Teixeira regularization functions, see \defnref{STphi}. This involves only minor modifications. In fact, for the Sotomayor and Teixeira regularization functions, the scaling chart \eqref{yhat} associated with the blowup \eqref{blowup0} is (by \eqref{XepsProp}) enough to prove the theorem. We can also easily extend the result to other non-Sotomayor and Teixeira regularization functions that satisfy: There exists a smallest $k_+ \in \mathbb N$ ($k_-\in \mathbb N$) such that the $k_+$th-derivative ($k_-$-derivative) of $\phi_+$ ($\phi_-$, respectively) is non-zero at $u=0$: $\phi_\pm^{(k_\pm)}(0)\ne 0$ where $$\phi(s) = \phi_\pm(s^{-1}),\,s\gtrless 0,$$ as in \eqref{phipmProp}. (For $\phi=\arctan$ we have $k_\pm = 1$.) The resulting algebraic decay of $\phi$ at $\pm \infty$ enables us to extend \appref{z1StarPos} and the blowup of $\hat q_1^{4,b}$ fairly easy (we just have to change the weights in \eqref{finalblowup}). But for the regularization function $\phi(s) = \tanh(s)$ we have $\phi_\pm^{(k)}(0)=0$ for every $k\in \mathbb N$. To blowup $\hat q_1^{4,b}$ for this regularization function, one way forward would be to use the approach in \cite{kriBlowup} for blowup of flat slow manifolds.



%





We could also replace the PWL system \bpwl{} with the full nonlinear PWS system \bpws{}. In fact, the dynamics on $\overline{\overline q}$ is completely unchanged if we replace \bpwl{} by \bpws{}. We still obtain $\overline{\overline Q}^{j,a}$ and $\overline{\overline Q}^{j,b}$ with identical hyperbolicity properties. 
In \cite[Theorem 7.1]{krihog}, we showed that if the non-degeneracy condition (B) holds then the lines $\hat{\gamma}^{w,s}$ in \lemmaref{lines} perturb into a
\textit{weak canard} $\hat{\gamma}^w(\epsilon)$ and a \textit{strong canard} $\hat{\gamma}^s(\epsilon)$, respectively, for $\epsilon>0$ sufficiently small. 
These orbits are transverse intersections of extended versions of the (now non-unique) Fenichel slow manifolds $\widehat S_{a,\epsilon}$ and $\widehat S_{r,\epsilon}$, similar to \eqsref{Saeps1}{Sreps1}. 
Their projections onto the $(x,z)$-plane have tangents at $(x,z)=0$ that are $\mathcal O(\sqrt{\epsilon})$-close to the eigenvectors strong/weak eigenvectors $v_{\pm}$. 
For the regularization of the full nonlinear PWS system \bpws{}, the strong canard tends to the unique solution of the sliding equations that are tangent to the strong eigenvector $v_-$ at the two-fold, as $\epsilon\rightarrow 0$; compare \propref{pwsCanards}(a). But the limit of the weak canard $\hat{\gamma}^w(\epsilon)$ is more complicated. There is a whole funnel of singular weak canard candidates, recall \propref{pwsCanards}(b), that $\gamma^w(\epsilon)$ can limit to. Hence a priori, for general initial conditions within the funnel, it is impossible to determine on what side of the canard the initial conditions belong to. This is important for the generalisation of our results to the regularization of \bpws{}. To handle this, we propose to add a condition of the form
\begin{itemize}
\item[(D)] There exists a $K>0$ sufficiently large so that 
\begin{align*}\text{dist}\,(\Pi_{\text{in},\epsilon},\hat{\gamma}^w(\epsilon))\ge K^{-1}>0,\end{align*} for all $0<\epsilon\le \epsilon_0$ sufficiently small.
\end{itemize}
Unfortunately, such a condition is implicit. In particular, condition (D) will, interestingly, most likely depend upon the choice of regularization function. We do not need condition (D) when we use the truncation \bpwl{} because there the weak canard $\hat{\gamma}^w$ \eqref{gammaw} is explicitly known and independent of $\epsilon$. Similar issues arise with weak canards of folded nodes in standard slow-fast systems in $\mathbb R^3$, see \cite{brons-krupa-wechselberger2006:mixed-mode-oscil, Vo}. The authors of \cite{brons-krupa-wechselberger2006:mixed-mode-oscil} also (implicitly) assume \cite{morten} a condition like (D) in their Theorem 4.1. 



\subsection*{Acknowledgement}We would like to thank an anonymous referee whose many suggestions have greatly improved the manuscript.

\bibliography{refs}

\begin{thebibliography}{10}

\bibitem{morten}
M.~Br{\o}ns.
\newblock Private communication.
\newblock 2015.

\bibitem{brons-krupa-wechselberger2006:mixed-mode-oscil}
M.~Br{\o}ns, M.~Krupa, and M.~Wechselberger.
\newblock Mixed mode oscillations due to the generalized canard phenomenon.
\newblock In W.~Nagata and N.~Sri Namachchivaya, editors, {\em Bifurcation
  Theory and Spatio-Temporal Pattern Formation}, volume~49 of {\em Fields
  Institute Communications}, pages 39--64. American Mathematical Society, 2006.

\bibitem{car1}
J.~Carr.
\newblock {\em Applications of centre manifold theory}, volume~35.
\newblock New York: Springer-Verlag, 1981.

\bibitem{colombojeffrey2011}
A.~Colombo and M.~R. Jeffrey.
\newblock {Nondeterministic chaos, and the two-fold singularity in piecewise
  smooth flows}.
\newblock {\em {SIAM Journal on Applied Dynamical Systems}},
  {10}({2}):{423--451}, {2011}.

\bibitem{desroches_canards_2011}
M.~Desroches and M.~R. Jeffrey.
\newblock Canards and curvature: nonsmooth approximation by pinching.
\newblock {\em Nonlinearity}, 24(5):1655--1682, May 2011.

\bibitem{Bernardo08}
M.~di~Bernardo, C.~J. Budd, A.~R. Champneys, and P.~Kowalczyk.
\newblock {\em Piecewise-smooth Dynamical Systems: Theory and Applications}.
\newblock Springer Verlag, 2008.

\bibitem{dumortier_1991}
F.~Dumortier.
\newblock Local study of planar vector fields: Singularities and their
  unfoldings.
\newblock In H.~W.~Broer et~al, editor, {\em Structures in Dynamics, Finite
  Dimensional Deterministic Studies}, volume~2, pages 161--241. Springer
  Netherlands, 1991.

\bibitem{dumortier_1993}
F.~Dumortier.
\newblock Techniques in the theory of local bifurcations: Blow-up, normal
  forms, nilpotent bifurcations, singular perturbations.
\newblock In Dana Schlomiuk, editor, {\em Bifurcations and Periodic Orbits of
  Vector Fields}, volume 408 of {\em NATO ASI Series}, pages 19--73. Springer
  Netherlands, 1993.

\bibitem{dumortier_1996}
F.~Dumortier and R.~Roussarie.
\newblock Canard cycles and center manifolds.
\newblock {\em Mem. Amer. Math. Soc.}, 121:1--96, 1996.

\bibitem{fen1}
N.~Fenichel.
\newblock Persistence and smoothness of invariant manifolds for flows.
\newblock {\em Indiana University Mathematics Journal}, 21:193--226, 1971.

\bibitem{fen2}
N.~Fenichel.
\newblock Asymptotic stability with rate conditions.
\newblock {\em Indiana University Mathematics Journal}, 23:1109--1137, 1974.

\bibitem{filippov1988differential}
A.F. Filippov.
\newblock {\em Differential Equations with Discontinuous Righthand Sides}.
\newblock Mathematics and its Applications. Kluwer Academic Publishers, 1988.

\bibitem{gomory1955a}
R.~E. Gomory.
\newblock Trajectories tending to a critical point in 3-space.
\newblock {\em Annals of Mathematics}, 61(1):140--153, 1955.

\bibitem{guglielmi2017a}
N.~Guglielmi and E.~Hairer.
\newblock Solutions leaving a codimension-2 sliding.
\newblock {\em Nonlinear Dynamics}, 88(2):1427--1439, 2017.

\bibitem{jeffrey_geometry_2011}
M.~R. Jeffrey and S.~J. Hogan.
\newblock The geometry of generic sliding bifurcations.
\newblock {\em {SIAM} Review}, 53(3):505--525, January 2011.

\bibitem{kriBlowup}
K.~Uldall Kristiansen.
\newblock Blowup for flat slow manifolds.
\newblock {\em Nonlinearity}, 30:2138--2184, 2017.

\bibitem{krihog}
{K. Uldall} Kristiansen and {S. J.} Hogan.
\newblock {On the use of blowup to study regularizations of singularities of
  piecewise smooth dynamical systems in $\mathbb{R}^3$}.
\newblock {\em SIAM Journal on Applied Dynamical Systems}, 14(1):382--422,
  2015.

\bibitem{krihog2}
{K. Uldall} Kristiansen and {S. J.} Hogan.
\newblock Regularizations of two-fold bifurcations in planar piecewise smooth
  systems using blowup.
\newblock {\em SIAM Journal on Applied Dynamical Systems}, 14(4):1731--1786,
  2015.

\bibitem{krupa_extending_2001}
M.~Krupa and P.~Szmolyan.
\newblock Extending geometric singular perturbation theory to nonhyperbolic
  points - fold and canard points in two dimensions.
\newblock {\em {SIAM} Journal on Mathematical Analysis}, 33(2):286--314, 2001.

\bibitem{krupa_extending2_2001}
M.~Krupa and P.~Szmolyan.
\newblock Extending slow manifolds near transcritical and pitchfork
  singularities.
\newblock {\em Nonlinearity}, 14(6):1473, 2001.

\bibitem{kuehn2015}
C.~Kuehn.
\newblock {\em {Multiple Time Scale Dynamics}}.
\newblock Springer-Verlag, Berlin, 2015.

\bibitem{MakarenkovLamb12}
O.~Makarenkov and J.~S.~W. Lamb.
\newblock Dynamics and bifurcation of nonsmooth systems: {A} survey.
\newblock {\em Physica D}, 241:1826--1844, 2012.

\bibitem{simpson2014a}
D.~J.~W. Simpson.
\newblock On resolving singularities of piecewise-smooth discontinuous vector
  fields via small perturbations.
\newblock {\em Discrete and Continuous Dynamical Systems}, 34(10):3803--3830,
  2014.

\bibitem{Sotomayor96}
J.~Sotomayor and M.~A. Teixeira.
\newblock Regularization of discontinuous vector fields.
\newblock In {\em Proceedings of the International Conference on Differential
  Equations, Lisboa}, pages 207--223, 1996.

\bibitem{szmolyan_canards_2001}
P.~Szmolyan and M.~Wechselberger.
\newblock Canards in $\mathbb{R}^3$.
\newblock {\em J. Diff. Eq.}, 177(2):419--453, December 2001.

\bibitem{takens1974a}
F.~Takens.
\newblock Singularities of vector fields.
\newblock {\em Publications Math\'{e}matiques de L'institut des Hautes
  \'{E}tudes Scientifiques}, 43(1):47--100, 1974.

\bibitem{Vo}
T.~Vo, R.~Bertram, and M.~Wechselberger.
\newblock Bifurcations of canard-induced mixed mode oscillations in a pituitary
  lactotroph model.
\newblock {\em Discrete and Continuous Dynamical Systems}, 32(8):2879--2912,
  2012.

\end{thebibliography}
\bibliographystyle{plain}
\newpage
\appendix

\section{Proof of \lemmaref{sigma0}}\applab{lemma0}
$\vartheta$ is well-defined since $l^-= \{(x,y,z)\in \mathbb R^3\vert x=y=0\}$ is an invisible fold line. See \figref{VISliding}. 
 The first part of the result therefore follows from simple calculations. 
 \eqref{Dsigmav} is obtained by differentiating \eqref{eq.sigma0Neg}. The inequality for $z_1^*$ in \eqref{z1chiN0} is obtained from \eqref{chipm}, the positivity of $b$ and using (A). Indeed
 \begin{align*}
 b(z_1^*- \chi_-)=2\gamma + c-\gamma = c+\gamma>\sqrt{(c-\gamma)^2-4b\beta}>0 
 \end{align*}
using \eqref{cgamma} in the last two inequalities.  
\section{Details for the proof of \propref{WidetildeX}}\applab{appHALLO}
\subsection{The chart $\bar y=-1$}\seclab{BarYNeg1}
Inserting \eqref{chartEpshat} into \brpwlf{} (together with the trivial equation $\epsilon'= 0$), we obtain the following equations using \eqref{phipm}:
\begin{align}
x'&= -y (\beta^{-1} c(1+\phi_-(\hat \epsilon))-(1-\phi_-(\hat \epsilon))),\eqlab{xyzhatEps0Neg}\\
y'&= -y (b {z} (1+\phi_-(\hat \epsilon)) -\beta  x(1-\phi_-(\hat \epsilon))),\nonumber\\
 z'&=-y  ((1+\phi_-(\hat \epsilon))+b^{-1} \gamma(1-\phi_-(\hat \epsilon))),\nonumber\\
\hat \epsilon' & =\hat \epsilon (b {z} (1+\phi_-(\hat \epsilon)) -\beta  x(1-\phi_-(\hat \epsilon))),\nonumber
\end{align}
where $(y,\hat \epsilon) \in (-\infty,0]^2$, after division of the right hand side by the common factor $-\hat \epsilon$.  

\begin{remark}
The flow of system \eqref{xyzhatEps0Neg} preserves $\epsilon=y\hat \epsilon$, so $\epsilon=0$ implies either $y=0$ or $\hat \epsilon=0$. The corresponding sets $\{y=0\}$ and $\{\hat \epsilon=0\}$ are invariant. 

Within $\{\hat \epsilon=0\}$ we recover the vector field $X^-$ of \bpwl{}$_{y<0}$ from \eqref{xyzhatEps0Neg}:
\begin{align}
\dot x&=-1,\eqlab{XNegativeNew}\\
\dot y &= -\beta x,\nonumber\\
\dot z &=b^{-1}\gamma,\nonumber
\end{align}
after further division of the right hand side by $-2y< 0$, using $\phi_-(0)=-1$. 
%
%

Within $\{y=0\}$ we recover $\widehat S_{0}\cap \{\hat y<0\}$ from \propref{criticalManifold} as a set of critical points, having the same hyperbolicity properties, upon the  coordinate change \eqref{hatEpshatY}. In particular,
\begin{align}
 \hat q:\,x=z=y=0, \, \hat \epsilon\le 0,\eqlab{newqNeg}
\end{align}
(extending it to $\hat \epsilon=0$ where it remains nonhyperbolic)
and
\begin{align*}
 \widehat S_{a,r}:\, \hat \epsilon = h_-(x^{-1}z),\,\quad (x,y,z)\in \Sigma_{sl}^\pm\cap \{x^{-1}z \in (0,b^{-1} \beta)\},
\end{align*}
respectively, 
with $h_-:(0,b^{-1}\beta)\rightarrow (-\infty,0]$ defined by
\begin{align*}
  h_-(s)= \phi_-^{-1}\left(\frac{1-\beta^{-1} b s}{1+\beta^{-1} b s}\right).
\end{align*}
\end{remark}
In the chart $\bar y=-1$, along the intersection $\{y=\hat \epsilon=0\}$ of the invariant sets $\{y=0\}$ and $\{\hat \epsilon=0\}$, we obtain the following.
\begin{lemma}
 The set $\widehat M^- \equiv \{y=\hat \epsilon =0\}$ is a set of critical points of \eqref{xyzhatEps0Neg}. It is of saddle-type for $x\ne 0$: The linearization about any point in $\widehat M^-$ has only two non-trivial eigenvalues $\pm 2bx$ with associated eigenvectors
 \begin{align*}
  \begin{pmatrix}
   1\\ 
\beta x\\ 
-b^{-1} \gamma\\
0
  \end{pmatrix},\begin{pmatrix}
0\\
0\\
0\\
1
  \end{pmatrix},
 \end{align*}
respectively. The line 
\[\hat l^-=\widehat M^-\cap \{x=0\},\] is nonhyperbolic: The linearization about any point in $\hat l^-$ has only zero eigenvalues. It becomes $l^-$ upon returning to the $(x,y,z)$-variables for $\epsilon=0$.
\end{lemma}
\begin{proof}
 Simple calculations.
\end{proof}

\section{A lemma}\applab{lemma1}

\begin{lemma}
The forward orbit $U$ of $X^+$, becomes 
\begin{align}
 \widehat U &= \left\{(x,y,z,\hat \epsilon)=\left(r_3,\frac12  r_3^2 c^{-2}b \beta^2,r_3 c^{-1}\beta,0\right) \vert\, r_3\in [0,\infty)\right\}\eqlab{Uexpr3}
 \end{align}
 or equivalently
 \begin{align}
 \widehat U&=\left\{(x,y,z,\hat \epsilon)=\left(r_2\beta^{-1} c \sqrt{2b^{-1}},r_2^2 ,r_2\sqrt{2b^{-1}},0\right) \vert\, r_2\in [0,\infty)\right\},
 \eqlab{Uexpr2}
\end{align}
in the chart $\bar y=1$.
\end{lemma}
\begin{proof}
By eliminating $z$ through $x(z)=r_3$ in \eqref{Uexpr},  we obtain \eqref{Uexpr3}. Similarly, eliminating $z$ through $y(z)=r_2^2$ in \eqref{Uexpr} gives \eqref{Uexpr2}. 
\end{proof}

\section{Case (b) with $z_1^*< 0$}\applab{z1StarPos}
Consider first chart $(\bar \epsilon=1,\kappa_1)$ and the equations \eqref{BarEps1Kappa1Eqns}. Then in the case under consideration, we have $\hat y'=0$ for $\hat y=h(-z_1^*)\in \mathbb R$ within $r_1=\epsilon_1=0,\,z_1=z_1^*$. Therefore 
$\widehat Q_1^{3,b}:\, r_1=\epsilon_1= 0,\,z_1 = z_1^*,\,\hat y\in (-\infty,h(-z_1^*)]$ (extending it to $\hat y=h(-z_1^*)$), in the chart $(\bar \epsilon = 1,\kappa_1)$, is a hyperbolic fiber of
\begin{align}
 \hat q_1^{3,b}:\,r_1=\epsilon_1=0,\,z_1=z_1^*,\,\hat y=h(-z_1^*),\eqlab{hatq13b}
\end{align}
belonging to the normally hyperbolic and attracting line $\widehat L_{a,1}$ \eqref{La1}. On $\widehat L_{a,1}$, we obtain a slow flow by \eqref{reducedSaeps1}$_{\epsilon_1=0}$. Now, $z_1=\chi_-$ is an unstable node for \eqref{reducedSaeps1}$_{\epsilon_1=0}$. But then since $z_1^*>\chi_-$, recall \eqref{z1chiN0}, we have  $z_1'>0$ at \eqref{hatq13b}. We therefore obtain the following subsequent singular orbit segment
\begin{align*}
 \widehat Q_1^{4,b}:\,r_1=\epsilon_1=0,\,\hat y = h(-z_1),\,z_1 \in [z_1^*,0).
\end{align*}
The variable $\hat y$ increases unboundedly on $\widehat Q_1^{4,b}$ since $z_1$ increases by the slow flow. We therefore move to the chart $(\bar \epsilon=1,\kappa_1)$. We illustrate the results in $(\bar \epsilon=1,\kappa_1)$ within $r_1=\epsilon_1=0$ in \figref{Q4}. 

\begin{figure}[h!] 
\begin{center}
{\includegraphics[width=0.55\textwidth]{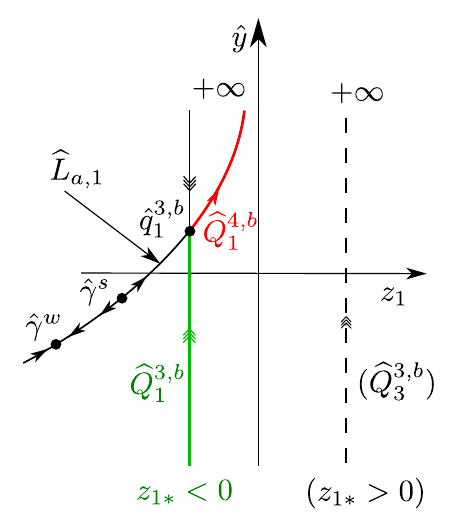}}
\end{center} 
 \caption{\ed{Dynamics within $r_1=\epsilon_1$. For $z_1^*<0$ $\widehat Q_1^{3,b}$ (green) is a hyperbolic stable fiber of a base-point on $\widehat L_{a,1}$. The base point is by \eqref{z1chiN0} contained outside the funnel region and the slow flow along $\widehat Q_{1}^{4,b}$ (red) therefore moves towards the ``visible'' fold at $z_1=0$.  Note, on the other hand, if $z_1^*\ge 0$ (considered in \appref{z1Star0} and \appref{z1StarNeg}) then we have $\hat y'>0$ along the orbit $\widehat Q_1^{3,b}$ (illustrated as dashed line for $z_1^*>0$)}.  }
\figlab{Q4}
\end{figure}

\subsubsection*{Chart $(\bar \epsilon=1,\kappa_1)$}
In this chart, the dynamics is described by \eqref{BarY1Kappa1Eqns}. Within $r_1=y_1=0$ we then re-discover the normally hyperbolic and attracting line of critical points 
\begin{align*}
 \widehat L_{a,1}:\,r_1=y_1=0,\,\hat \epsilon = h_+(-z_1),\,z_1 \in (-b^{-1} \beta,0),
\end{align*}
containing $\widehat Q_1^{4,b}$. The slow flow on $\widehat L_{a,1}$, described by \eqref{reducedSaeps1}$_{\epsilon_1=0}$, reaches the boundary point at $z_1=0$:
\begin{align}
 \hat q_1^{4,b}:\,r_1=y_1 = z_1=\hat \epsilon=0,\eqlab{qq1}
\end{align}
in finite time. The point $\hat q_1^{4,b}$ {is due to the nonhyperbolicity of $\widehat l^+$ \eqref{hatlPlus} also nonhyperbolic}. To describe the dynamics near $\hat q_1^{4,b}$, we apply the following blowup transformation $(r_1,y_1,z_1,\hat \epsilon)\mapsto (r_1,\varrho,(\bar y_1,\bar z_1,\bar{\hat \epsilon}))$ defined by:
 \begin{align}
  y_1 = \varrho^2 \bar y_1,\,z_1 = \varrho \bar z_1,\,\hat \epsilon = \varrho \bar{\hat \epsilon},\,\varrho \ge 0,\,(\bar y_1,\bar z_1,\bar{\hat \epsilon})\in S^2,\eqlab{finalblowup}
 \end{align}
 and apply desingularization through the division of the right hand side by $\varrho$. 
 Notice that the $r_1$-equation decouples in our simplified setting and that the blowup does not involve $r_1$. We describe the blowup using the directional charts
 \begin{align}
  (\bar y=1,\kappa_{11}):\,y_1 &= \varrho_1^2 y_{11},\,z_1 = \varrho_1 z_{11},\,\hat \epsilon = \varrho_1,\eqlab{chartk11}\\
  (\bar y=1,\kappa_{12}):\,y_1 &= \varrho_2^2,\,z_1 = \varrho_2 z_{12},\,\hat \epsilon = \varrho_2\hat \epsilon_2,\eqlab{chartk12}
 \end{align}
obtained by setting $\bar{\hat \epsilon}=1$ and $\bar y_1=1$, respectively. We have the following coordinate change:
\begin{align}
 y_{11}=\hat \epsilon_2^{-2},\,z_{11} = z_{12}\hat \epsilon_2^{-1},\,\varrho_2 = \varrho_1 \sqrt{y_{11}},\eqlab{kappa1112}
\end{align}
for $\hat \epsilon_2>0$. We consider each of the charts in the following.

\subsubsection*{Chart $(\bar y=1,\kappa_{11})$}
In this chart, we obtain the following equations:
  \begin{align*}
  \dot r_1&= -r_1 \varrho_1 y_{11} J(\varrho_1),\\
\dot y_{11} &=2y_{11} (K_{11}(z_{11},\varrho_1)+\varrho_1 y_{11} J(\varrho_1)),\\
\dot z_{11}&=y_{11}(K_{11}(z_{11},\varrho_1)+\varrho_1 y_{11} J(\varrho_1))+z_{11} K_{11}(z_{11},\varrho_1),\\
\dot \varrho_1 &=-\varrho_1 K_{11}(z_{11},\varrho_1),
  \end{align*}
    where
  \begin{align*}
   K_{11}(z_{11},\varrho_1) = bz_{11}(1-\pi^{-1}{\varrho_1}(1+\phi_2(\varrho_1))))+\pi^{-1}\beta(1+\phi_2(\varrho_1)),
  \end{align*}
  cf. \eqref{phipm}.
\begin{lemma}
 The unique slow manifold $S_{a,\epsilon,1}$ (see \eqref{Saeps1} for the expression of $S_{a,\epsilon,1}$ in chart $(\bar \epsilon=1,\kappa_1)$)  can be extended into $\kappa_{11}$ as a hyperbolic and attracting invariant manifold:
 \begin{align}
  z_{11} =  -\beta (b\pi)^{-1}(1+l_1(\varrho_1)))+y_{11} (\pi \beta^{-1} + y_{11} l_2(y_{11},\varrho_1)),\quad r_1\in [0,\delta],\,\varrho_1 \in [0,\mu],\,y_{11} \in [0,\xi], \eqlab{z11Manifold}
 \end{align}
 for $\mu>0$ and $\xi>0$ sufficiently small. 
The intersection of \eqref{z11Manifold} with the invariant sub-space $\{r_1=\varrho_1=0\}$: 
\begin{align}
 z_{11} = -\beta (b\pi)^{-1}+y_{11} (\pi \beta^{-1} + y_{11} l_2(y_{11},0)),\quad r_1=\varrho_1=0,\,y_{11}\in [0,\xi], \eqlab{z11Manifoldrho0}
\end{align}
is unique. $y_{11}$ increases on \eqref{z11Manifold}. 
\end{lemma}
 \begin{proof}
 The point 
\begin{align*}
\hat q_{11}^{4,b} =\{(r_1,y_{11},z_{11},\varrho_1)\vert \varrho_1 = 0,\,y_{11} = 0,\,z_{11} = -\beta (b\pi)^{-1}\},
\end{align*}
is {partially hyperbolic}, the linearization having eigenvalues $-\beta/\pi,0,0,0$ and associated eigenvectors:
\begin{align*}
 \begin{pmatrix}
 0\\
 0\\
 1\\
 0\end{pmatrix},
 \begin{pmatrix}
 0\\
 \pi^{-1} \beta\\
 1\\
 0\end{pmatrix},\begin{pmatrix}
 1\\
 0\\
 0\\
 0\end{pmatrix},\begin{pmatrix}
 0\\
 0\\
 0\\
 1\end{pmatrix}.
\end{align*}
By simple calculations we obtain \eqref{z11Manifold}. We recover $S_{a,\epsilon,1}$ at $\varrho_1=\mu$ and can therefore select the non-unique \eqref{z11Manifold} so that it coincides with this unique manifold there. Finally, the manifold \eqref{z11Manifoldrho0} is overflowing since $y_{11}'>0$ and therefore it is unique. 
 \end{proof}

 \subsubsection*{Chart $(\bar y=1,\kappa_{12})$}
  Using \eqref{kappa1112}, \eqref{z11Manifoldrho0} becomes
 \begin{align}
z_{12}  =\hat \epsilon_2\left(-\beta (b\pi)^{-1}+\hat \epsilon_2^{-2} (\pi \beta^{-1} + \hat \epsilon_2^{-2} l_2(\hat \epsilon_2^{-2},0))\right),\quad r_1=\varrho_2=0,\,\hat \epsilon_2\ge 1/\sqrt{\xi},  \eqlab{z12Manifoldrho0}
 \end{align}
Within the invariant $\{r_1=0\}$-subspace, we also obtain the following equations
\begin{align}
\dot \varrho_2 &=\frac12 \varrho_2 (K_{12}(z_{12}, \varrho_2,\hat \epsilon_2)+2\varrho_2 J(\varrho_2\hat\epsilon_2)), \eqlab{rho2z12hatEps2}\\
\dot z_{12}&=L(\varrho_2 \hat \epsilon_2) -\frac12 z_{12}K_{12}(z_{12}, \varrho_2,\hat \epsilon_2),\nonumber\\
\dot{\hat \epsilon}_2&=-\hat \epsilon_2 \left(\frac32 K_{12}(z_{12}, \varrho_2,\hat \epsilon_2)+\varrho_2 J(\varrho_2\hat \epsilon_2)\right),\nonumber
\end{align}
where
\begin{align*}
 K_{12}(z_{12}, \varrho_2,\hat \epsilon_2) =bz_{12}(1-\pi^{-1}{\varrho_2\hat\epsilon}(1+\phi_2(\varrho_2\hat \epsilon)))+\pi^{-1}\beta \hat \epsilon_2(1+\phi_2(\varrho_2\hat \epsilon_2)).
\end{align*}
\begin{lemma}
Consider the invariant subspace $\{r_1=0\}$. Then 
the point 
 \begin{align}
 \hat q_{12}^{5,b}:\,(r_1,\varrho_2,z_{12},\hat \epsilon_2) =(0,0,\sqrt{2b^{-1}},0),\eqlab{point}
 \end{align}
 is a hyperbolic equilibrium of \eqref{rho2z12hatEps2}, with eigenvalues $-3\sqrt{b/2},\,-\sqrt{b}$ and $\sqrt{b/2}$. The set $$W^s(\hat q_{12}^{5,b}) = \{r_1=0,\,\varrho_2=0,\,(z_{12},\hat \epsilon_2)\ne (-\sqrt{2b^{-1}},0),\,\hat \epsilon_2\ge 0\},$$ is the associated global $2D$-stable manifold while $$W^u(\hat q_{12}^{5,b}) = \{r_1=0,\,z_{12}=\sqrt{2b^{-1}},\hat \epsilon_2=0,\,\varrho_{2}\ge 0\},$$ is the associated global $1D$-stable manifold.  
\end{lemma}
\begin{proof}
The first part of the lemma follows easily. The fact that \eqref{point} attracts all points in $W^s$ follow from a simple phase plane analysis of the $\{r_1=\varrho_2=0\}$-subsystem:
\begin{align}
 \dot z_{12} &=1-\frac12 z_{12} (bz_{12}+\pi^{-1}\beta \hat \epsilon_2),\eqlab{z12hatEpsLayer}\\
  \dot{\hat \epsilon}_2&=-\hat \epsilon_2(bz_{12}+\pi^{-1}\beta \hat \epsilon_2).\nonumber
\end{align}
The remainder of the proof then follows from straightforward calculations. 
\end{proof}
On the unstable set $W^u(\hat q_{12}^{5,b})$, $\varrho_2$ (and hence $y_1$ by \eqref{chartk12}) increases. We can therefore return to the chart $(\bar y=1,\kappa_1)$. Here $W^u(\hat q_{12}^{5,b})$ becomes  
\begin{align}
  \widehat Q_1^{5,b} = \{(r_1,y_1,z_1,\hat \epsilon) \vert r_1=\hat \epsilon=0,\,y_1 = \varrho_2^2,\,z_1 = \varrho_2\sqrt{2b^{-1}},\,\varrho_2\ge 0\}.\eqlab{hatQ15b}
\end{align}
Since $y_1$ increases along $\widehat Q_1^{5,b}$ we then move to chart $(\bar y=1,\kappa_2)$. 

\subsubsection*{Chart $(\bar y=1,\kappa_2)$}
In this chart we obtain the following equations
\begin{align}
\dot x_2&=J(\hat \epsilon)-\frac12 x_2 K_2(x_2,z_2,\hat\epsilon),\eqlab{BarY1Kappa2Eqns}\\
 \dot r_2 &=\frac12 r_2 K_2(x_2,z_2,\hat \epsilon),\nonumber\\
 \dot z_2 &=L(\hat \epsilon) - \frac12 z_2 K_2(x_2,z_2,\hat \epsilon),\nonumber\\
 \dot{\hat \epsilon}&=-\hat \epsilon K_2(x_2,z_2,\hat \epsilon),\nonumber\\
\end{align}
where
\begin{align*}
 K_2(x_2,z_2,\hat \epsilon ) = \frac12 bz_2(1+\phi_+(\hat \epsilon))-\frac12 \beta x_2(1-\phi_+(\hat \epsilon)).
\end{align*}
Notice that $K_2(x_2,z_2,0) = bz_2$ for all $z_2\in \mathbb R$. 
We then have that
\begin{lemma}
In chart $(\bar y=1,\kappa_2)$, 
\begin{align}
 \widehat Q_2^{5,b} =\{(x_2,r_2,z_2,\hat \epsilon) \vert r_2=\hat \epsilon=0,\,x_2 \in (-\infty,\beta^{-1}c \sqrt{2b^{-1}}],\,z_2 = \sqrt{2b^{-1}}\},\eqlab{hatQ25b}
\end{align}
is contained within the stable manifold of the hyperbolic equilibrium
\begin{align*}
 \widehat u_2 &= \left(\beta^{-1}c \sqrt{2b^{-1}},0,\sqrt{2b^{-1}},0)\right).
\end{align*}
The $1D$ unstable manifold of $\widehat u_2$ is 
\begin{align*}
 \widehat U_2 =\{(x_2,r_2,z_2,\hat\epsilon)\vert x_2 = \beta^{-1} c \sqrt{2b^{-1}},\,z_2=\sqrt{2b^{-1}},\,r_2\ge 0,\hat\epsilon=0\}.
\end{align*}
\end{lemma}
\begin{proof}
 Transforming \eqref{hatQ15b} into chart $(\bar y=1,\kappa_2)$ gives
 \begin{align*}
  \{(x_2,r_2,z_2,\hat \epsilon)\vert r_2=\hat \epsilon=0,\,x_2 = -1/\varrho_2,\,,\,z_2 = \sqrt{2b^{-1}},\varrho_2>0\}.
 \end{align*}
 The set $r_2=\hat \epsilon=0,\,z_2=\sqrt{2b^{-1}}$ is invariant and on this line we obtain
 \begin{align*}
  \dot x_2 &=\beta^{-1} c-\frac12 x_2 b\sqrt{2b^{-1}} = \sqrt{b/2}\left(\beta^{-1}\sqrt{2b^{-1}} c-x_2\right) .
 \end{align*}
Here $x_2 = \beta^{-1} c\sqrt{2b^{-1}}$ is a stable node. The remainder of the proof then follows from straightforward calculations. 
\end{proof}

By blowing back down to chart $\bar y=1$ and the variables $(x,y,z,\hat \epsilon)$, we realize (see \appref{lemma1}) that $\widehat U_2$ becomes
$\widehat U$ in \eqref{Uexpr2}, as desired. ({Clearly, $\widehat u_2$ and $\widehat U_2$ coincide with $\widehat u_3$ and $\widehat U_3$ in \eqsref{widehatu3}{widehatU3}, respectively, upon the coordinate change \eqref{kappa32New}}.) Therefore, using standard hyperbolic methods, it is then possible to complete the proof of \thmref{mainThm} in this case (b) with $z_1^*<0$ and guide the image $\widehat{\mathcal L}_{3,\text{out}}(\Pi_{\text{in}})$ along $\overline{\overline{Q}}^{1,b}$, $\overline{\overline{Q}}^{2,b}$, $\overline{\overline{Q}}^{3,b}$,$\overline{\overline{Q}}^{4,b}$, $\overline{\overline{Q}}^{5,b}$ and finally $\overline{\overline{U}}$, by working in the appropriate charts of the blowup \eqref{blowup1}. We omit the simple, but lengthy details.


%

\section{Case (b) with $z_1^*=0$}\applab{z1Star0}
In the case under consideration, $\widehat Q_1^{3,b}$ in the chart $(\bar y=1,\kappa_1)$ is asymptotic to the nonhyperbolic point 
\begin{align*}
 \hat q_1^{4,b}:\,r_1=y_1 = z_1=\hat \epsilon=0,
\end{align*}
see also \eqref{qq1}. Following the analysis for the case $z_1^*<0$, it is then, by working with the blowup \eqref{finalblowup} and the charts \eqsref{chartk11}{chartk12}, respectively, possible to connect $\widehat Q_1^{3,b}$, to $\widehat Q_2^{5,b}$, see \eqref{hatQ25b}, in chart $(\bar y=1,\kappa_2)$. The proof of \thmref{mainThm} can then be completed in this case too. In comparison with $z_1^*<0$, the singular orbit for $z_1^*=0$ is obtained by simply removing the slow segment $\widehat Q_1^{4,b}$ from the case $z_1^*<0$. 

\section{Case (b) with $z_1^*>0$}\applab{z1StarNeg}
In the case under consideration, we have $\hat y'>0$ in chart $(\bar \epsilon = 1,\kappa_1)$, see \eqref{BarEps1Kappa1Eqns} and \figref{Q4}. \eqref{Q13b} therefore becomes $\widehat Q_1^{3,b}:\,r_1=\epsilon_1= 0,\,z_1 = z_1^*,\,\hat y\in (-\infty,\infty)$. Hence we move to chart $(\bar y=1,\kappa_1)$. 
\subsubsection*{Chart $(\bar y=1,\kappa_1)$}
In this chart we obtain the following equations
\begin{align}
 \dot r_1 &=-r_1 y_1 J(\hat \epsilon),\eqlab{BarY1Kappa1Eqns}\\
 \dot y_1 &=y_1 (K_1(z_1,\hat \epsilon) +2y_1 J(\hat \epsilon)),\nonumber\\
 \dot z_1 &=y_1(L(\hat \epsilon)+z_1J(\hat \epsilon)),\nonumber\\
 \dot{\hat \epsilon}&=-\hat \epsilon K_1(z_1,\hat \epsilon),\nonumber
\end{align}
where
\begin{align*}
 K_1(z_1,\hat \epsilon) = \frac{1}{2}bz_1 (1+\phi_+(\hat \epsilon))+\frac12 \beta (1-\phi_+(\hat \epsilon)).
\end{align*}
In particular,
\begin{align*}
 K_1(z_1,0)  = b z_1.
\end{align*}
using that $\phi_+(0)=1$. Similar to case (a), in particular \lemmaref{caseAFinalLemma}, we obtain the following:
\begin{lemma}
The set 
\begin{align*}
 \widehat M_1^+:\,r_1\ge 0,\,y_1 = 0,\,z_1\in  \mathbb R_+,\,\hat \epsilon = 0,
\end{align*}
is a set of critical points of \eqref{BarY1Kappa3Eqns} of saddle-type: The linearization about any point $(r_1,y_1,z_1,\hat \epsilon)\in \widehat M_1^+$ has only two non-zero eigenvalues
\begin{align*}
 \pm b z_1,
\end{align*}
with corresponding eigenvectors:
\begin{align}
 \begin{pmatrix}
  -\beta^{-1} c r_1\\
  b z_1\\
  1+\beta^{-1} cz_1\\
  0
 \end{pmatrix},\,\begin{pmatrix}
  0\\
  0\\
  0\\
  1
 \end{pmatrix},\eqlab{tangentq13b}
\end{align}
respectively.
Let 
\begin{align*}
 \hat q_1^{3,b}=(0,0,z_1^*,0)\in \widehat M_1^+\cap \{r_1=0\}.
\end{align*}
Then $\widehat Q_1^{3,b}$ is the stable manifold of $\hat q_1^{3,b}$. On the other hand, the unstable manifold of $\hat q_1^{3,b}$ is 
\begin{align}
 \widehat Q_1^{4,b} &= \bigg\{(r_1,y_1,z_1,\hat \epsilon)\vert r_1=\hat \epsilon=0,\,z_1 = c^{-1} \beta\left( -1+e^{\beta^{-1} cs} \left(1 +\beta^{-1} cz_1^*\right)  \right) ,\nonumber\\
 &y_1 =  \frac12 c^{-2} b\beta^2 \left( \left(1+2\beta^{-1} cz_1^* \right) e^{2\beta^{-1} cs}+1  - 2\left( 1+ \beta^{-1} cz_1^* \right) e^{\beta^{-1} cs}\right),\nonumber\\
 &\textnormal{for}\,s\in \overline{\mathbb R}_+\bigg\},
\end{align}
tangent to $(0,
  bz_1^*,
  1+\beta^{-1} cz_1^*,  0)^T$, see \eqref{tangentq13b} with $r_1=0$, at $\hat q_1^{3,b}$. 

\end{lemma}
Working in chart $(\bar y=1,\kappa_2)$ we can carry $\overline{\widehat Q}^{4,b}$ into the chart $(\bar y=1,\kappa_3)$. Here it is contained within the stable manifold of the hyperbolic equilibrium $\widehat u_3$ \eqref{widehatu3}. From here we obtain $\widehat U_3$ as the $1D$ unstable manifold.  

Now using standard hyperbolic methods, it is then possible to complete the proof of \thmref{mainThm} in this case (b) with $z_1^*>0$ and guide the image $\widehat{\mathcal L}_{3,\text{out}}$ along $\overline{\overline{Q}}^{1,b}$, $\overline{\overline{Q}}^{2,b}$,  $\overline{\overline{Q}}^{3,b}$,$\overline{\overline{Q}}^{4,b}$ and finally $\overline{\overline{U}}$, by working in the appropriate charts. We omit the simple, but lengthy details.


%
%
 \end{document}